\documentclass[A4paper,10pt,draft]{amsart}
\usepackage{amssymb,amsmath,amsfonts,amsthm}
\usepackage[utf8]{inputenc}
\usepackage[english]{babel}
\usepackage[dvipsnames]{xcolor}
\usepackage{mathtools}
\mathtoolsset{showonlyrefs}

\newtheorem{theorem}{Theorem}[section]
\newtheorem{corollary}{Corollary}[section]
\newtheorem{lemma}{Lemma}[section]
\newtheorem{proposition}{Proposition}[section]
\newtheorem{definition}{Definition}[section]
\newtheorem{remark}{Remark}

\newcommand\supp{\mathop{\rm supp}\nolimits}

\def\supp{\text{supp}}

\def\<{\langle}
\def\>{\rangle}
\begin{document}
\title[Microlocal Regularity of PDE in Fourier Lebesgue spaces]{Microlocal Regularity of  nonlinear PDE
 in quasi-homogeneous Fourier Lebesgue spaces}
\author{Gianluca Garello, Alessandro  Morando}
\maketitle
\begin{abstract}
We study the continuity in weighted Fourier Lebesgue spac\-es for a class of pseudodifferential operators,  whose symbol has finite Fourier Lebesgue regularity with respect to $x$ and satisfies a quasi-homogeneous decay of derivatives with respect to the $\xi$ variable. Applications to Fourier Lebesgue microlocal regularity of linear and nonlinear partial differential equations are given.
\end{abstract}
\textbf{Keywords.} Microlocal Analysis, Pseudodifferential Operators , Fourier Lebesgue Spaces.\\
\textbf{MSC2020.} 35J60, 35J62, 35S05.

\maketitle
\tableofcontents
\section{Introduction}\label{int}
In \cite{GM-18} we study inhomogeneuos local and microlocal propagation of singularities of generalized Fourier Lebesgue type for a class of semilinear partial differential equations (shortly written PDE): other results on the topic may be found in \cite{G1}, \cite{PTT2}, \cite{PTT1}.   The present paper is a natural  continuation of the same subject, where Fourier Lebesgue microlocal regularity for nonlinear  PDE is considered. To introduce the problem, let us first consider the following general equation
\begin{equation}\label{nl_eq}
F(x,\partial^\alpha u)_{\alpha\in\mathcal I}=0\,,
\end{equation}
where $\mathcal I$ is a finite set of multi-indices $\alpha\in\mathbb Z^n_+$, and $F(x,\zeta)\in C^\infty(\mathbb R^n\times\mathbb C^N)$ is a nonlinear function of $x\in\mathbb R^n$ and $\zeta=(\zeta^\alpha)_{\alpha\in\mathcal I}\in \mathbb C^N$. In order to study the regularity of solutions of \eqref{nl_eq}, we can move the investigation  to the {\em linearized equations} obtained from differentiation with respect to $x_j$ 
\[
\sum\limits_{\alpha\in\mathcal I}\frac{\partial F}{\partial\zeta^\alpha}F(x,\partial^\beta u)_{\beta\in\mathcal I}\partial^\alpha\partial_{x_j}u=-\frac{\partial F}{\partial x_j}(x,\partial^\beta u)_{\beta\in\mathcal I}\,,\quad j=1,\dots,n.
\] 
Notice that the regularity of the coefficients $a_\alpha(x):=\frac{\partial F}{\partial\zeta^\alpha}(x,\partial^\beta u)_{\beta\in\mathcal I}$ depends on some a priori  smoothness of the solution $u=u(x)$ and the nonlinear function $F(x,\zeta)$. This naturally leads to the study of linear PDE whose coefficients have only limited regularity, in our case they will belong to some generalized Fourier Lebesgue space.\\
Results about local and microlocal regularity for semilinear and non linear PDE in Sobolev and Besov framework may be found in \cite{GM-05}, \cite{GM-15}.\\
Failing of any symbolic calculus for pseudodifferential operators with symbols $a(x,\xi)$ with limited smoothness in $x$, one needs to refer to paradifferential calculus of Bony-Meyer \cite{BO}, \cite{ME1}  or decompose the non smooth symbols according to the general technique introduced by M.Taylor in \cite[Proposition 1.3 B]{MT-NLPDE}; here we will follow this second approach. By the way both methods rely on the dyadic decomposition of distributions, based on a partition of the  frequency space $\mathbb R^n_\xi$ by means of suitable family of crowns, see again Bony \cite{BO}.\\
In this paper we consider a natural framework where such a decomposition method can be adapted, namely we deal with symbols which exhibit a behavior at infinity of quasi-homogeneous type,  called in the following \textit{quasi-homogeneous symbols}.  When the behavior of symbols at infinity does not satisfy any kind of homogeneity, the dyadic decomposition method seems to fail. \\
In general  the technique of Taylor quoted above  splits the symbols $a(x,\xi)$ with limited smoothness in $x$ into
\begin{equation}\label{Tsplit}
a(x,\xi)= a^\# (x,\xi)+ a^\natural (x,\xi).
\end{equation}
While $a^\natural(x,\xi)$ keeps the same regularity of $a(x,\xi)$, with a slightly improved decay at infinitive, $a^\#(x, \xi)$ is a smooth symbols of type $(1,\delta)$, with $\delta>0$.
\\
From Sugimoto-Tomita \cite{SUTO08}, it is known that, in general, pseudodifferential operators with symbol in $S^0_{1,\delta}$,  are not bounded on modulation spaces $M^{p,q}$  as long as $0<\delta\le 1$ and $q\neq	 2$. Since the Fourier-Lebesgue and modulation spaces are {\em locally} the same, see \cite{Gr-book} for details, it follows from \cite{SUTO08} that the operators $a^\#(x, D)$ are generally unbounded on Fourier Lebesgue spaces, when the exponent is different of 2. We  are able to avoid this difficulty by carefully analyzing the behavior of the term $a^\#(x,\xi)$ as described in the next Sections \ref{smooth_symb_sct}, \ref{propation_sing_FL_sct_0}.\\
In the first Section all the main results of the paper are presented. The proof are postponed in the subsequent sections. Precisely in Section \ref{FL-sp_sct} a generalization to the quasi-homogeneous framework of the characterization of Fourier Lebesgue spaces  , by means of dyadic decomposition is detailed. Section \ref{PROOF1} is completely devoted to the proof of Thoerem \ref{FL_thm}. The symbolic calculus of pseudodifferential operators with smooth symbols is developed in Section \ref{smooth_symb_sct}, while Section \ref{propation_sing_FL_sct_0} is devoted to the generalization of the Taylor splitting technique. In the last Section we study the microlocal behavior of pseudodifferential operators with smooth symbols, jointly with their applications to nonlinear PDE.

\section{Main results} \label{int_sct}
\subsection{{Notation}}\label{not}
In this preliminary section we give the main definitions and notation most frequently used in the paper. $\mathbb R_+$ and $\mathbb N$ are respectively the sets of strictly positive real and integer numbers. For $M=(\mu_1,\dots,\mu_n)\in \mathbb R_+^n$, $\xi\in \mathbb R^n$ we define  :
\begin{equation}\label{M-h-w}
\langle\xi\rangle_M:=\left(1+\vert\xi\vert_M^2\right)^{1/2}\quad (M-\mbox{\em weight})\,,
\end{equation}
where 
\begin{equation}\label{M-h-n}
\vert\xi\vert_M^2:=\sum\limits_{j=1}^n\vert\xi_j\vert^{2\mu_j}\quad (M-\mbox{\em norm})\,.
\end{equation}
For $t>0$ and $\alpha\in\mathbb Z^n_+$, we set
\begin{equation}\label{M-notations}
\begin{split}
& t^{1/M}\xi:=(t^{1/\mu_1}\xi_1,\dots,t^{1/\mu_n}\xi_n)\,; \\
&\langle\alpha,1/M\rangle:=\sum\limits_{j=1}^n\alpha_j/\mu_j\,;\\
& \mu_\ast:=\min\limits_{1\le j\le n}\mu_j\,,\qquad\mu^\ast:=\max\limits_{1\le j\le n}\mu_j\,.
\end{split}
\end{equation}
We call $\mu_\ast$ and $\mu^\ast$ respectively the {\em minimum} and the {\em maximum order} of $\langle \xi\rangle_M$; furthermore, we will refer to $\langle\alpha,1/M\rangle$ as the $M-${\em order} of $\alpha$. In the case of $M=(1,\dots,1)$, \eqref{M-h-n} reduces to the {\em Euclidean} norm $\vert\xi\vert$, and the $M-$ weight \eqref{M-h-w} reduces to the standard {\em homogeneous weight} $\langle\xi\rangle=(1+\vert\xi\vert^2)^{1/2}$.

The following properties can be easily proved, see \cite{GM-06} and the references therein.

\begin{lemma}\label{lemma_M-h-w}
For any $M\in\mathbb R^n_+$, there exists a suitable positive constant $C$ such that the following hold for any $\xi\in \mathbb R^n$:
\begin{eqnarray}
\label{PG}
\frac1{C}\langle\xi\rangle^{\mu_\ast}\le\langle\xi\rangle_M\le C\langle\xi\rangle^{\mu^\ast}\,,& \text{Polynomial growth};\\
\label{sub-add}
\vert\xi+\eta\vert_M\le C\left\{\vert\xi\vert_M+\vert\eta\vert_M\right\}\,, & M- \text{sub-additivity};\\
\label{M-hom}
\vert t^{1/M}\xi\vert_M=t\vert\xi\vert_M\,,\quad t>0\,, & M-\text{homogeneity} .
\end{eqnarray}
\end{lemma}

For  $\phi$  in the space of \textit{rapidly decreasing functions} $ \mathcal S(\mathbb R^n)$, the Fourier transform is defined by $\hat\phi(\xi)=\mathcal F \phi(\xi)=\int e^{- ix\cdot\i}\phi(x)\, dx$, $x\cdot\xi=\sum_{j=1}^n x_j\xi_j$; $\hat u =\mathcal F u$, defined by $\langle \hat u, \phi\rangle=\langle u, \hat \phi\rangle$, is its analogous in the dual space of \textit{tempered distributions} $\mathcal S'(\mathbb R^n)$  
\subsection{Pseudodifferential operators with symbols in Fourier Lebesgue spa\-c\-es}\label{FLcont_sct}
\begin{definition}\label{def_FL}
For $s\in\mathbb R$ and $p\in[1,+\infty]$ we denote by $\mathcal FL^p_{s,M}$ the class of all  $u\in\mathcal S^\prime(\mathbb R^n)$ such that $\hat u$ is a measurable function in $\mathbb R^n$ and
$\langle\cdot\rangle^s_M\hat u\in L^p(\mathbb R^n)$.
$\mathcal FL^p_{s,M}$, endowed with the natural norm
\begin{equation}\label{FL-norm}
\Vert u\Vert_{\mathcal FL^p_{s,M}}:=\Vert\langle\cdot\rangle^s_M\hat u\Vert_{L^p}\,,
\end{equation}
is a Banach space, said \textit{$M-$homogeneous Fourier Lebesgue space} of order $s$ and exponent $p$. 
\end{definition}
Notice that for $p=2$, Plancherel's Theorem yields that $\mathcal FL^2_{s,M}$ reduces to the $M-$homogeneous {\em Sobolev space} of order $s$, see \cite{GM-06} for details; in this case $\mathcal FL^2_{s,M}$ inherits from  $L^2(\mathbb R^n)$ the structure of Hilbert space, with  inner product
$
(u,v)_{\mathcal FL^2_{s,M}}:=(\langle\cdot\rangle^s_M\hat u,\langle\cdot\rangle^s_M\hat v)_{L^2}\,,
$
where $(f,g)_{L^2}$ is the standard inner product in $L^2(\mathbb R^n)$.
\newline
In the case of $M=(1,\dots,1)$, $\mathcal FL^p_{s,M}$ reduces to the {\em homogeneous} Fourier Lebesgue space $\mathcal FL^p_s$ and, in particular, we set $\mathcal FL^p:=\mathcal FL^p_0$.

\smallskip

The {\em pseudodifferential operator} $a(x,D)$ with general symbol  $a(x,\xi)\in\mathcal S^\prime(\mathbb R^{2n})$  and standard Kohn--Nirenberg quantization is the bounded linear map
\begin{equation}\label{psdo_formula}
\begin{array}{lcll}
a(x,D):&\mathcal S(\mathbb R^n)&\rightarrow&\mathcal S^\prime(\mathbb R^n)\\
&u&\rightarrow & a(x,D)u(x):=(2\pi)^{-n}\int e^{ix\cdot\xi}a(x,\xi)\widehat u(\xi)d\xi\,, 
\end{array}
\end{equation}
where the integral above must be understood in the distributional sense.
\newline
We introduce here some classes of symbols $a(x,\xi)$, of $M-$homogeneous type, with limited Fourier Lebesgue smoothness with respect to the space variable $x$.
\begin{definition}\label{def_bdd_smooth_symb}
For $m, r\in\mathbb R$, $\delta\in[0,1]$, $p\in[1,+\infty]$ and $N$ a strictly positive integer number, we denote by $\mathcal FL^{p}_{r,M}S^m_{M,\delta}(N)$ the set of $a(x,\xi)\in\mathcal S^\prime(\mathbb R^{2n})$ such that for all $\alpha\in\mathbb Z^n_+$ with $\vert\alpha\vert\le N$, the map $\xi\mapsto\partial^\alpha_\xi a(\cdot,\xi)$ is measurable  in $\mathbb R^n$ with values in $\mathcal FL^p_{r,M}\cap\mathcal FL^1$ and satisfies for any $\xi\in \mathbb R^n$ the following estimates
\begin{equation}\label{FL_est}
\Vert\partial^\alpha_\xi a(\cdot,\xi)\Vert_{\mathcal FL^1}\le C\langle\xi\rangle_M^{m-\langle\alpha, 1/M\rangle}\,,
\end{equation}
\begin{equation}\label{X_est}
\Vert\partial^\alpha_\xi a(\cdot,\xi)\Vert_{\mathcal FL^p_{r,M}}\le C\langle\xi\rangle_M^{m-\langle\alpha, 1/M\rangle+\delta\left(r-\frac{n}{\mu_\ast q}\right)}\,,
\end{equation}
where $C$ is a suitable positive constant and $q$ is the conjugate exponent of $p$.
\end{definition}
When $\delta=0$, we will write for shortness $\mathcal FL^p_{r,M}S^m_{M}(N)$.
\newline
The first result concerns with the Fourier Lebesgue boundedness of pseudodifferential operators with symbol in $\mathcal FL^p_{s,M}S^m_{M,\delta}(N)$.
\begin{theorem}\label{FL_thm}
Consider $p\in[1,+\infty]$, $q$ its  conjugate exponent, $r>\frac{n}{\mu_\ast q}$, $\delta\in[0,1]$, $m\in\mathbb R$, $N>n+1$ and $a(x,\xi)\in\mathcal FL^p_{r,M}S^m_{M,\delta}(N)$. Then for all $s\in\left](\delta-1)(r-\frac{n}{\mu_\ast q}),r\right[$ the pseudodifferential operator $a(x,D)$ extends to a bounded operator
\begin{equation}
a(x,D):\mathcal FL^p_{s+m,M}\rightarrow\mathcal FL^p_{s,M}\,.
\end{equation}
If $\delta<1$ then the above continuity property holds true also for $s=r$.
\end{theorem}
The proof is given in the next Sect.\ref{PROOF1}.
\begin{remark}\label{obs_thm}
{\rm Observe that in the case of $\delta=0$, the above result was already proved in \cite[Proposition 6]{GM-18}, where a much more general setting than the framework of $M-$homoge\-neo\-us symbols  was considered and very weak growth conditions on symbols with respect to $\xi$ were assumed.}
\end{remark}

\subsection{$M$-homogeneous smooth symbols}\label{hsy}
Smooth symbols satisfying $M-$quasi-homogenous decay of derivatives at infinity are useful for the study of microlocal propagation of singularities for pseudodifferential operators with non smooth symbols and nonlinear PDE.
\begin{definition}\label{sm-sym_dfn}
For $m\in\mathbb R$ and $\delta\in[0,1]$, $S^m_{M,\delta}$ is the class of the functions $a(x,\xi)\in C^\infty(\mathbb R^{2n})$ such that for all $\alpha,\beta\in\mathbb Z^n_+$, and $x,\xi\in\mathbb R^n$
\begin{equation}\label{sm-sym_est}
\vert\partial^\alpha_\xi\partial^\beta_x a(x,\xi)\vert\le C_{\alpha,\beta}\langle\xi\rangle_{M}^{m-\langle\alpha,1/M\rangle+\delta\langle\beta,1/M\rangle}\,,
\end{equation}
for a suitable constant $C_{\alpha,\beta}$.
\end{definition}
In the following, we set for shortness $S_{M}:=S_{M,0}$. Notice that for any $\delta\in[0,1]$ we have
$
\bigcap\limits_{m\in\mathbb R}S^m_{M,\delta}\equiv S^{-\infty}\,,
$
where $S^{-\infty}$ denotes the set of  the functions $a(x,\xi)\in C^\infty(\mathbb R^{2n})$ such that for all $\mu>0$ and $\alpha,\beta\in\mathbb Z^n_+$
\begin{equation}\label{sm-sym_inf_est}
\vert\partial^\alpha_\xi\partial^\beta_x a(x,\xi)\vert\le C_{\mu,\alpha,\beta}\langle\xi\rangle^{-\mu}\,,\quad x,\xi\in\mathbb R^n\,,
\end{equation}
for a suitable positive constant $C_{\mu,\alpha,\beta}$.
\smallskip

We recall that a pseudodifferential operator $a(x,D)$ with symbol $a(x,\xi)\in S^{-\infty}$ is {\em smoothing}, namely it extends as a linear bounded operator from
$\mathcal S^\prime(\mathbb R^n)$ ($\mathcal E^\prime(\mathbb R^n)$) to     
$\mathcal P(\mathbb R^n)$ ($\mathcal S(\mathbb R^n)$),
where $\mathcal P(\mathbb R^n)$ and $\mathcal E^\prime(\mathbb R^n)$ are respectively the space of smooth functions polynomially bounded together with their derivatives and the space of compactly supported distributions.

\smallskip

As long as $0\le\delta<\mu_\ast/\mu^\ast$, for the $M-$homogeneous classes $S^m_{M,\delta}$ a complete symbolic calculus is available, see e.g. Garello - Morando \cite{GM-08, GM-09} for details.

\smallskip
Pseudodifferential operators with symbol in $S^0_M$ are known to be locally bounded on Fourier Lebesgue spaces $\mathcal FL^p_{s,M}$ for all $s\in\mathbb R$ and $1\le p\le +\infty$, see e.g. Tachizawa \cite{TA-94} and Rochberg-Tachizawa \cite{ROTA-98}. For continuity of Fourier Integral Operators on Fourier Lebesgue spaces see \cite{C_N_R}.  On the other hand, by easily adapting the arguments used in the homogeneous case $M=(1,\dots,1)$ by Sugimoto-Tomita \cite{SUTO08}, it is known that pseudodifferential operators with symbol in $S^0_{M,\delta}$ are not locally bounded on $\mathcal FL^p_{s,M}$, as long as $0<\delta\le 1$ and $p\neq	 2$\footnote{Actually the boundedness result given in \cite{TA-94, ROTA-98, SUTO08} are stated in the framework of {\em modulation spaces}. Nevertheless, it also tells about the local behavior on Fourier Lebesgue spaces, as Fourier-Lebesgue and modulation spaces are {\em locally} the same, see \cite{Gr-book} for details.}.
\newline
For this reason we introduce suitable subclasses of $M-$homogeneous symbols in $S^m_{M,\delta}$, $\delta\in[0,1]$, whose related pseudodifferential operators are (locally) well-behav\-ed on weighted Fourier Lebesgue spaces. These symbols will naturally come into play in  the splitting method presented in Sect. \ref{propation_sing_FL_sct_0} and used in Sect. \ref{propation_sing_FL_sct} to derive local and microlocal Fourier Lebesgue regularity of linear PDE with non smooth coefficients.\\
In view of such applications, it is useful that the vector $M=(\mu_1,\dots,\mu_n)$ has strictly positive integer components. Let us assume it for the rest of Sect. \ref{int_sct}, unless otherwise explicitly stated.\\
In the following $t_+:= \max\{t,0\}$, $[t]:=\max\{n\in \mathbb Z ; n\leq t\}$ are respectively the {\em positive part} and the {\em integer part} of $t\in \mathbb R$.
\begin{definition}\label{sm-sym_k_dfn}
For $m\in\mathbb R$, $\delta\in[0,1]$ and $\kappa>0$ we denote by $S^m_{M,\delta, \kappa}$ the class of all functions $a(x,\xi)\in C^\infty(\mathbb R^{2n})$ such that for $\alpha,\beta\in\mathbb Z^n_+$ and $x,\xi\in\mathbb R^{n}$
\begin{eqnarray}
\vert \partial^\alpha_\xi \partial^\beta_x a(x,\xi)\vert\le C_{\alpha,\beta}\langle\xi\rangle_M^{m-\langle\alpha,1/M\rangle+\delta\left(\langle\beta,1/M\rangle-\kappa\right)_+}\,,\quad\mbox{if}\,\,\,\langle\beta,1/M\rangle\neq\kappa\,,\label{sm-sym_k_est_1}\\
\vert \partial^\alpha_\xi \partial^\beta_x a(x,\xi)\vert\le C_{\alpha,\beta}\langle\xi\rangle_M^{m-\langle\alpha,1/M\rangle}\log\left(1+\langle\xi\rangle_M^{\delta}\right)\,,\quad\mbox{if}\,\,\,\langle\beta,1/M\rangle=\kappa\label{sm-sym_k_est_2}
\end{eqnarray}
holds with some positive constant $C_{\alpha,\beta}$. 
\end{definition}
\begin{remark}\label{rmk_sym_k}
{\rm It is easy to see that for any $\kappa>0$, the symbol class $S^m_{M,\delta,\kappa}$ defined above is included in $S^m_{M,\delta}$ for all $m\in\mathbb R$ and $\delta\in[0,1]$ (notice in particular that $S^m_{M,0,\kappa}\equiv S^m_{M,0}\equiv S^m_M$ whatever is $\kappa>0$). Compared to Definition \ref{sm-sym_dfn}, symbols in $S_{M,\delta,\kappa}$ display a better behavior face to the growth at infinity of derivatives; 
the loss of decay $\delta\langle \beta, 1/M\rangle$, connected  to the $x$ derivatives when $\delta>0$, does not  occur when the $M$- order of $\beta$ is less than $\kappa$; for the subsequent derivatives the loss is decreased of the fixed amount $\kappa$.
\newline

Since for $M=(\mu_1,\dots, \mu_n)$, with positive integer components, the $M-$ order of any multi-index $\alpha\in\mathbb Z^n_+$ is a rational number, we notice that symbol derivatives never exhibit the "logarithmic growth" \eqref{sm-sym_k_est_2} for an irrational $\kappa>0$.}
\end{remark}
\begin{theorem}\label{FL_sym_cor}
Assume that
\begin{equation}\label{kappa_ass}
\kappa>[n/\mu_\ast]+1
\end{equation}
Then for all $p\in[1,+\infty]$ a pseudodifferential operator with symbol $a(x,\xi)\in S^m_{M,\delta,\kappa}$, satisfying the \textit{localization} condition
\begin{equation}\label{localization}
{\rm supp}\,a(\cdot,\xi)\subseteq \mathcal K\,,\quad\forall\,\xi\in\mathbb R^n\, ,
\end{equation}
for a suitable compact set $\mathcal K\subset\mathbb R^n$, extends as a linear bounded operator
\begin{eqnarray}
& &a(x,D):\mathcal FL^p_{s+m,M}\rightarrow\mathcal FL^p_{s,M}\,,\quad\forall\,s\in\mathbb R\,,\quad\mbox{if}\,\,\,0\le\delta<1\,, \label{FL_bdd}\\
& &a(x,D):\mathcal FL^p_{s+m,M}\rightarrow\mathcal FL^p_{s,M}\,,\quad\forall\,s>0\,,\quad\mbox{if}\,\,\,\delta=1\,.\label{FL_bdd1}
\end{eqnarray}
\end{theorem}
The proof of Theorem \ref{FL_sym_cor} is postponed to Sect. \ref{cont_smooth_symb_sct}.
\smallskip

Taking $\delta=0$, we directly obtain the boundedness property \eqref{FL_bdd}, for any pseudodifferential operator with symbol in $S^m_M$.
\newline

The following result concerning the Fourier multipliers readily follows from H\"older's inequaltity.
\begin{proposition}\label{Fm_prop}
Let a tempered distribution $a(\xi)\in\mathcal S^\prime(\mathbb R^n)$ satisfy 
\[
\langle\xi\rangle_M^{-m}a(\xi)\in L^\infty(\mathbb R^n)
\]
for $m\in\mathbb R$. Then the Fourier multiplier $a(D)$ extends as a linear bounded operator from $\mathcal FL^p_{s+m,M}$ to $\mathcal FL^p_{s,M}$, for all $p\in[1,+\infty]$ and $s\in\mathbb R$.
\end{proposition}
\subsection{Microlocal propagation of  Fourier Leb\-es\-gue singularities}\label{pfl}
Consider a vector $M=(\mu_1,\dots,\mu_n)\in\mathbb{N}^n$ and set $T^{\circ}\mathbb{R}^n:=\mathbb{R}^n\times(\mathbb{R}^n\setminus\{0\})$.

We say that a set $\Gamma_M\subset\mathbb{R}^n\setminus\{0\}$ is $M-$\textit{conic}, if $t^{1/M}\xi\in\Gamma_M$ for any $\xi\in \Gamma_M$ and $t>0$.
\begin{definition}\label{pfl_def_WF}
For $s \in\mathbb R$, $p\in [1,+\infty]$,  $u\in \mathcal S'(\mathbb R^n)$, we say that  $(x_0, \xi^0)\in T^{\circ}\mathbb R^n$ does not belong to the M-conic \textit{wave front set} $WF_{\mathcal F L^p_{s, M}}u$, if there exist $\phi \in C^\infty_0(\mathbb R^n)$, $\phi(x_0)\neq 0$, and a symbol $\psi(\xi)\in S^0_M$, satisfying $\psi(\xi)\equiv 1$ on $\Gamma_M\cap \{\vert \xi\vert_M>\varepsilon_0\}$, for suitable M-conic neighborhood $\Gamma_M\subset \mathbb R^n\setminus\{ 0\}$ of $\xi^0$ and  $0<\varepsilon _0< \vert \xi^0\vert _M$, such that
\begin{equation}\label{WFL1}
\psi(D)(\phi u)\in \mathcal FL^p_{s,M}.
\end{equation}
We say in this case that $u$ is $FL^p_{s,M}-$ microlocally regular at the point $(x_0,\xi^0)$ and we write $u\in \mathcal F L^p_{s, M, \textup{mcl}}(x_0,\xi^0)$.
\newline
We say that $u\in\mathcal S^\prime(\mathbb R^n)$ belongs to $\mathcal FL^p_{s,M,{\rm loc}}(x_0)$ if there exists a smooth function $\phi\in C^\infty_0(\mathbb R^n)$ satisfying $\phi(x_0)\neq 0$ such that
\[
\phi u\in\mathcal FL^p_{s,M}\,.
\]
\end{definition}
\begin{remark}\label{rmk_mcl_FL}
{\rm In view of Definition \ref{def_FL}, it is easy to verify that $u\in\mathcal FL^p_{s,M, {\rm mcl}}(x_0,\xi^0)$ if and only if
\begin{equation}\label{FL_mcl}
\chi_{\varepsilon_0,\Gamma_M}\,\langle\cdot\rangle_M^r\widehat{\phi u}\in L^p(\mathbb R^n)\,,
\end{equation}
where $\phi$ and $\Gamma_M$ are considered as in Definition \ref{pfl_def_WF} and $\chi_{\varepsilon_0,\Gamma_M}$ is the characteristic function of $\Gamma_M\cap\{\vert\xi\vert_M>\varepsilon_0\}$. 
}
\end{remark}
\begin{definition}\label{Melliptic} We say that a symbol $a(x,\xi)\in S^m_{M,\delta}$ is microlocally $M-$ellip\-tic at $(x_0,\xi^0)\in T^{\circ}\mathbb{R}^n$ if there exist an open neighborhood $U$ of $x_0$ and an $M-$conic open neighborhood $\Gamma_M$ of $\xi^0$ such that for $c_0>0$, $\rho_0>0$:
\begin{equation}\label{Mellipticity}
|a(x,\xi)|\ge c_0\langle\xi\rangle_M^m\,,\quad (x,\xi)\in U\times\Gamma_M\,,\quad |\xi|_M>\rho_0\,.
\end{equation}
Moreover the characteristic set of $a(x,\xi)$ is ${\rm Char}(a)\subset T^{\circ}\mathbb{R}^n$ defined by
\begin{equation}
(x_0,\xi^0)\in T^{\circ}\mathbb{R}^n\setminus{\rm Char}(a)\,\,\Leftrightarrow\,\,\,\,a\,\, \text {is\,\,microlocally\,\,M-elliptic\,\,at}\,\,(x_0,\xi^0)\,.
\end{equation}
\end{definition}
\begin{theorem}\label{FL_cor}
for $0\le\delta<\mu_\ast/\mu^\ast$, $\kappa>[n/\mu_\ast]+1$, $m\in\mathbb R$, $a(x,\xi)\in S^m_{M,\delta,\kappa}$ and $u\in\mathcal S^\prime(\mathbb R^n)$, the following inclusions
\begin{equation*}
WF_{\mathcal FL^p_{s,M}}(a(x,D)u)\subset WF_{\mathcal \mathcal FL^p_{s+m,M}}(u)\subset WF_{\mathcal FL^p_{s,M}}(a(x,D)u)\cup{\rm Char}(a)
\end{equation*}
hold true for every $s\in\mathbb R$ and $p\in[1,+\infty]$.
\end{theorem}
The proof of Theorem \ref{FL_cor} will be given in Sect. \ref{Xs_mcl_res_sct}.
\subsection{Linear PDE with non smooth coefficients}\label{linear_pde_sct}
In this section we discuss the  $M-$homogeneous Fourier Lebesgue microlocal regularity for linear PDE of the type
\begin{equation}\label{APPL:1}
a(x,D)u:=\sum_{\langle\alpha,1/M\rangle\leq 1}c_\alpha(x) D^\alpha u=f(x)\,,
\end{equation}
where $D^{\alpha}:=(-i)^{|\alpha|}\partial^{\alpha}$, while the coefficients $c_{\alpha}$, as well as the source $f$ in the right-hand side, are assumed to have suitable {\em local} $M-$homogeneous Fourier Lebesgue regularity\footnote{Without loss of generality, we assume that derivatives involved in the expression of the linear partial differential operator $a(x,D)$ in the left-hand side of \eqref{APPL:1} have $M-$ order not larger than one, since for any finite set $\mathcal A$ of multi-indices $\alpha\in\mathbb Z^n_+$ it is always possible selecting a vector $M=(\mu_1,\dots,\mu_n)\in\mathbb R_+^n$ so that $\langle\alpha,1/M\rangle\le 1$ for all $\alpha\in\mathcal A$.}.
\medskip
Let $(x_0,\xi^0)\in T^\circ\mathbb R^n$, $p\in[1,+\infty]$ and $r>\frac{n}{\mu_\ast q}+\left[\frac{n}{\mu_\ast}\right]+1$ be given. We make on $a(x,D)$ in \eqref{APPL:1} the following assumptions:
\begin{itemize}
\item[(i)] $c_\alpha\in\mathcal FL^p_{r,M,{\rm loc}}(x_0)$ for $\langle\alpha,1/M\rangle\le 1$;
\item[(ii)] $a_M(x_0,\xi^0)\neq 0$, where $a_M(x,\xi):=\!\!\!\!\!\sum\limits_{\langle\alpha,1/M\rangle=1}\!\!\!\!\!c_\alpha(x)\xi^\alpha$ is the  {\em $M-$principal symbol} of $a(x,D)$.
\end{itemize}
Arguing on continuity and $M-$homogeneity in $\xi$ of $a_M(x,\xi)$, it is easy to prove  that, for suitable open neighborhood $U\subset\mathbb R^n$ of $x_0$ and open $M-$conic neighborhood $\Gamma_M\subset\mathbb R^n\setminus\{0\}$ of $\xi^0$
\begin{equation}\label{principal}
a_M(x,\xi)\neq 0, \quad \text{for}\, (x,\xi)\in U\times \Gamma_M\,,
\end{equation}

\begin{theorem}\label{cor_APPLICATION}
Consider  $(x_0,\xi^0)\in T^\circ\mathbb R^n$, $p\in[1,+\infty]$ and $q$ its  conjugate exponent, $r>\frac{n}{\mu_\ast q}+\left[\frac{n}{\mu_\ast}\right]+1$ and
$0<\delta<\mu_\ast/\mu^{\ast}$. Assume moreover that 
\begin{equation}\label{ranges}
1+(\delta-1)\left(r-\frac{n}{\mu_\ast q}\right)<s\le r+1.
\end{equation}
Let $u\in\mathcal FL^p_{s-\delta\left(r-\frac{n}{\mu_\ast q}\right), M, {\rm loc}}(x_0)$ be a solution of the equation \eqref{APPL:1}, with given source $f\in\mathcal FL^p_{s-1, M, {\rm mcl}}(x_0,\xi^0)$. Then $u\in\mathcal FL^p_{s, M, {\rm mcl}}(x_0,\xi^0)$, that is 
\begin{equation}\label{cor_APPLICATION_1}
WF_{\mathcal F L^p_{s,M}}(u)\subset WF_{\mathcal F L^p_{s-1, M}}(f)\cup \textup{Char} (a). 
\end{equation}
\end{theorem}
The proof of Theorem \ref{cor_APPLICATION} is postponed to Sect. \ref{thm1.4_sct}. We end up by illustrating a simple application of Theorem \ref{cor_APPLICATION}.

\medskip
{\em Example.} Consider the linear partial differential operator in $\mathbb R^2$
\begin{equation}\label{pdo_ex}
P(x,D)=c(x)\partial_{x_1}+i\partial_{x_1}-\partial^2_{x_2}\,,
\end{equation}
where
\[
c(x):=\frac{x_1^{k_1}}{k_1!}\frac{x_2^{k2}}{k_2!}e^{-a_1x_1}e^{-a_2x_2}H(x_1)H(x_2)\,,\quad x=(x_1,x_2)\in\mathbb R^2\,,
\]
being 
$
H(t)=\chi_{(0, \infty)}(t)
$
the Heaviside function,
$k_1, k_2$ some positive integers and $a_1, a_2$ positive real numbers. 
\newline
It tends out that $c\in L^1(\mathbb R^2)$ and a direct computation gives:
\[
\widehat{c}(\xi)=\frac1{(a_1+i\xi_1)^{k_1+1}(a_2+i\xi_2)^{k_2+1}}\,,\quad\xi=(\xi_1,\xi_2)\in\mathbb R^2\,.
\]
Let us consider the vector $M=(1,2)$ and the related $M-$weight function $\langle\xi\rangle_M:=(1+\xi_1^2+\xi_2^4)^{1/2}$.
\newline
For any $p\in[1,+\infty]$ and $r>2/q+3$, $\frac1{p}+\frac1{q}=1$, one  easily proves, for a suitable constant $C=C(a_1,a_2,k_1,k_2, r)$
\[
\langle\xi\rangle_M^r\vert\widehat{c}(\xi)\vert\le\frac{C}{(1+\vert\xi_1\vert)^{k_1+1-r}(1+\vert\xi_2\vert)^{k_2+1-2r}}\,,
\]
thus $c\in\mathcal FL^p_{r,M}(\mathbb R^2)$, provided that $k_1$, $k_2$ satisfy
\begin{equation}\label{cond_k1k2}
k_1>r-1/q\quad\mbox{and}\quad k_2>2r-1/q\,.
\end{equation}
Then, under condition \eqref{cond_k1k2}, the symbol $P(x,\xi)=ic(x)\xi_1-\xi_1+\xi_2^2$ of the operator $P(x,D)$ defined in \eqref{pdo_ex} belongs to $\mathcal FL^p_{r,M}S^1_M$, cf. Definition \ref{def_bdd_smooth_symb}.
\newline
Let us set $\Omega:=\mathbb R^2\setminus\mathbb R^2_+$. Since $\vert P(x,\xi)\vert^2=c^2(x)\xi_1^2+(-\xi_1+\xi_2^2)^2$, the characteristic set of $P$ is just ${\rm Char}(P)=\Omega\times\{(\xi_1,\xi_2)\in\mathbb R^2\setminus\{(0,0)\}\,:\,\,\xi_1=\xi_2^2\}$ (cf. Definition \ref{Melliptic}) or, equivalently, $P$ is microlocally $M-$elliptic at a point $(x_0,\xi^0)=(x_{0,1},x_{0,2},\xi^0_1,\xi^0_2)\in T^\circ\mathbb R^2$ if and only if 
\[
x_{0,1}>0\,,\,\,x_{0,2}>0\quad\mbox{or}\quad\xi^0_1\neq (\xi^0_2)^2\,. 
\]    
Applying Theorem \ref{cor_APPLICATION}, for any such a point $(x_0,\xi^0)$ we have
\begin{equation*}
\begin{array}{l}
u\in\mathcal FL^p_{s-\delta\left(r-\frac{2}{q}\right), M, {\rm loc}}(x_0)\\
P(x,D)u\in\mathcal FL^p_{s-1, M, {\rm mcl}}(x_0,\xi^0)
\end{array}\Rightarrow\quad u\in\mathcal FL^p_{s, M, {\rm mcl}}(x_0,\xi^0)\,,
\end{equation*}
as long as $0<\delta<1/2$ and $1+(\delta-1)\left(r-\frac{2}{q}\right)<s\le r+1$.
\subsection{Quasi-linear PDE}\label{qlPDE}
In the last two sections, we consider few applications to the study of $M-$homogeneous Fourier Lebesgue singularities of solutions to certain classes of nonlinear PDEs.\\
Let us start with the {\em $M-$ quasi-linear} equations. Namely consider
\begin{equation}\label{quasi-lin_PDE}
\sum\limits_{\langle\alpha,1/M\rangle\le 1}a_\alpha(x,D^\beta u)_{\langle\beta,1/M\rangle\le 1-\epsilon}D^\alpha u=f(x)\,,
\end{equation}
where $a_\alpha=a_\alpha(x,D^\beta u)$ are given suitably regular functions of $x$ and partial derivatives of the unknown $u$ with $M-$order $\langle\beta,1/M\rangle$ less than or equal to $1-\epsilon$, for a given $0<\epsilon\le 1$, and where the source $f=f(x)$ is sufficiently smooth.
\newline
We define the $M-${\em principal part} of the differential operator in the left-hand side of \eqref{quasi-lin_PDE} by
\begin{equation}\label{quasi-lin_principal}
A_M(x,\xi,\zeta):=\sum\limits_{\langle\alpha,1/M\rangle=1}a_\alpha(x,\zeta)\xi^\alpha,
\end{equation}
where $x,\xi\in\mathbb R^n$, $\zeta=(\zeta_\beta)_{\langle\beta,1/M\rangle\le 1-\epsilon}\in\mathbb C^N$, $N=N(\epsilon):=\#\{\beta\in\mathbb Z^n_+\,:\,\,\langle\beta,1/M\rangle\le 1-\epsilon\}$.
It is moreover assumed that $a_\alpha$ is not identically zero for at least one multi-index $\alpha$ with $\langle\alpha,1/M\rangle=1$.
\newline
Let us take a point $(x_0,\xi^0)\in T^\circ\mathbb R^n$; we make on the equation \eqref{quasi-lin_PDE} the following assumptions:
\begin{itemize}
\item[(a)] for all $\alpha\in\mathbb Z^n_+$ satisfying $\langle\alpha,1/M\rangle\le 1$, the coefficients $a_\alpha(x,\zeta)$ are {\em locally smooth with respect to $x$ and entire analytic with respect to $\zeta$ uniformly in $x$}; that is, for some open neighborhood $U_0$ of $x_0$ 
\begin{equation}\label{analytic}
a_\alpha(x,\zeta)=\sum_{\gamma\in\mathbb{Z}^N_+}a_{\alpha,\gamma}(x)\zeta^{\gamma}, \quad \quad a_{\alpha,\gamma}\in C^{\infty}(U_0), \ \zeta\in \mathbb{C}^N\,,
\end{equation}
where for any $\beta\in\mathbb{Z}^n_+$, $\gamma\in\mathbb{Z}^N_+$ and suitable $c_{\alpha,\beta}>0$, $\sup\limits_{x\in U_0}\vert \partial_x^{\beta}a_{\alpha,\gamma}(x)\vert\leq c_{\alpha,\beta}\lambda_{\gamma}$ and the expansion $F_1(\zeta):=\sum\limits_{\gamma\in\mathbb{Z}_+^N}\lambda_{\gamma}\zeta^{\gamma}$ defines an entire analytic function;
\item[(b)] \eqref{quasi-lin_PDE} is {\em microlocally $M-$elliptic} at $(x_0,\xi^0)$, that is the $M-$prin\-ci\-pal part \eqref{quasi-lin_principal} satisfies, for some $\Gamma_M$ $M-$conic neighborhood of $\xi^0$,
\begin{equation}\label{qom}
A_M(x,\xi,\zeta)\neq 0\,,\quad\mbox{for}\,\,\,(x,\xi)\in U_0\times\Gamma_M,\,\,\,\zeta\in\mathbb C^N\,.
\end{equation}
\end{itemize}
Under the previous assumptions, we may prove the following
\begin{theorem}\label{quasi-lin_thm}
Let $p\in[1,+\infty]$, $r>\frac{n}{\mu_\ast q}+\left[\frac{n}{\mu_\ast}\right]+1$, $\frac{1}{p}+\frac{1}{q}=1$, $0<\epsilon\le 1$ and $(x_0,\xi^0)\in T^\circ\mathbb R^n$ be given, consider the quasi-linear $M-$homogeneous PDE \eqref{quasi-lin_PDE}, satisfying assumptions (a) and (b). For any $s$ such that
\begin{equation}\label{range_s}
r+1+\delta\left(r-\frac{n}{\mu_\ast q}\right)-\epsilon\le s\le r+1\,,
\end{equation}
with
\begin{equation}\label{range_delta}
0<\delta\le\frac{\epsilon}{r-\frac{n}{\mu_\ast q}}\quad\mbox{and}\quad 0<\delta<\frac{\mu_\ast}{\mu^\ast}\,,
\end{equation}
consider $u\in\mathcal FL^p_{s-\delta\left(r-\frac{n}{\mu_\ast q}\right),M,{\rm loc}}(x_0)$ a solution to \eqref{quasi-lin_PDE} with source term 
\begin{equation*}\label{source}
f\in\mathcal FL^p_{s-1,M,{\rm mcl}}(x_0,\xi^0);
\end{equation*} 
then $u\in\mathcal FL^p_{s,M,{\rm mcl}}(x_0,\xi^0)$.
\end{theorem}
\begin{proof}
From \eqref{range_s} and the other assumptions on $r$, in view of Proposition \ref{Fm_prop} (see also \cite[Proposition 8]{GM-18}) and \cite[Corollary 2]{GM-18}, from $u\in\mathcal FL^p_{s-\delta\left(r-\frac{n}{\mu_\ast q}\right),M,{\rm loc}}(x_0)$ it follows that
\begin{equation*}\label{inclusion}
D^\beta u\in \mathcal FL^p_{s-\delta\left(r-\frac{n}{\mu_\ast q}\right)-1+\epsilon, M,{\rm loc}}(x_0)\hookrightarrow\mathcal FL^p_{r, M,{\rm loc}}(x_0),
\end{equation*}
as long as $\langle\beta, 1/M\rangle\le 1-\epsilon$, hence $a_\alpha(\cdot, D^{\beta}u)_{\langle\beta, 1/M\rangle\le 1-\epsilon}\in\mathcal FL^p_{r, M,{\rm loc}}(x_0)$ for $\langle\alpha,1/M\rangle\le 1$.
\newline
Notice that conditions \eqref{range_delta} ensure that $\delta$ belongs to the interval $\left]0,\frac{\mu_\ast}{\mu^\ast}\right[$ as required by Theorem \ref{cor_APPLICATION}, see Remark \ref{best_values} below. Notice also that for $r$ satisfying the condition required by Theorem \ref{quasi-lin_thm}, $(\delta-1)\left(r-\frac{n}{\mu_\ast q}\right)+1<r+1+\delta\left(r-\frac{n}{\mu_\ast q}\right)-\epsilon$ hence the range of $s$ in \eqref{range_s} is included in the range of $s$ in the statement of Theorem \ref{cor_APPLICATION}. Therefore, we are in the position to apply Theorem \ref{cor_APPLICATION} to the symbol
\begin{equation}\label{Au}
A_u(x,\xi):=\sum\limits_{\langle\alpha,1/M\rangle\le 1}a_{\alpha}(x,D^\beta u)_{\langle\beta,1/M\rangle\le 1-\epsilon}\xi^\alpha\,,
\end{equation} 
which is of the type involved in \eqref{APPL:1} and, in particular, is microlocally $M-$elliptic at $(x_0,\xi^0)$ in the sense of \eqref{principal}. This shows the result.
\end{proof}
\begin{remark}\label{best_values}
{\rm According to the proof, we underline that in the statement of Theorem \ref{quasi-lin_thm} the assumption (b) could be relaxed to the weaker assumption that the symbol \eqref{Au} of the linear operator, which is obtained by making explicit the expression of the operator in the left-hand side of \eqref{quasi-lin_PDE} at the given solution $u=u(x)$, is microlocally $M-$elliptic at $(x_0,\xi^0)$ in the sense of \eqref{principal}.

\smallskip 
Concerning the assumptions \eqref{range_delta} on $\delta$, we note that $\frac{\mu_\ast}{\mu^\ast}\le\frac{\epsilon}{r-\frac{n}{\mu_\ast q}}$ if and only if $r\le\frac{n}{\mu_\ast q}+\frac{\mu^\ast}{\mu_\ast}\epsilon$, otherwise $\frac{\epsilon}{r-\frac{n}{\mu_\ast q}}$ is strictly smaller than $\frac{\mu_\ast}{\mu^\ast}$. Since $0<\epsilon\le 1$ and $\frac{\mu^\ast}{\mu_\ast}\ge 1$, in principle $r>\frac{n}{\mu_\ast q}+\left[\frac{n}{\mu_\ast}\right]+1$ could be either smaller or greater than $\frac{n}{\mu_\ast q}+\frac{\mu^\ast}{\mu_\ast}\epsilon$, therefore the two assumptions on $\delta$ in \eqref{range_delta} cannot be unified.
\newline
Assuming in particular $r>\frac{n}{\mu_\ast q}+\frac{\mu^\ast}{\mu_\ast}\epsilon$ and taking in the statement of Theorem \ref{quasi-lin_thm} $s=r+1$ and the best (that is biggest) amount of microlocal regularity of $u$, quantified by $\delta=\frac{\epsilon}{r-\frac{n}{\mu_\ast q}}$, we obtain
\begin{equation}\label{implicazione_r}
f\in\mathcal FL^p_{r,M,{\rm mcl}}(x_0,\xi^0)\quad\Rightarrow\quad u\in\mathcal FL^p_{r+1,M,{\rm mcl}}(x_0,\xi^0)
\end{equation}
for any solution $u$ to the equation \eqref{quasi-lin_PDE} belonging a priori to $\mathcal FL^p_{r+1-\epsilon,M,{\rm loc}}(x_0)$.
\newline
Assume now $r\le\frac{n}{\mu_\ast q}+\frac{\mu^\ast}{\mu_\ast}\epsilon$ and set again $s=r+1$ in the statement of Theorem \ref{quasi-lin_thm}; since $\frac{\epsilon}{r-\frac{n}{\mu_\ast q}}\ge\frac{\mu_\ast}{\mu^\ast}$, in this case the value $\frac{\epsilon}{r-\frac{n}{\mu_\ast q}}$ cannot be attained by $\delta\in\left]0,\frac{\mu_\ast}{\mu^\ast}\right[$, and we get that \eqref{implicazione_r} remains true for any solution belonging a priori to $\mathcal FL^p_{r+1-\delta\left(r-\frac{n}{\mu_\ast q}\right),M,{\rm loc}}(x_0)$ for any positive $\delta<\frac{\mu_\ast}{\mu^\ast}$.}
\end{remark} 
\begin{remark}\label{feneral_order}
{\rm As in the case of linear PDEs (see e.g. Theorem \ref{APPLICATION}), also in the framework of quasi-linear PDEs the result of Theorem \ref{quasi-lin_thm} can be stated for a $M-$homogeneous quasi-linear equation of {\em arbitrary} positive order $m$, namely
\begin{equation}\label{quasi-lin_PDE_m}
\sum\limits_{\langle\alpha,1/M\rangle\le m}a_\alpha(x,D^\beta u)_{\langle\beta,1/M\rangle\le m-\epsilon}D^\alpha u=f(x)\,,
\end{equation}
with $m>0$ and $0<\epsilon\le m$. In this case, the range \eqref{range_s} of $s$ will be replaced by
\begin{equation}\label{range_s_m}
r+m+\delta\left(r-\frac{n}{\mu_\ast q}\right)-\epsilon\le s\le r+m
\end{equation}
with $\delta$ satisfying \eqref{range_delta}, and the result becomes
\[
f\in\mathcal FL^p_{s-m,M,{\rm mcl}}(x_0,\xi^0)\quad\Rightarrow\quad u\in\mathcal FL^p_{s,M,{\rm mcl}}(x_0,\xi^0)
\]
for any solution $u\in\mathcal FL^p_{s-\delta\left(r-\frac{n}{\mu_\ast q}\right),M,{\rm loc}}(x_0)$ of \eqref{quasi-lin_PDE_m}.} 
\end{remark}
\subsection{Nonlinear PDE}\label{nlPDE}
Let us consider now the fully nonlinear equation
\begin{equation}\label{nleq}
F(x,D^\alpha u)_{\alpha\cdot 1/M\le m}=f(x)\,,
\end{equation}
where $F(x,\zeta)$ is locally smooth with respect to $x\in \mathbb R^n$ and entire analytic in $\zeta\in \mathbb C^N$, uniformly in $x$. Namely, for $N=\#\{ \alpha \in \mathbb Z^n_+\, :\, \langle \alpha, \frac{1}{M}\rangle\leq 1\}$ and some open neighborhood $U_0$ of $x_0$, 
\begin{equation}\label{analytic1}
F(x,\zeta)=\sum_{\gamma\in\mathbb{Z}^M_+}c_{\gamma}(x)\zeta^{\gamma}, \quad \quad c_{\gamma}\in C^{\infty}(U_0), \ \zeta\in \mathbb{C}^N\,,
\end{equation}
where for any $\beta\in\mathbb{Z}^n_+$, $\gamma\in\mathbb{Z}^N_+$ and some positive $a_\beta$, $\lambda_\gamma$, $\sum\limits_{\gamma\in\mathbb{Z}_+^N}\lambda_{\gamma}\zeta^{\gamma}$ is entire analytic in $\mathbb C^N$ and $\sup\limits_{x\in U_0}\vert \partial_x^{\beta}c_{\gamma}(x)\vert\leq a_{\beta}\lambda_{\gamma}$.

Let the equation \eqref{nleq} be microlocally $M-$elliptic at $(x_0,\xi_0)\in T^\circ\mathbb R^n$, meaning that the {\it linearized} $M-$principal symbol $A_M(x,\xi,\zeta):=\sum\limits_{\langle \alpha, 1/M\rangle=1}\frac{\partial F}{\partial\zeta_\alpha}(x,\zeta)\xi^\alpha$ satisfies
\begin{equation}\label{nlell}
\sum\limits_{\langle \alpha, 1/M\rangle =1}\frac{\partial F}{\partial\zeta_\alpha}(x,\zeta)\xi^\alpha\neq 0\,\,\,\mbox{for}\,\,(x,\xi)\in U_0\times\Gamma_M,
\end{equation}
for $\Gamma_M$ a suitable $M-$conic neighborhood of $\xi_0$.
\begin{theorem}\label{nlmicroreg}
Assume that equation \eqref{nleq} is microlocally $M-$elliptic at $(x_0,\xi^0)\in T^\circ\mathbb R^n$. For $1\le p\le +\infty$, $r>\frac{n}{\mu_\ast q}+ \left[ \frac{n}{\mu_\ast}\right]+1$, $0<\delta<\frac{\mu_\ast}{\mu^*}$, let $u\in \mathcal F L^p_{M, r+1, \textup{loc}}(x_0)$, satisfying in addition
\begin{equation}\label{der}
\partial_{x_j}u\in  \mathcal F L^p_{M, r+1-\delta (r-\frac{n}{\mu_\ast q}), \textup{loc}}(x_0) \,,\quad j=1,\dots,n\,,
\end{equation}
be a solution to \eqref{nleq}. If moreover the forcing term satisfies
\begin{equation}\label{der_f}
\partial_{x_j}f\in \mathcal F L^p_{r, M, \textup{mcl}} (x_0, \xi^0),\quad j=1,\dots,n\,,
\end{equation}
we obtain 
\begin{equation}\label{der_mcl}
\partial_{x_j}u\in  \mathcal F L^p_{r+1, M, \textup{mcl}} (x_0, \xi^0)\,,\quad j=1,\dots, n\,.
\end{equation}
\end{theorem}
\begin{proof}
For each $j=1,\dots,n$, we differentiate \eqref{nleq} with respect to $x_j$  finding that $\partial_{x_j}u$ must solve the linearized equation
\begin{equation}\label{lineq}
\sum\limits_{\langle\alpha, 1/M\rangle \le 1}\frac{\partial F}{\partial\zeta_\alpha}(x,D^\beta u)_{\langle\beta,1/M\rangle\le 1}D^\alpha\partial_{x_j}u=\partial_{x_j}f-\frac{\partial F}{\partial x_j}(x, D^\beta u)_{\langle\beta,1/M\rangle\le 1}\,.
\end{equation}
From Theorems \ref{FL_sym_cor} and \cite[Corollary 2]{GM-18}, $u\in \mathcal F L^p_{M, r+1, \textup{loc}}(x_0)$ yields that
\begin{equation*}\label{comp} 
\frac{\partial F}{\partial\zeta_\alpha}(\cdot ,D^\beta u)_{\beta\cdot 1/M\le m}\in \mathcal F L^p_{M, r, \textup{loc}}(x_0).
\end{equation*}
Because of hypotheses \eqref{der}, \eqref{der_f}, for each $j=1,\dots, n$, Theorem \ref{cor_APPLICATION} applies to $\partial_{x_j}u$, as a solution of the equation \eqref{lineq} (which is microlocally $M-$elliptic at $(x_0,\xi_0)$ in view of \eqref{nlell}), taking $s=r+1$. This proves the result.
\newline
\end{proof}
\begin{lemma}
For every $M\in\mathbb R^n_+$, $s\in\mathbb R$, $1\le p\le +\infty$, assume that $u, \partial_{x_j}u\in\mathcal F L^p_{s,M}(\mathbb R^n)$ for all $j=1,\dots,n$. Then $u\in \mathcal F L^p_{ s+\frac{\mu_*}{\mu^*},M}(\mathbb R^n)$. The same is still true if the Fourier Lebesgue spaces $\mathcal F L^p_{s,M}(\mathbb R^n)$,  $\mathcal F L^p_{ s+\frac{\mu_*}{\mu^*},M}(\mathbb R^n)$ 
are replaced by  $\mathcal F L^p_{s,M, \textup{mcl}}(x_0, \xi^0)$,  $\mathcal F L^p_{ s+\frac{\mu_*}{\mu^*},M, \textup{mcl}}(x_0,\xi^0)$ at a given point $(x_0,\xi^0)\in T^\circ\mathbb R^n$.
\end{lemma}
\begin{proof}
Let us argue for simplicity in the case of the spaces $\mathcal F L^p_{s,M}(\mathbb R^n)$, the microlocal case being completely analogous.
\newline
Notice that $u\in\mathcal F L^p_{s+\frac{\mu_*}{\mu^*},M}(\mathbb R^n)$ is completely equivalent to show that $\langle D\rangle_M^{\mu_*/\mu^*}u\in \mathcal F L^p_{s,M}(\mathbb R^n)$. By using the known properties of the Fourier transform, we may rewrite $\langle D\rangle_M^{\mu_*/\mu^*}u$ in the form
\begin{equation*}
\langle D\rangle_M^{\mu_*/\mu^*}u=\langle D\rangle_M^{\mu_*/\mu^*-2}u+\sum\limits_{j=1}^n\Lambda_{j, M}(D)(D_{x_j}u)\,,
\end{equation*}
where $\Lambda_{j,M}(D)$ is the pseudodifferential operator with symbol $\langle\xi\rangle_M^{\mu_*/\mu^*-2}\xi_j^{2\mu_j-1}$, that is
\begin{equation*}
\Lambda_{j,M}(D)v:=\mathcal F^{-1}\left(\langle\xi\rangle_M^{\mu_*/\mu^*-2}\xi_j^{2\mu_j-1}\widehat v\right)\,,\quad j=1,\dots,n\,.
\end{equation*}
Since $\langle\xi\rangle_M^{\mu_*/\mu^*-2}\xi_j^{2\mu_j-1}\in S^{\mu_*/\mu^*-\mu_*/\mu_j}_M$, the result follows at once from Proposition \ref{Fm_prop}.
\end{proof}
\noindent
As a straightforward application of the previous lemma, the following consequence of Theorem \ref{nlmicroreg} can be proved.
\begin{corollary}\label{thm_nl}
Under the same assumptions of Theorem \ref{nlmicroreg} we have that $u\in \mathcal F L^p_{r+1+\frac{\mu_*}{\mu^*}, M, \textup{mcl}}(x
_0,\xi^0)$.
\end{corollary}
\begin{remark}
{\rm Notice that if $\left(r-\frac{n}{\mu_* q}\right)\delta\ge 1$ then any $u\in \mathcal F L^p_{r+1, M,\textup{loc}}(x_0)$ rightly satisfies \eqref{der}.\\
 Thus $\partial_{x_j}u\in \mathcal F L^p_{r+1-\frac{\mu_*}{\mu_j}, M, \textup{loc}}(x_0)\hookrightarrow\mathcal F L^p_{M, r+1-\delta (r-\frac{n}{\mu_\ast q}), \textup{loc}}(x_0) $ being $\mu_*/\mu_j\le 1\le \left(r-\frac{n}{\mu_* q}\right)\delta$ for each $j=1,\dots,n$. Notice that, for $r>\frac{n}{\mu_* q}+\frac{\mu^*}{\mu_*}$ we can  find $\delta^*\in]0,\mu_*/\mu^*[$ such that $\left(r-\frac{n}{\mu_* q}\right)\delta\ge 1$: it suffices to choose an arbitrary $\delta^*\in\left[\frac{1}{r-\frac{n}{\mu_* q}}, \frac{\mu_*}{\mu^*}\right[$. Hence, applying Theorem \ref{nlmicroreg} with such a $\delta^*$ we conclude that if $r>\frac{n}{\mu_* q}+\frac{\mu^*}{\mu_*}$ and the right-hand side $f$ of equation \eqref{nleq} obeys to condition \eqref{der_f} at a point $(x_0,\xi^0)\in T^{\circ}\mathbb R^n$, then every solution $u\in \mathcal F L^p_{r+1, M,\textup{loc}}(x_0)$ to such an equation satisfies condition \eqref{der_mcl}; in particular $u\in \mathcal F L^p_{r+1+\frac{\mu_*}{\mu^*}, M, \textup{mcl}}(x
_0,\xi^0)$.}
\end{remark}

\section{Dyadic decomposition}\label{FL-sp_sct}


We are going to introduce a covering of the frequency space $\mathbb R^n_\xi$ made by {\it dyadic shells} modeled on the $M-$ norm $\vert\cdot\vert_M$ defined in \eqref{M-h-n}, and an associated smooth {\it partition of unity}. In the following, this will be used to provide a useful characterization of $M-$homogeneous Fourier Lebesgue regularity.

For a fixed number $K\ge 1$ we set
\begin{equation}\label{dyad_cov}
\begin{split}
&\mathcal C^{M, K}_{-1}:=\left\{\xi\in\mathbb R^n\,:\,\,\vert\xi\vert_M\le K\right\}\,,\\
&\mathcal C^{M, K}_{h}:=\{\xi\in\mathbb R^n\,:\,\,\frac1{K}2^{h-1}\le\vert\xi\vert_M\le K2^{h+1}\}\,,\,\,\,h=0,1,\dots\,.
\end{split}
\end{equation}
It is clear that the shells (crowns) $\mathcal C^{M,K}_{h}$, for $h\ge -1$, provide a covering of $\mathbb R^n$. For the sequel of our analysis, a fundamental property of this covering is that the number of overlapping shells does not increase with the index $h$; more precisely there exists a positive number $N_0=N_0(K)$ such that
\begin{equation}\label{overlap_cond}
\mathcal C^{M,K}_p\cap\mathcal C^{M,K}_q=\emptyset\,,\quad\mbox{for}\,\,\vert p-q\vert>N_0\,.
\end{equation}
Consider now a real-valued function $\Phi=\Phi(t)\in C^\infty([0,+\infty[)$ satisfying
\begin{equation}\label{cond_phi}
\begin{split}
& 0\le\Phi(t)\le 1\,,\quad\forall\,t\ge 0\,,\\
& \Phi(t)=1\quad\mbox{for}\,\,0\le t\le\frac{1}{2K}\,,\quad\Phi(t)=0\quad\mbox{for}\,\,t>K\,,
\end{split}
\end{equation}
and define the sequence $\{\varphi_h\}_{h=-1}^{+\infty}$ in $C^\infty(\mathbb R^n)$ by setting for $\xi\in\mathbb R^n$
\begin{equation}\label{part_unity}
\varphi_{-1}(\xi):=\Phi(\vert\xi\vert_M)\,,\quad\varphi_{h}(\xi):=\Phi\left(\frac{\vert\xi\vert_M}{2^{h+1}}\right)-\Phi\left(\frac{\vert\xi\vert_M}{2^h}\right)\,,\,\,\,h=0,1,\dots\,.
\end{equation}

It is easy to check that the sequence $\{\varphi_h\}_{h=-1}^{\infty}$ defined above enjoys the following properties:
\begin{align}
&{\rm supp}\,\varphi_h\subseteq\mathcal C_h^{M, K}\,,\quad\mbox{for}\,\,h\ge -1;\label{supp_cond}\\
&\sum\limits_{h=-1}\limits^\infty\varphi_h(\xi)=1\,,\quad\text{for all}\,\, \xi\in \mathbb R^n;\label{unity_cond}\\
&\sum\limits_{h=-1}\limits^\infty u_h=u\,, \quad\text{with convergence in }\,\, \mathcal S^\prime(\mathbb R^n)\,,\label{conv_cond}\\
\end{align}
where it is set
\begin{equation}\label{dyad_blocks}
u_h:=\varphi_h(D)u\qquad\mbox{for}\,\,\,h\ge -1\,.
\end{equation}
Moreover, as a consequence of \eqref{overlap_cond}, for any fixed $\xi\in\mathbb{R}^n$ the sum in \eqref{unity_cond} reduces to a finite number of terms independently of the choice of $\xi$ itself, namely an integer $h_0=h_0(\xi)\ge -1$ exists such that
\begin{equation}\label{conv_finite}
\sum\limits_{h=-1}^{\infty}\varphi_h(\xi)\equiv\sum\limits_{h=\tilde h_0}^{h_0+N_0}\varphi_h(\xi)\,,\quad\mbox{where}\,\,\,\tilde h_0=\tilde h_0(\xi):=\max\{-1,h_0-N_0\} 
\end{equation}
and $N_0$ is the integer number involved in \eqref{overlap_cond}.

\medskip
The sequence $\{\varphi_h\}_{h=-1}^{+\infty}$ of functions under the above assumptions will be referred to as a {\em $M-$homogeneous dyadic partition of unity}, and the series in the left-hand side of \eqref{conv_cond} will be called a {\em $M-$homogeneous dyadic decomposition} of $u\in\mathcal S^\prime(\mathbb R^n)$; in the homogeneous case $M=(1,\dots,1)$, such a decomposition reduces to the classical {\em Littlewood--Paley decomposition} of $u$, cf. for example \cite{AG-book}. A characterization similar to that of Sobolev spaces on $L^2$ in terms of a dyadic decomposition of distributions, see \cite{BO}, can be proven in the $M-$homogeneous case for the spaces $\mathcal FL^p_{s,M}$ with arbitrary exponent $p\in[1,+\infty]$.
\begin{proposition}\label{prop_1}
For $M=(\mu_1,\dots,\mu_n)\in\mathbb R^n_+$, $s\in\mathbb R$ and $p\in[1.+\infty]$, a distribution $u\in\mathcal S^\prime(\mathbb R^n)$ belongs to the space $\mathcal FL^p_{s,M}$ if and only if
\begin{equation}\label{Lp_cond1}
\widehat u_h\in L^p(\mathbb R^n)\,,\qquad\mbox{for all}\,\,h\ge -1\,,
\end{equation}
and
\begin{equation}\label{Lp_cond2}
\sum\limits_{h=-1}^{+\infty}2^{shp}\Vert\widehat u_h\Vert_{L^p}^p<+\infty\,.
\end{equation}
Under the above assumptions,
\begin{equation}\label{equiv_FL-norm}
\left(\sum\limits_{h=-1}^{+\infty}2^{shp}\Vert\widehat u_h\Vert_{L^p}^p\right)^{1/p}
\end{equation}
provides a norm in $\mathcal FL^p_{s,M}$ equivalent to \eqref{FL-norm}.

For $p=+\infty$, condition \eqref{Lp_cond2} (as well as the norm \eqref{equiv_FL-norm}) must be modified in the obvious way.
\end{proposition}
\begin{proof}
Let us first observe that the $M-$ weight $\langle\cdot\rangle_M$ is equivalent to $2^h$ on the support of $\varphi_h$; indeed
\begin{equation}\label{equiv:h}
\begin{split}
&1\le\langle\xi\rangle_M\le(1+K^2)^{1/2}\,,\qquad\mbox{for}\,\,\xi\in{\supp}\,\varphi_{-1}\,;\\
&\frac1{2K}2^h\le\langle\xi\rangle_M\le (1+4K^2)^{1/2}2^h\,,\qquad\mbox{for}\,\,\xi\in{\rm supp}\,\varphi_h\,\,\,\mbox{and}\,\,\,h\ge 0\,,
\end{split}
\end{equation}
being $K$ the positive number involved in \eqref{cond_phi}.

For $p\in[1,+\infty[$, it is enough arguing on smooth functions $u\in\mathcal S(\mathbb R^n)$ in view of density of $\mathcal S(\mathbb R^n)$ in $\mathcal FL^p_{s,M}$. For $\xi\in\mathbb R^n$, from \eqref{unity_cond}, \eqref{conv_finite} we derive
\begin{equation}\label{equiv:prel}
\begin{split}
\sum\limits_{h=-1}^{\infty}\vert\widehat u_h(\xi)\vert^p &\equiv\sum\limits_{h=\tilde h_0}^{h_0+N_0}\varphi_h(\xi)^p\vert\widehat u(\xi)\vert^p\le\vert\widehat u(\xi)\vert^p\equiv\left(\sum\limits_{h=\tilde h_0}^{h_0+N_0}\varphi_h(\xi)\vert\widehat u(\xi)\vert\right)^p\\
&\le C_{N_0, p}\sum\limits_{h=\tilde h_0}^{h_0+N_0}\varphi_h(\xi)^p\vert\widehat u(\xi)\vert^p\equiv C_{N_0, p}\sum\limits_{h=-1}^{\infty}\vert\widehat u_h(\xi)\vert^p\,,
\end{split}
\end{equation}
where $h_0=h_0(\xi)$ and $\tilde h_0=\tilde h_0(\xi)$ are the integers from \eqref{conv_finite} and $C_{N_0, p}>1$ is some constant depending only on $N_0$ and $p$ (in particular, it is independent of $h$ and $\xi$). Hence, multiplying each side of \eqref{equiv:prel} by $\langle\xi\rangle_M^{sp}$, making use of \eqref{equiv:h} and integrating on $\mathbb R^n$ yields
\[
\frac1{C_{s,p,K}}\sum\limits_{h=-1}^{\infty}2^{shp}\Vert\widehat u_h\Vert^p\le\Vert u\Vert^p_{\mathcal FL^p_{s,M}}\le C_{s,p,K}\sum\limits_{h=-1}^{\infty}2^{shp}\Vert\widehat u_h\Vert^p\,,
\]
for a suitable constant $C_{s, p, K}>1$ depending only on $s$, $p$ and $K$. This proves the statement of Proposition \ref{prop_1}.

For the case $p=+\infty$, let us argue directly (in the absence of the density of $\mathcal S(\mathbb R^n)$ in $\mathcal FL^\infty_{s,M}$). Thus for arbitrary $u\in\mathcal FL^\infty_{s,M}$ and every $h\ge -1$, writing
\begin{equation}\label{id1}
\widehat u_h=\frac{\varphi_h}{\langle\cdot\rangle_M^s}\langle\cdot\rangle_M^s\widehat u\,,
\end{equation}
we get $\widehat u_h\in L^\infty(\mathbb R^n)$, since $\langle\cdot\rangle_M^s\widehat u\in L^\infty(\mathbb R^n)$ and, in view of \eqref{equiv:h} and \eqref{supp_cond},
\begin{equation}\label{est_phi_h}
2^{sh}\left\vert\frac{\varphi_h(\xi)}{\langle\xi\rangle_M^s}\right\vert\le C_{s,K}\,,\qquad\forall\,\xi\in\mathbb R^n\,,
\end{equation}
where the constant $C_{s,K}$ depends only on $s$ and $K$. From \eqref{id1} and \eqref{est_phi_h} 
\[
2^{sh}\vert\widehat u_h(\xi)\vert\le C_{s,K}\Vert u\Vert_{\mathcal FL^\infty_{s,M}}\,,\quad\forall\,\xi\in\mathbb R^n\,,\,\,\,h\ge -1\,,
\]
follows at once, which implies \eqref{Lp_cond2} with $p=+\infty$.

Conversely, let us suppose that $u\in\mathcal S^\prime(\mathbb R^n)$ makes conditions \eqref{Lp_cond1}, \eqref{Lp_cond2} to be satisfied. From \eqref{overlap_cond}, \eqref{conv_cond} and \eqref{equiv:h} we get for an arbitrary $\ell\ge -1$ and every $\xi\in\mathcal C^{M,K}_\ell$:
\begin{equation*}
\begin{split}
\vert\langle\xi\rangle_M^s &\widehat u(\xi)\vert\le\sum\limits_{h=-1}^{+\infty}\vert\langle\xi\rangle_M^s\widehat u_h(\xi)\vert=\sum\limits_{h=\ell-N_0}^{\ell+N_0}\vert\langle\xi\rangle_M^s\widehat u_h(\xi)\vert\le C_{s,K}\sum\limits_{h=\ell-N_0}^{\ell+N_0}2^{sh}\vert\widehat u_{h}(\xi)\vert\\
&\le C_{s,K}(2N_0+1)\sup\limits_{h\ge -1}2^{sh}\Vert\widehat u_h\Vert_{L^\infty}\,,
\end{split}
\end{equation*}
giving that $u$ belongs to $\mathcal FL^\infty_{s,M}$ and satisfies
\begin{equation*}
\Vert u\Vert_{\mathcal FL^\infty_{s,M}}\le C_{s,K}(2N_0+1)\sup\limits_{h\ge -1}2^{sh}\Vert\widehat u_h\Vert_{L^\infty}\,.
\end{equation*}
The proof is complete.
\end{proof}
\begin{remark}\label{rmk:partition_unity}
{\rm Arguing along the same lines followed in the proof of estimates \eqref{est_phi_h}, one can prove the following estimates for the derivatives of functions $\varphi_h$'s: for all $\nu\in\mathbb Z^n_+$ a positive constant $C_\nu$ exists such that
\begin{equation}\label{der_phi_h}
\vert D^\nu_\xi\varphi_h(\xi)\vert\le C_\nu 2^{-h\langle\nu,1/M\rangle}\,,\quad\forall\,\xi\in\mathbb R^n\,,\,\,\,h=-1,0,1,\dots\,.
\end{equation}
Notice also, in view of \eqref{equiv:h}, estimates
\eqref{der_phi_h} can be stated in the equivalent form
\begin{equation*}
\vert D^\nu_\xi\varphi_h(\xi)\vert\le C_\nu \langle\xi\rangle_M^{-\langle\nu,1/M\rangle}\,,\quad\forall\,\xi\in\mathbb R^n\,,\,\,\,h=-1,0,1,\dots\,.
\end{equation*}
}
\end{remark}
Along the same arguments of Bony \cite{BO}, one can show the following
\begin{proposition}\label{prop_2}
Let $M=(\mu_1,\dots,\mu_n)\in\mathbb R^n_+$ and $p\in[1.+\infty]$.
\begin{itemize}
\item[(i)] For $s\in\mathbb R$, let $\left\{u_{h}\right\}_{h=-1}^{+\infty}$ be a sequence of distributions $u_h\in\mathcal S^\prime(\mathbb R^n)$ satisfying the following conditions:

\smallskip
\begin{itemize}
\item[(a)] there exists a constant $K\ge 1$ such that
\begin{equation*}
{\rm supp}\,\widehat u_h\subseteq\mathcal C^{M,K}_h\,,\qquad\mbox{for all}\,\,h\ge -1\,;
\end{equation*}
\item[(b)]
\begin{equation}\label{sum_cond}
\sum\limits_{h=-1}^{+\infty}2^{shp}\Vert\widehat u_h\Vert_{L^p}^p<+\infty
\end{equation}
(with obvious modification for $p=+\infty$).
\end{itemize}
Then $u=\sum\limits_{h=-1}^{+\infty}u_h\in\mathcal S^\prime(\mathbb R^n)$ (where the series in the right-hand side is convergent in $\mathcal S^\prime(\mathbb R^n)$) belongs to $\mathcal FL^p_{s,M}$ and satisfies
\begin{equation}\label{est_6}
\Vert u\Vert_{\mathcal FL^p_{s,M}}\le C_{s,p,K}\left(\sum\limits_{h=-1}^{+\infty}2^{shp}\Vert\widehat u_h\Vert_{L^p}^p\right)^{1/p}\,,
\end{equation}
with positive constant $C$ depending only on $s$, $p$ and $K$.
\item[(ii)] If $s>0$, the same result stated in (i) is still valid when a distribution sequence $\left\{u_{h}\right\}_{h=-1}^{+\infty}$ satisfies the following condition

\smallskip
\begin{itemize}
\item[(a')] there exists a constant $K\ge 1$ such that
\begin{equation*}
{\rm supp}\,\widehat u_h\subseteq\mathcal B^{M,K}_h:=\{\xi\in\mathbb R^n\,:\,\,\vert\xi\vert_M\le K2^{h+1}\}\,,\qquad\mbox{for all}\,\,h\ge -1\,,
\end{equation*}
\end{itemize}
instead of (a) (notice that $\mathcal B_{-1}^{M,K}\equiv \mathcal C_{-1}^{M,K}$), and condition (b).
\end{itemize}
\end{proposition}

\section{Proof of Theorem \ref{FL_thm}}\label{PROOF1}
Following closely the arguments in Coifmann--Meyer \cite{CM-book}, see also Garello--Morando \cite{GM-05}, one proves that every zero order symbol in the  classes $\mathcal FL^p_{r,M}S^0_{M,\delta}(N)$ can be expanded into a series of ``elementary terms''.
\begin{lemma}\label{el_symb_lemma}
For $p\in[1,+\infty]$, $r>\frac{n}{\mu_\ast q}$ (being $q$ the conjugate exponent of $p$), $N>n+1$ a positive integer number and $\delta\in[0,1]$, let $a(x,\xi)\in\mathcal FL^{p}_{r,M}S^0_{M,\delta}(N)$. Then there exist positive numbers $C, L$, $K>1$ and a sequence $\{c_k\}_{k\in\mathbb Z^n}$ of positive numbers $c_k$ satisfying $\sum\limits_{k\in\mathbb Z^n}c_k<+\infty$ such that
\begin{equation}\label{dec_el_symb}
a(x,\xi)=\sum\limits_{k\in\mathbb Z^n}c_ka_k(x,\xi)\,,
\end{equation}
with absolute convergence in $L^\infty(\mathbb R^n\times\mathbb R^n)$ of the series in the right-hand side, and, for each $k\in\mathbb Z^n$,
\begin{equation}\label{k_el_symb}
a_k(x,\xi)=\sum\limits_{h=-1}^{+\infty}d_h^k(x)\psi^k_h(\xi)\,,
\end{equation}
with suitable sequences $\{d^k_h\}_{h=-1}^{+\infty}$ in $\mathcal FL^1\cap\mathcal FL^p_{r,M}$ and $\{\psi^k_h\}_{h=-1}^{+\infty}$ in $C^\infty_0(\mathbb R^n)$ obeying the following conditions:
\begin{itemize}
\item[{\rm (a)}] $\Vert d^k_h\Vert_{\mathcal FL^1}\le H\,,\quad\Vert d^k_h\Vert_{\mathcal FL^p_{r,M}}\le H 2^{h\delta\left(r-\frac{n}{\mu_\ast q}\right)}$ for all $h=-1,0,\dots$;
\item[{\rm (b)}] ${\rm supp}\,\psi^k_h\subseteq\mathcal C^{M,K}_h$, $h=-1,0,\dots$;
\item[{\rm (c)}] $\vert\partial^\alpha\psi^k_h(\xi)\vert\le C2^{-\langle\alpha,1/M\rangle h}$, $\forall\,\xi\in\mathbb R^n$, $\vert\alpha\vert\le N$.
\end{itemize}
\end{lemma}
In view of \eqref{overlap_cond} and condition (b) above, the series in the right-hand side of \eqref{k_el_symb} has only finitely many nonzero terms at each point $(x,\xi)$. Conditions (a)-(c) above also imply that $a_k(x,\xi)$ defined by \eqref{k_el_symb} belongs to $\mathcal FL^{p}_{r,M}S^0_{M,\delta}(N)$ for each $k\in\mathbb Z^n$. A symbol of the form \eqref{k_el_symb} will be referred to as an {\em elementary symbol}.
\medskip

The proof of Theorem \ref{FL_thm} follows the same arguments as in \cite{GM-08}. Without loss of generality, we may reduce to prove the statement of the theorem in the case of a symbol $a(x,\xi)\in\mathcal FL^p_{r,M}S^0_{M,\delta}(N)$. Also, because of Lemma \ref{el_symb_lemma}, it will be enough to show the result in the case when $a(x,\xi)$ is an elementary symbol, namely
\[
a(x,\xi)=\sum\limits_{h=-1}^{+\infty}d_h(x)\psi_h(\xi)\,,
\]
where the sequences $\{d_h\}_{h=-1}^{+\infty}$ and $\{\psi_h\}_{h=-1}^{+\infty}$ obey the assumptions (a)--(c).

In view of Lemma \ref{el_symb_lemma} there holds
\begin{equation}\label{dec_oprt}
a(x,D)u(x)=\sum\limits_{h=-1}^{+\infty}d_h(x)u_h(x)\,,\quad\forall\,u\in\mathcal S(\mathbb R^n)\,,
\end{equation}
where
\begin{equation}\label{uh}
u_h:=\psi_h(D)u\,,\qquad h=-1,0,\dots\,.
\end{equation}
Let $\{\varphi_\ell\}_{\ell\ge -1}$ be an $M-$homogeneous dyadic partition of unity; then we may decompose \eqref{dec_oprt} as follows
\begin{equation}
a(x,D)u(x)=\sum\limits_{h=-1}^{+\infty}\sum\limits_{\ell=-1}^{+\infty}d_{h,\ell}(x)u(x)=T_1u(x)+T_2u(x)+T_3u(x)\,,
\end{equation}
where it is set
\begin{equation}\label{T1}
T_1u(x):=\sum\limits_{h=N_0-1}^{+\infty}\sum\limits_{\ell=-1}^{h-N_0}d_{h,\ell}(x)u_h(x)\,,
\end{equation}
\begin{equation}\label{T2}
T_2u(x):=\sum\limits_{h=-1}^{+\infty}\sum\limits_{\ell=\ell_h}^{h+N_0-1}d_{h,\ell}(x)u_h(x)\qquad (\ell_h:=\max\{-1,h-N_0+1\})\,,
\end{equation}
\begin{equation}\label{T3}
T_3u(x):=\sum\limits_{\ell=N_0-1}^{+\infty}\sum\limits_{h=-1}^{\ell-N_0}d_{h,\ell}(x)u_h(x)\,,
\end{equation}
with sufficiently large integer $N_0>0$, and
\begin{equation}
d_{h,\ell}:=\varphi_\ell(D)d_h\,,\qquad h,\ell=-1,0,\dots\,.
\end{equation}
Then the proof of Theorem \ref{FL_thm} follows from combining the following continuity results concerning the different operators $T_1$, $T_2$, $T_3$.
\medskip

Henceforth, the following general notation will be adopted: for every pair of Banach spaces $X$, $Y$, we will write $\Vert T\Vert_{X\rightarrow Y}$ to mean the operator norm of every linear bounded operator $T$ from $X$ into $Y$.
\begin{lemma}\label{T1_lemma}
For all $s\in\mathbb R$, $T_1$ extends to a linear bounded operator
\begin{equation}\label{T1_bdd}
T_1:\mathcal FL^p_{s,M}\rightarrow\mathcal FL^p_{s,M}
\end{equation}
and there exists a positive constant $C=C_{s,p}$ such that
\begin{equation}\label{norma_T1}
\Vert T_1\Vert_{\mathcal FL^p_{s,M}\rightarrow \mathcal FL^p_{s,M}}\le C\sup\limits_{h\ge -1}\Vert d_h\Vert_{\mathcal FL^1}
\end{equation}
\end{lemma}
\begin{proof}
Taking $N_0>0$ sufficiently large, we find a suitable $T>1$ such that
\begin{equation*}
{\rm supp}\,\widehat{d_{h,\ell}u_h}\subseteq\mathcal C^{M,T}_h\,,\quad\mbox{for}\,\,-1\le\ell\le h-N_0\,\,\,\mbox{and}\,\,\,h\ge N_0-1\,.
\end{equation*}
Then in view of Proposition \ref{prop_2} (i), for every $s\in\mathbb R$ a positive constant $C=C_{s,p}$ exists such that
\[
\Vert T_1u\Vert_{\mathcal FL^p_{s,M}}^p\le C\sum\limits_{h\ge N_0-1}2^{shp}\left\Vert\sum\limits_{\ell=-1}^{h-N_0}\widehat{d_{h,\ell}u_h}\right\Vert_{L^p}^p\,;
\]
on the other hand
\[
\sum\limits_{\ell=-1}^{h-N_0}\widehat{d_{h,\ell}u_h}=(2\pi)^{-n}\sum\limits_{\ell=-1}^{h-N_0}\widehat{d_{h,\ell}}\ast\widehat{u_h}=(2\pi)^{-n}\sum\limits_{\ell=-1}^{h-N_0}\varphi_{\ell}\widehat{d_{h}}\ast\widehat{u_h}
\]
hence Young's inequality yields
\[
\left\Vert\sum\limits_{\ell=-1}^{h-N_0}\widehat{d_{h,\ell}u_h}\right\Vert_{L^p}\le(2\pi)^{-n}\left\Vert\sum\limits_{\ell=-1}^{h-N_0}\varphi_{\ell}\widehat{d_{h}}\right\Vert_{L^1}\Vert\widehat u_h\Vert_{L^p}
\]
and, in view of \eqref{unity_cond},
\[
\left\Vert\sum\limits_{\ell=-1}^{h-N_0}\varphi_{\ell}\widehat{d_{h}}\right\Vert_{L^1}=\int\sum\limits_{\ell=-1}^{h-N_0}\varphi_{\ell}(\xi)\vert\widehat{d_h}(\xi)\vert d\xi\le\int\vert\widehat{d_h}(\xi)\vert d\xi=\Vert\widehat{d_h}\Vert_{L^1}\,.
\]
Combining the preceding estimates and because of Lemma \ref{el_symb_lemma} and Proposition \ref{prop_1} we get
\begin{equation*}\label{T1}
\Vert T_1u\Vert_{\mathcal FL^p_{s,M}}^p\!\!\le C\!\left(\sup\limits_{h\ge -1}\Vert d_h\Vert_{\mathcal FL^1}\right)^p\!\!\!\!\!\sum\limits_{h\ge N_0-1}2^{shp}\vert\widehat{u_h}\vert_{L^p}^p\!\le C\!\left(\sup\limits_{h\ge -1}\Vert d_h\Vert_{\mathcal FL^1}\!\right)^p\!\!\Vert u\Vert_{\mathcal FL^p_{s,M}}^p\,.
\end{equation*}
This ends the proof of lemma.
\end{proof}
\begin{lemma}\label{T2_lemma}
	For all $s>(\delta-1)\left(r-\frac{n}{\mu_\ast q}\right)$, $T_2$ extends to a linear bounded operator
	\begin{equation}\label{T2_bdd}
	T_2:\mathcal FL^p_{s,M}\rightarrow\mathcal FL^p_{s+(1-\delta)\left(r-\frac{n}{\mu_\ast q}\right),M}
	\end{equation}
	and there exists a positive constant $C=C_{N_0,p,r,s}$ such that
	\begin{equation}\label{norma_T2}
	\Vert T_2\Vert_{\mathcal FL^p_{s,M}\rightarrow\mathcal FL^p_{s+(1-\delta)\left(r-\frac{n}{\mu_\ast q}\right)}}\le C\sup\limits_{h\ge -1}2^{-\delta\left(r-\frac{n}{\mu_\ast q}\right)h}\Vert d_h\Vert_{\mathcal FL^p_{r,M}}\,.
	\end{equation}
\end{lemma}
\begin{proof}
Taking $N_0>0$ sufficiently large, we find a suitable $T>1$ such that
\begin{equation}\label{suppT2}
{\rm supp}\,\widehat{d_{h,\ell}u_h}\subseteq\{\xi\,:\,\,\vert\xi\vert_M\le T2^{h+1}\}\,,\quad\mbox{for}\,\,\ell_h\le\ell\le h+N_0-1\,,\,\,\,h\ge -1
\end{equation}
and where $\ell_h:=\max\{-1,h-N_0+1\}$. From Proposition \ref{prop_2} (ii), for $s>(\delta-1)\left(r-\frac{n}{\mu_\ast q}\right)$ we get
\[
\Vert T_2u\Vert_{\mathcal FL^p_{s+(1-\delta)\left(r-\frac{n}{\mu_\ast q}\right),M}}^p\le C\sum\limits_{h\ge -1}2^{\left(s+(1-\delta)\left(r-\frac{n}{\mu_\ast q}\right)\right)hp}\left\Vert\sum\limits_{\ell=\ell_h}^{h+N_0-1}\widehat{d_{h,\ell}u_h}\right\Vert_{L^p}^p\,;
\]
and again from Young's inequality
\[
\left\Vert\sum\limits_{\ell=\ell_h}^{h+N_0-1}\widehat{d_{h,\ell}u_h}\right\Vert_{L^p}\!\!\!\le(2\pi)^{-n}\sum\limits_{\ell=\ell_h}^{h+N_0-1}\Vert\widehat{d_{h,\ell}}\ast\widehat{u_h}\Vert_{L^p}\le(2\pi)^{-n}\sum\limits_{\ell=\ell_h}^{h+N_0-1}\Vert\widehat{d_{h,\ell}}\Vert_{L^p}\Vert\widehat{u_h}\Vert_{L^1}\,;
\]
thus, since the number of indices $\ell$ such that $\ell_h\le\ell\le h+N_0-1$ is bounded above by a positive number independent of $h$ one has
\[
\begin{split}
\Vert T_2u &\Vert_{\mathcal FL^p_{s+(1-\delta)\left(r-\frac{n}{\mu_\ast q}\right),M}}^p\le C\sum\limits_{h\ge -1}2^{\left(s+(1-\delta)\left(r-\frac{n}{\mu_\ast q}\right)\right)hp}\left(\sum\limits_{\ell=\ell_h}^{h+N_0-1}\Vert\widehat{d_{h,\ell}}\Vert_{L^p}\Vert\widehat{u_h}\Vert_{L^1}\right)^p\\
&\le C_{N_0,p}\sum\limits_{h\ge -1}2^{\left(s+(1-\delta)\left(r-\frac{n}{\mu_\ast q}\right)\right)hp}\Vert\widehat{u_h}\Vert_{L^1}^p\sum\limits_{\ell=\ell_h}^{h+N_0-1}\Vert\widehat{d_{h,\ell}}\Vert_{L^p}^p\\
&=C_{N_0,p}\sum\limits_{h\ge -1}2^{shp}\,2^{-\frac{n}{\mu_\ast q}hp}\Vert\widehat{u_h}\Vert_{L^1}^p\,2^{rhp}2^{-\delta\left(r-\frac{n}{\mu_\ast q}\right)hp}\sum\limits_{\ell=\ell_h}^{h+N_0-1}\Vert\widehat{d_{h,\ell}}\Vert_{L^p}^p\,.
\end{split}
\]
Notice also that H\"older's inequality yields
\[
\Vert\widehat{u_h}\Vert_{L^1}=\int_{\mathcal C^{M,K}_h}\vert\widehat{u_h}(\xi)\vert d\xi\le\Vert\widehat{u_h}\Vert_{L^p}\left(\int_{\mathcal C^{M,K}_h }d\xi\right)^{1/q}\le C\Vert\widehat{u_h}\Vert_{L^p} 2^{\frac{nh}{\mu_\ast q}}\,,
\]
hence
\[
2^{-\frac{nhp}{\mu_\ast q}}\Vert\widehat{u_h}\Vert_{L^1}^p\le C\Vert\widehat{u_h}\Vert_{L^p}^p\,.
\]
Moreover, for a suitable constant $C_{N_0}>0$ depending only on $N_0$,
\[
2^h\le C_{N_0}2^{\ell}\,,\quad\mbox{for}\,\,\, \ell_h\le\ell\le h+N_0-1\,.
\]
Hence we get
\[
\begin{split}
2^{rhp} &2^{-\delta\left(r-\frac{n}{m_ast q}\right)hp}\sum\limits_{\ell=\ell_h}^{h+N_0-1}\Vert\widehat{d}_{h,\ell}\Vert_{L^p}^p\le C_{N_0,r,p}2^{-\delta\left(r-\frac{n}{m_ast q}\right)hp}\sum\limits_{\ell=\ell_h}^{h+N_0-1}2^{r\ell p}\Vert\widehat{d}_{h,\ell}\Vert_{L^p}^p\\
&\le\widetilde{C}_{N_0,r,p}2^{-\delta\left(r-\frac{n}{m_ast q}\right)hp}\Vert d_h\Vert_{\mathcal FL^p_{r,M}}^p\le\widetilde{C}_{N_0,r,p}H^p\,,
\end{split}
\]
where
\begin{equation}\label{H}
H:=\sup\limits_{h\ge -1}2^{-\delta\left(r-\frac{n}{m_\ast q}\right)h}\Vert d_h\Vert_{\mathcal FL^p_{r,M}}\,,
\end{equation}
and, in view of Proposition \ref{prop_1},
\[
\Vert T_2u\Vert^p_{\mathcal FL^p_{s+(1-\delta)\left(r-\frac{n}{\mu_\ast q}\right),M}}\le\widetilde C_{N_0,r,p}H^p\sum\limits_{h=-1}^{+\infty}2^{shp}\Vert\widehat u_h\Vert_{L^p}^p\le\widehat C_{N_0,p,r,s}H^p\Vert u\Vert^p_{\mathcal FL^p_{s,M}}\,.
\]
This ends the proof of Lemma \ref{T2_lemma}.
\end{proof}
\begin{remark}\label{cons_lemma}
{\rm Since for $0\le\delta\le 1$ and $r>\frac{n}{\mu_\ast q}$ we have $s+(1-\delta)\left(r-\frac{n}{\mu_\ast q}\right)\ge s$, as an immediate consequence of Lemma \ref{T2_lemma}, we get the boundedness of $T_2$ as a linear operator $T_2:\mathcal FL^p_{s,M}\rightarrow\mathcal FL^p_{s,M}$.}
\end{remark}
\begin{lemma}\label{T3_lemma}
	For all $s<r$, $T_3$ extends to a linear bounded operator
	\begin{equation}\label{T3_bdd_1}
	T_3:\mathcal FL^p_{s+(\delta-1)\left(r-\frac{n}{\mu_\ast q}\right),M}\rightarrow\mathcal FL^p_{s,M}
	\end{equation}
	and there exists a positive constant $C=C_{s,p,r}$ such that
	\begin{equation}\label{norma_T3_1}
	\Vert T_3\Vert_{\mathcal FL^p_{s+(\delta-1)\left(r-\frac{n}{\mu_\ast q}\right),M}\rightarrow\mathcal FL^p_{s,M}}\le C\sup\limits_{h\ge -1}2^{-\delta\left(r-\frac{n}{\mu_\ast q}\right)h}\Vert d_h\Vert_{\mathcal FL^p_{r,M}}\,.
	\end{equation}
	Moreover for $0\le\delta<1$ and arbitrary $\varepsilon>0$, $T_3$ extends to a linear bounded operator
	\begin{equation}\label{T3_bdd_2}
	T_3:\mathcal FL^p_{\varepsilon+\delta r-(\delta-1)\frac{n}{\mu_\ast q},M}\rightarrow\mathcal FL^p_{r,M}
	\end{equation}
	and there exists a positive constant $C=C_{r,p,\varepsilon}$ such that:
	\begin{equation}\label{norma_T3_2}
	\Vert T_3\Vert_{\mathcal FL^p_{\varepsilon+\delta r-(\delta-1)\frac{n}{\mu_\ast q},M}\rightarrow\mathcal FL^p_{r,M}}\le C\sup\limits_{h\ge -1}2^{-\delta\left(r-\frac{n}{\mu_\ast q}\right)h}\Vert d_h\Vert_{\mathcal FL^p_{r,M}}\,.
	\end{equation}
\end{lemma}
\begin{proof}
Let us prove the first statement of Lemma \ref{T3_lemma}. For $N_0>0$ sufficiently large we have
\[
{\rm supp}\,\widehat{d_{h,\ell}u_h}\subseteq\mathcal C^{T}_\ell\,,\quad\mbox{for}\,\,\ell\ge N_0-1\,,\,\,-1\le h\le\ell-N_0\,.
\]
Hence Proposition \ref{prop_2} and Young's inequality imply, for finite $p\ge 1$,
\begin{equation}\label{T3_eq:1}
\begin{split}
\Vert T_3u\Vert_{\mathcal FL^p_M}^p&\le C\!\!\!\sum\limits_{\ell=N_0-1}^{+\infty}2^{s\ell p}\left\Vert\sum\limits_{h=-1}^{\ell-N_0}\widehat{d_{h,\ell}u_h}\right\Vert_{L^p}^p\!\!\!=C\!\!\!\sum\limits_{\ell=N_0-1}^{+\infty}2^{s\ell p}\left\Vert\sum\limits_{h=-1}^{\ell-N_0}\widehat{d_{h,\ell}}\ast\widehat{u_h}\right\Vert_{L^p}^p\\
&\le C\!\!\!\sum\limits_{\ell=N_0-1}^{+\infty}2^{s\ell p}\left(\sum\limits_{h=-1}^{\ell-N_0}\Vert\widehat{d_{h,\ell}}\Vert_{L^p}\Vert\widehat{u_h}\Vert_{L^1}\right)^p \\
&=C\!\!\!\sum\limits_{\ell=N_0-1}^{+\infty}\!\!\!\left(\sum\limits_{h=-1}^{\ell-N_0}2^{(s-r)\ell}2^{r\ell}\Vert\widehat{d_{h,\ell}}\Vert_{L^p}\Vert\widehat{u_h}\Vert_{L^1}\right)^p
\end{split}
\end{equation}
(with obvious modifications in the case of $p=+\infty$); on the other hand, condition (a) and Proposition \ref{prop_1} yield
\[
\sum\limits_{\ell=-1}^{+\infty}2^{r\ell p}\Vert\widehat{d}_{h,\ell}\Vert_{L^p}^p\le H 2^{\delta\left(r-\frac{n}{\mu_\ast q}\right) ph}\,,\quad\mbox{for}\,\,h\ge -1\,,
\]
hence
\begin{equation}\label{T3_eq:2}
2^{r\ell}\Vert\widehat{d}_{h,\ell}\Vert_{L^p}\le H 2^{\delta\left(r-\frac{n}{\mu_\ast q}\right)h}\,,\quad\mbox{for}\,\,\ell\ge -1\,,
\end{equation}
where $H$ is the constant in \eqref{H}.

Combining \eqref{T3_eq:1}, \eqref{T3_eq:2} and using Bernstein's inequality
\begin{equation}\label{B}
2^{-\frac{n}{\mu_\ast q}h}\Vert\widehat{u}_h\Vert_{L^1}\le C\Vert\widehat{u}_h\Vert_{L^p}
\end{equation}
we get
\begin{equation*}
\begin{split}
&\Vert T_3u\Vert_{\mathcal FL^p_M}\le C H^p\sum\limits_{\ell=N_0-1}^{+\infty}\left(\sum\limits_{h=-1}^{\ell-N_0}2^{(s-r)\ell}2^{\delta\left(r-\frac{n}{\mu_\ast q}\right)h}\Vert\widehat{u_h}\Vert_{L^1}\right)^p\\
&= C H^p\sum\limits_{\ell=N_0-1}^{+\infty}\left(\sum\limits_{h=-1}^{\ell-N_0}2^{(s-r)(\ell-h)}2^{(s-r)h}2^{\delta\left(r-\frac{n}{\mu_\ast q}\right)h}2^{\frac{n}{\mu_\ast q}h}2^{-\frac{n}{\mu_\ast q}h}\Vert\widehat{u_h}\Vert_{L^1}\right)^p\\
&= C H^p\sum\limits_{\ell=N_0-1}^{+\infty}\left(\sum\limits_{h=-1}^{\ell-N_0}2^{(s-r)(\ell-h)}2^{\left(s+(\delta-1)\left(r-\frac{n}{\mu_\ast q}\right)\right)h}2^{-\frac{n}{\mu_\ast q}h}\Vert\widehat{u_h}\Vert_{L^1}\right)^p\\
&\le C H^p\sum\limits_{\ell=N_0-1}^{+\infty}\left(\sum\limits_{h=-1}^{\ell-N_0}2^{(s-r)(\ell-h)}2^{\left(s+(\delta-1)\left(r-\frac{n}{\mu_\ast q}\right)\right)h}\Vert\widehat{u_h}\Vert_{L^p}\right)^p\,.
\end{split}
\end{equation*}
The last quantity above is the general term of the discrete convolution of the sequences
\[
b:=\{2^{(s-r)k}\}_{k\ge N_0-1}\,,\qquad c:=\{2^{\left(s+(\delta-1)\left(r-\frac{n}{\mu_\ast q}\right)\right)k}\Vert\widehat{u_k}\Vert_{L^p}\}_{k\ge N_0-1}\,.
\]
Since $b\in\ell^1$, for $s<r$, discrete Young's inequality and Proposition \ref{prop_1} yield
\begin{equation*}
\begin{split}
\Vert T_3u\Vert_{\mathcal FL^p_M}^p &\le CH^p\Vert b\Vert_{\ell^1}\Vert c\Vert_{\ell^p}\le \tilde CH^p\sum\limits_{\ell\ge -1}2^{\left(s+(\delta-1)\left(r-\frac{n}{\mu_\ast q}\right)\right)\ell p}\Vert\widehat{u_\ell}\Vert_{L^p}^p\\
&\le\hat CH^p\Vert u\Vert_{\mathcal FL^p_{s+(\delta-1)\left(r-\frac{n}{\mu_\ast q}\right),M}}\,.
\end{split}
\end{equation*}
This proves the first continuity property \eqref{T3_bdd_1} together with estimate \eqref{norma_T3_1}.

\medskip
Let us now prove the second statement of Lemma \ref{T3_lemma}, so we assume that $\delta\in[0,1[$. For an arbitrary $\varepsilon>0$ similar arguments to those used above give
\begin{equation}\label{T3_eq:3}
\begin{split}
\Vert T_3 u\Vert_{\mathcal FL^p_{r,M}}^p&\le C\sum\limits_{\ell=N_0-1}^{+\infty}2^{r\ell p}\left\Vert\sum\limits_{h=-1}^{\ell-N_0}\widehat{d_{h,\ell}}\ast\widehat{u_h}\right\Vert_{L^p}^p\\
&\le C\sum\limits_{\ell=N_0-1}^{+\infty}2^{r\ell p}\left(\sum\limits_{h=-1}^{\ell-N_0}\Vert\widehat{d_{h,\ell}}\Vert_{L^p}\Vert\widehat{u}_h\Vert_{L^1}\right)^p\\
&=C\sum\limits_{\ell=N_0-1}^{+\infty}\left(\sum\limits_{h=-1}^{\ell-N_0}2^{-\varepsilon h}2^{\varepsilon h}2^{r\ell}\Vert\widehat{d_{h,\ell}}\Vert_{L^p}\Vert\widehat{u}_h\Vert_{L^1}\right)^p\\
&\le C\sum\limits_{\ell=N_0-1}^{+\infty}\left(\sum\limits_{h=-1}^{\ell-N_0}2^{-\varepsilon h q}\right)^{p/q}\left(\sum\limits_{h=-1}^{\ell-N_0}2^{\varepsilon h p}2^{r\ell p}\Vert\widehat{d_{h,\ell}}\Vert_{L^p}^p\Vert\widehat{u}_h\Vert_{L^1}^p\right)\\
&\le C_{\varepsilon,p}\sum\limits_{\ell=N_0-1}^{+\infty}\sum\limits_{h=-1}^{\ell-N_0}2^{\varepsilon h p}2^{r\ell p}\Vert\widehat{d_{h,\ell}}\Vert_{L^p}^p\Vert\widehat{u}_h\Vert_{L^1}^p\\
&=C_{\varepsilon,p}\sum\limits_{h=-1}^{+\infty}2^{\varepsilon h p}\sum\limits_{\ell\ge h+N_0}2^{r\ell p}\Vert\widehat{d_{h,\ell}}\Vert_{L^p}^p\Vert\widehat{u}_h\Vert_{L^1}^p\,,
\end{split}
\end{equation}
where in the last quantity above the summation index order was interchanged.\\
Again from condition (a) and Proposition \ref{prop_1}
\[
\sum\limits_{\ell\ge h+N_0}2^{r\ell p}\Vert\widehat{d_{h,\ell}}\Vert_{L^p}^p\le C_{r,p}\Vert d_h\Vert_{\mathcal FL^p_{r,M}}\le C_{r,p}H^p2^{\delta\left(r-\frac{n}{\mu_\ast q}\right)hp}\,,
\]
with $H$ defined in \eqref{H}. Using the above to estimate the right-hand side of \eqref{T3_eq:3}, Bernstein's inequality \eqref{B} and Proposition \ref{prop_1} we obtain
\begin{equation*}
\begin{split}
\Vert T_3 &u\Vert_{\mathcal FL^p_{r,M}}^p\le C_{r,\varepsilon,p}H^p\sum\limits_{h=-1}^{+\infty}2^{\varepsilon h p}2^{\delta\left(r-\frac{n}{\mu_\ast q}\right)hp}\Vert\widehat{u}_h\Vert_{L^1}^p\\
&\le C_{r,\varepsilon,p}H^p\sum\limits_{h=-1}^{+\infty}2^{\left(\varepsilon+\delta r-(\delta-1)\frac{n}{\mu_\ast q}\right) h p}\Vert\widehat{u}_h\Vert_{L^p}^p\\
&\le C_{r,p,\varepsilon}H^p\Vert u\Vert_{\mathcal FL^p_{\varepsilon+\delta r-(\delta-1)\frac{n}{\mu_\ast q},M}}^p\,.
\end{split}
\end{equation*}
This completes the proof of the continuity \eqref{T3_bdd_2} together with estimate \eqref{norma_T3_2}.
\end{proof}
\begin{remark}\label{rmk:exponent}
{\rm Let us collect some observations concerning Lemma \ref{T3_lemma}. 
\newline
We first notice that for $s<r$ the boundedness of $T_3$ as a linear operator $T_3:\mathcal FL^p_{s,M}\rightarrow\mathcal FL^p_{s,M}$ follows as an immediate consequence of \eqref{T3_bdd_1}, since $\mathcal FL^p_{s,M}\hookrightarrow\mathcal FL^p_{s+(\delta-1)\left(r-\frac{n}{\mu_\ast q}\right),M}$ for $\delta$ and $r$ under the assumptions of Lemma \ref{T3_lemma}.
\newline
As regards to the second part of Lemma \ref{T3_lemma} (see \eqref{T3_bdd_2}), we notice that the Fourier-Lebesgue esponent $\varepsilon+\delta r-(\delta-1)\frac{n}{\mu_\ast q}$, with any positive $\varepsilon$, is a little more restrictive than the one that should be recovered from the exponent $s+(\delta-1)\left(r-\frac{n}{\mu_\ast q}\right)$, in the first part of the Lemma, in the limiting case as $s\to r$.
\newline
Notice eventually that when $0<\varepsilon<(1-\delta)\left(r-\frac{n}{\mu_\ast q}\right)$ is taken in the second part of the statement of Lemma \ref{T3_lemma}, then $\varepsilon+\delta r-(\delta-1)\frac{n}{\mu_\ast q}<r$. Hence we get the boundedness of $T_3$, as a linear operator $T_3:\mathcal FL^p_{r,M}\rightarrow\mathcal FL^p_{r,M}$, as long as $0\le\delta<1$, as an immediate consequence of the boundedness \eqref{T3_bdd_2}.}
\end{remark}
\section{Calculus for pseudodifferential operators with smooth symbols}\label{smooth_symb_sct}

In this section we investigate the properties of pseudodifferential operators with  $M$-homogeneous smooth  symbols introduced in Sect. \ref{hsy}, see Definition \ref{sm-sym_k_dfn}.\\
At first notice that, despite  $M-$ weight \eqref{M-h-w} is not smooth in $\mathbb R^n$, for an arbitrary vector $M=(\mu_1,\dots,\mu_n)\in\mathbb R^n_+$, one can always find an {\em equivalent} weight which is also a smooth symbol in the class $S^1_M$.  

More precisely, in view of \cite[Proposition 2.9]{GM-14}, the following proposition holds true.
\begin{proposition}\label{h-w-equiv_prop}
For any  vector $M=(\mu_1,\dots,\mu_n)\in\mathbb R^n_+$ there exists a symbol $\pi=\pi_M(\xi)\in S^1_M$, independent of $x$,which is equivalent to the $M-$ weight \eqref{M-h-w}, in the sense that a positive constant $C$ exists such that
\begin{equation}
\frac{1}{C}\pi_M(\xi)\le\langle\xi\rangle_M\le C\pi_M(\xi)\,,\quad\forall\,\xi\in\mathbb R^n\,.
\end{equation}
\end{proposition}
In view of the subsequent analysis, it is worth noticing that in the case when the vector $M$ has positive integer components, in Proposition \ref{h-w-equiv_prop} we can take $\pi_M(\xi)=\langle\xi\rangle_M$.

\subsection{Symbolic calculus in $S^m_{M,\delta,\kappa}$}\label{symb_calc_sct}
The symbolic calculus can be developed for classes $S^m_{M,\delta,\kappa}$, that is pseudodifferential operators with symbol in $S^m_{M,\delta,\kappa}$ constitute a self-contained sub-algebra of the algebra of operators with symbols in $S^m_{M,\delta}$, for $m\in\mathbb R$, $\kappa>0$ and $0\le\delta<\mu_\ast/\mu^\ast$. The main properties of symbolc calculus are summarized in the following result.
\begin{proposition}\label{symb_calc_prop}
\begin{itemize}
\item[(i)]For $m, m^\prime\in\mathbb R$, $\kappa>0$ and $\delta,\delta^\prime\in[0,1]$, let $a(x,\xi)\in S^{m}_{M,\delta,\kappa}$, $b(x,\xi)\in S^{m^\prime}_{M,\delta^\prime,\kappa}$, $\theta,\nu\in\mathbb Z^n_+$. Then
\begin{equation}\label{diff_prod_symb}
\partial^{\theta}_{\xi}\partial^\nu_x a(x,\xi)\in S^{m-\langle\theta,1/M\rangle+\delta\langle\nu,1/M\rangle}_{M,\delta,\kappa}\,,\quad (ab)(x,\xi)\in S^{m+m^\prime}_{M,\max\{\delta,\delta^\prime\},\kappa}\,.
\end{equation}
\item[(ii)]Let $\{m_j\}_{j=0}^{+\infty}$ be a sequence of real numbers satisfying:
\begin{equation}\label{decreasing}
m_j>m_{j+1}\,,\,\,\, j=0,1,\dots\quad\mbox{and}\quad\lim\limits_{j\to +\infty}m_j=-\infty
\end{equation}
and $\{a_j\}_{j=0}^{+\infty}$ be a sequence of symbols $a_j(x,\xi)\in S^{m_j}_{M,\delta,\kappa}$ for each integer $j\ge 0$. Then there exists a unique (up to a remainder in $S^{-\infty}$) symbol $a(x,\xi)\in S^{m_0}_{M,\delta,\kappa}$ such that
\begin{equation}\label{asympt_exp}
a-\sum\limits_{j<N}a_j\in S^{m_N}_{M,\delta,\kappa}\,,\quad\mbox{for all integers}\,\,\,N>0\,.
\end{equation}
\item[(iii)] Let $a(x,\xi)$ and $b(x,\xi)$ be two symbols as in (i), and assume that $0\le\delta^\prime<\mu_\ast/\mu^\ast$. Then the product $c(x,D):=a(x,D)b(x,D)$ is a pseudodifferential operator with symbol $c(x,\xi)=(a\sharp b)(x,\xi)\in S^{m+m^\prime}_{M,\delta^{\prime\prime},\kappa}$, where $\delta^{\prime\prime}:=\max\{\delta,\delta^\prime\}$; moreover this symbol satisfies
\begin{equation}\label{asympt_prod}
a\sharp b-\sum\limits_{\vert\alpha\vert<N}\frac{(-i)^{\vert\alpha\vert}}{\alpha!}\partial^\alpha_\xi a\,\partial^\alpha_x b\in S^{m+m^\prime-(1/\mu^\ast-\delta^\prime/\mu_\ast)N}_{M,\delta^{\prime\prime},\kappa}\,,\quad\mbox{for all integers}\,\,\,N>0\,.
\end{equation}
\end{itemize}
\end{proposition}
\begin{proof}
{\em (i)}:  From estimates \eqref{sm-sym_k_est_1}, \eqref{sm-sym_k_est_2}, it is very easy to check that for any multi-index $\theta\in\mathbb Z^n_+$
\[
a(x,\xi)\in S^m_{M,\delta,\kappa}\quad\mbox{implies}\quad\partial^\theta_\xi a(x,\xi)\in S^{m-\langle\theta,1/M\rangle}_{M,\delta,\kappa}\,;
\]
hence we can limit the proof of (i) to $\theta=0$ and an arbitrary $\nu\in\mathbb Z^n_+$, $\nu\neq 0$.
\newline
Let $\alpha\,,\beta\in\mathbb Z^n_+$ be arbitrary multi-indices and assume, for the first, that $\langle\beta,1/M\rangle\neq\kappa$; if $\langle\nu+\beta,1/M\rangle\neq\kappa$, we then get
\begin{equation}\label{(i)_eq:1}
\begin{split}
\vert\partial^\alpha_\xi\partial^\beta_x\left(\partial^\nu_x a\right)(x,\xi)\vert &\le C_{\nu,\alpha,\beta}\langle\xi\rangle_M^{m-\langle\alpha,1/M\rangle+\delta\left(\langle\nu+\beta,1/M\rangle-\kappa\right)_+}\\ &\le  C_{\nu,\alpha,\beta}\langle\xi\rangle_M^{m-\langle\alpha,1/M\rangle+\delta\langle\nu,1/M\rangle+\delta\left(\langle\beta,1/M\rangle-\kappa\right)_+}\,,
\end{split}
\end{equation}
in view of estimates \eqref{sm-sym_k_est_1} and the sub-additivity inequality
\begin{equation}\label{sub_pos}
(x+y)_+\le x_++y_+\,,\quad\forall\,x,y\in\mathbb R\,.
\end{equation}
Assume now that $\langle\nu+\beta,1/M\rangle=\kappa$; then
\begin{equation}\label{(i)_eq:2}
\begin{split}
\vert\partial^\alpha_\xi\partial^\beta_x\left(\partial^\nu_x a\right)(x,\xi)\vert &\le C_{\nu,\alpha,\beta}\langle\xi\rangle_M^{m-\langle\alpha,1/M\rangle}\log(1+\langle\xi\rangle_M^{\delta})\,,
\end{split}
\end{equation}
in view of \eqref{sm-sym_k_est_2}. Since
\[
\langle\nu+\beta,1/M\rangle=\kappa\quad\mbox{and}\quad\langle\beta,1/M\rangle\neq\kappa
\]
imply
\[
\langle\beta,1/M\rangle<\kappa\quad\mbox{and}\quad\langle\nu,1/M\rangle>0\,,
\]
then
\[
\log(1+\langle\xi\rangle_M^{\delta})\le C_{\nu,\beta}\langle\xi\rangle_M^{\delta\langle\nu,1/M\rangle}\equiv C_{\nu,\beta}\langle\xi\rangle_M^{\delta\langle\nu,1/M\rangle+\delta\left(\langle\beta,1/M\rangle-\kappa\right)_+}\,,
\]
which, combined with \eqref{(i)_eq:2}, leads again to \eqref{(i)_eq:1}.
\newline
Assume now that $\langle\beta,1/M\rangle=\kappa$. Since also $\langle\nu,1/M\rangle>0$, from \eqref{sm-sym_k_est_1} we get
\[
\begin{split}
\vert\partial^\alpha_\xi\partial^\beta_x\left(\partial^\nu_x a\right)(x,\xi)\vert&\le C_{\nu,\alpha,\beta}\langle\xi\rangle_M^{m-\langle\alpha,1/M\rangle+\delta\left(\langle\nu+\beta,1/M\rangle-\kappa\right)_+}\\
&\le C^\prime_{\nu,\alpha,\beta}\langle\xi\rangle_M^{m-\langle\alpha,1/M\rangle+\delta\langle\nu,1/M\rangle}\log(1+\langle\xi\rangle_M^{\delta})\,,
\end{split}
\]
because $\left(\langle\nu+\beta,1/M\rangle-\kappa\right)_+=\langle\nu+\beta,1/M\rangle-\kappa=\langle\nu,1/M\rangle$ and we also use the trivial inequality
\begin{equation}\label{low_log}
\log 2\le\log(1+\langle\xi\rangle_M^{\delta})\,,\quad\forall\,\xi\in\mathbb R^n\,.
\end{equation}
The preceding calculations show that $\partial^\nu_x a(x,\xi)\in S^{m+\delta\langle\nu,1/M\rangle}_{M,\delta,\kappa}$.
\newline
Similar trivial, while overloading, arguments can be used to prove the second statement of (i) concerning the product of symbols.

\smallskip
{\em (ii)} It is known from the symbolic calculus in classes $S^m_{M,\delta}$, cf. \cite[Proposition 2.3]{GM-09}, that for a sequence of symbols $\{a_j\}_{j=0}^{+\infty}$, obeying the assumptions made in (ii), there exists $a(x,\xi)\in S^{m_0}_{M,\delta}$, which is unique up to a remainder in $S^{-\infty}$, such that
\begin{equation}\label{asympt_exp_0}
a-\sum\limits_{j<N}a_j\in S^{m_N}_{M,\delta}\,,\quad\mbox{for all integers}\,\,\,N>0\,.
\end{equation}
It remains to check that $a(x,\xi)$ actually belongs to $S^{m_0}_{M,\delta,\kappa}$, namely its derivatives satisfy inequalities \eqref{sm-sym_k_est_1}, \eqref{sm-sym_k_est_2}. In view of \eqref{asympt_exp_0}, for any positive integer $N$, the symbol $a(x,\xi)$ can be represented in the form
\begin{equation}\label{dec}
a(x,\xi)=a_N(x,\xi)+R_N(x,\xi)\,,
\end{equation}
where $a_N:=\sum\limits_{j<N}a_j$ e $R_N\in S^{m_N}_{M,\delta}$.
\newline
Since $\lim\limits_{j\to +\infty}m_j=-\infty$, for all $\alpha,\beta\in\mathbb Z^n_+$ an integer $N_{\alpha,\beta}>0$ can be found such that
\begin{equation}\label{dis_ordini}
\begin{split}
& m_{N_{\alpha,\beta}}+\delta\langle\beta,1/M\rangle\le m_0+\delta(\langle\beta,1/M\rangle-\kappa)_+\,,\quad\mbox{if}\,\,\,\langle\beta,1/M\rangle\neq\kappa\,,\\
& m_{N_{\alpha,\beta}}+\delta\kappa\le m_0\,,\quad\mbox{if}\,\,\,\langle\beta,1/M\rangle=\kappa\,,
\end{split}
\end{equation}
hence let $a$ be represented in form \eqref{dec} with $N=N_{\alpha,\beta}$ (from the above inequalities $N_{\alpha,\beta}$ can be chosen independent of $\alpha$, as a matter of fact). Since $\{m_j\}$ is decreasing, from $a_j\in S^{m_j}_{M,\delta,\kappa}$ for every $j\ge 0$, we deduce at once that $a_{N_{\alpha,\beta}}\in S^{m_0}_{M,\delta,\kappa}$. As for the remainder $R_{N_{\alpha,\beta}}$, from $R_{N_{\alpha,\beta}}\in S^{m_{N_{\alpha,\beta}}}_{M,\delta}$, inequalities \eqref{dis_ordini} and \eqref{low_log}, we deduce
\[
\begin{split}
\vert\partial^\alpha_\xi\partial^\beta_x R_{N_{\alpha,\beta}}(x,\xi)\vert &\le C_{\alpha,\beta}\langle\xi\rangle^{m_{N_{\alpha,\beta}}-\langle\alpha,1/M\rangle+\delta\langle\beta,1/M\rangle}_M\\
&\le \begin{cases}C^\prime_{\alpha,\beta}\langle\xi\rangle^{m_0-\langle\alpha,1/M\rangle+\delta(\langle\beta,1/M\rangle-\kappa)_+}_M\,,\quad\mbox{if}\,\,\,\langle\beta,1/M\rangle\neq\kappa\,,\\
C^\prime_{\alpha,\beta}\langle\xi\rangle^{m_0-\langle\alpha,1/M\rangle}_M\log(1+\langle\xi\rangle_M^{\delta})\,,\quad\mbox{if}\,\,\,\langle\beta,1/M\rangle=\kappa\,.\end{cases}
\end{split}
\]
From \eqref{dec} with $N=N_{\alpha,\beta}$ and estimates above, we deduce
\begin{equation*}
\begin{split}
\vert\partial^\alpha_\xi &\partial^\beta_x a(x,\xi)\vert\le\vert\partial^\alpha_\xi\partial^\beta_x a_{N_{\alpha,\beta}}(x,\xi)\vert+\vert\partial^\alpha_\xi\partial^\beta_x R_{N_{\alpha,\beta}}(x,\xi)\vert\\
&\le\begin{cases}C^{\prime\prime}_{\alpha,\beta}\langle\xi\rangle^{m_0-\langle\alpha,1/M\rangle+\delta(\langle\beta\rangle-\kappa)_+}_M\,,\quad\mbox{if}\,\,\,\langle\beta,1/M\rangle\neq\kappa\,,\\
C^{\prime\prime}_{\alpha,\beta}\langle\xi\rangle^{m_0-\langle\alpha,1/M\rangle}_M\log(1+\langle\xi\rangle_M^{\delta})\,,\quad\mbox{if}\,\,\,\langle\beta,1/M\rangle=\kappa\end{cases}
\end{split}
\end{equation*}
and, because of the arbitatriness of $\alpha$ and $\beta$, this shows that $a\in S^{m_0}_{M,\delta,\kappa}$.

\smallskip
{\em (iii)} By still referring to the symbolic calculus in classes $S^m_{M,\delta}$, cf \cite[Proposition 2.5]{GM-09}, it is known that the product of two pseudodifferential operators $a(x,D)$ e $b(x,D)$ with symbols like in the statement (iii) is again a pseudodifferential operator $c(x,D)=a(x,D)b(x,D)$ with symbol $c(x,\xi)=(a\sharp b)(x,\xi)\in S^{m+m^\prime}_{M,\delta^{\prime\prime}}$, if $0\le\delta^\prime<\mu_\ast/\mu^\ast$; moreover, such a symbol enjoys the asymptotic expansion
\begin{equation}\label{sv_prodotto}
c(x,\xi)-\sum\limits_{\vert\alpha\vert<N}\frac{(-i)^{\vert\alpha\vert}}{\alpha !}\partial^\alpha_{\xi}a(x,\xi)\partial^\alpha_x b(x,\xi)\in S^{m+m^\prime-(1/\mu^\ast-\delta^\prime/\mu_\ast)N}_{M,\delta^{\prime\prime}}\,,\quad N\ge 1\,.
\end{equation}
To end up, it sufficient applying statements (i) and (ii) above to the sequence $\{c_k\}_{k=0}^{+\infty}$ of symbols
\[
c_k(x,\xi):=\sum\limits_{\vert\alpha\vert=k}\frac{(-i)^{k}}{\alpha !}\partial^\alpha_{\xi}a(x,\xi)\partial^\alpha_x b(x,\xi)\,,\quad k=0,1,\dots\,.
\]
From statement (i) it is immediately seen that $c_k(x,\xi)\in S^{m+m^\prime-(1/\mu^\ast-\delta^\prime/\mu_\ast)k}_{M,\delta^{\prime\prime},\kappa}$ for all integers $k\ge 0$. Since the sequence $\{m_k\}_{k=0}^{+\infty}$ of orders $m_k:=m+m^\prime-(1/\mu^\ast-\delta^\prime/\mu_\ast)k$ is decreasing, in view of $0\le\delta^\prime<\mu_\ast/\mu^\ast$, it follows from (ii) that a symbol $\tilde c(x,\xi)\in S^{m+m^\prime}_{M,\delta^{\prime\prime},\kappa}$ exists such that the same as the asymptotic formula \eqref{sv_prodotto} holds true with $\tilde c(x,\xi)$ instead of $c(x,\xi)$; moreover, from uniqueness of $c(x,\xi)$ (up to a symbol in $S^{-\infty}$), it also follows that $\tilde c(x,\xi)-c(x,\xi)\in S^{-\infty}$, hence the symbol $c(x,\xi)$ actually belongs to $S^{m+m^\prime}_{M,\delta^{\prime\prime},\kappa}$.
\end{proof}
\subsection{Parametrix of an elliptic operator with symbol in $S^m_{M,\delta,\kappa}$}\label{par_sct}
In order to perform the analysis of local and microlocal propagation of singularities of PDE on $M-$Fourier Lebesgue spaces, cf. Sect. \ref{propation_sing_FL_sct}, this section in devoted to the construction of the parametrix of a $M-$elliptic operator with symbol in $S^m_{M,\delta,\kappa}$.
\newline
We first recall the notion of $M-$elliptic symbol, we are going to deal with, see \cite{GM-08}, \cite{GM-09}.
\begin{definition}\label{M_elliptic_symb}
We say that $a(x,\xi)\in S^m_{M,\delta}$, or the related operator $a(x,D)$, is $M-$elliptic if there are constants $c_0>0$ and $R>1$ satisfying
\begin{equation}\label{M-ell}
\vert a(x,\xi)\vert\ge c_0\langle\xi\rangle_M^m\,,\quad\forall\,(x,\xi)\in\mathbb R^{2n}\,,\,\,\,\vert\xi\vert_M\ge R\,.
\end{equation}
\end{definition}
\begin{proposition}\label{M_elliptic_prop}
For $m\in\mathbb R$, $\kappa>0$ and $0\le\delta<\mu_\ast/\mu^\ast$, let the symbol $a(x,\xi)\in S^m_{M,\delta,\kappa}$ be $M-$elliptic. Then there exists $b(x,\xi)\in S^{-m}_{M,\delta,\kappa}$ such that $b(x,D)$ is a parametrix of the operator $a(x,D)$, i.e.
\begin{equation}\label{par_id}
b(x,D)a(x,D)=I+l(x,D)\,,\qquad a(x,D)b(x,D)=I+r(x,D)\,,
\end{equation}
where $I$ is the identity operator and $l(x,D)$, $r(x,D)$ are pseudodifferential operators with symbols $l(x,\xi), r(x,\xi)\in S^{-\infty}$.
\end{proposition}
\begin{proof}
The proof follows along the standard arguments employed in construcing the parametrix of an elliptic operator, see e.g. \cite{Folland}.
\newline
The first step consists to define a symbol $b_0(x,\xi)$ to be the inverse of $a(x,\xi)$, for sufficiently large $\xi$, that is
\begin{equation}\label{b0}
b_0(x,\xi):=\langle\xi\rangle^{-m}_MF\left(\langle\xi\rangle^{-m}_M a(x,\xi)\right)\,,
\end{equation}
with some function $F=F(z)\in C^\infty(\mathbb C)$ satisfying $F(z)=1/z$ for $\vert z\vert\ge c_0$ and where $c_0$ is the positive constant from \eqref{M-ell}.
From the symbolic calculus in the framework of $S^\infty_{M,\delta}$ (cf. \cite{GM-09}), it is easily shown that $b_0(x,\xi)\in S^{-m}_{M,\delta}$ and $\rho_1(x,\xi):=(a\sharp b_0)(x,\xi)-1\in S^{m-(1/\mu^\ast-\delta/\mu_\ast)}_{M,\delta}$, where, according to the notation introduced in Proposition \ref{symb_calc_prop}-(iii), $a\sharp b_0$ stands for the symbols of the product $a(x,D)b_0(x,D)$.
\newline
Then an operator $b(x,D)$ satisfying the second identity in \eqref{par_id} (that is a right-parametrix of $a(x,D)$) is defined as $b(x,D):=b_0(x,D)\rho(x,D)$ and where $\rho(x,D)$ is given by the Neumann-type series $\rho(x,D)=\sum\limits_{j=0}^{+\infty}\rho_1^j(x,D)$; more precisely, $\rho(x,D)$ is the pseudodifferential operator with symbol associated to the sequence of symbols $\rho_j(x,\xi)\in S^{-(1/\mu^\ast-\delta/\mu_\ast)j}_{M,\delta}$ recursively defined by
\begin{equation}\label{potenze_symb}
\rho_0:=1\quad\mbox{and}\quad\rho_j:=\rho_1\sharp\rho_{j-1}\,,\qquad\mbox{for}\,\,j=1,2,\dots\,.
\end{equation}
Since the sequence of orders $-(1/\mu^\ast-\delta/\mu_\ast)j$ tends to $-\infty$, once again in view of the symbolic calculus in $S^\infty_{M,\delta}$ (cf. \cite{GM-09}), a symbol $\rho(x,\xi)\in S^0_{M,\delta}$ such that
\begin{equation}\label{potenze_symb}
\rho-\sum\limits_{j<N}\rho_j\in S^{-(1/\mu^\ast-\delta/\mu_\ast)N}_{M,\delta}\,,\quad\mbox{for all integers}\,\,N\ge 1\,,
\end{equation}
is defined uniquely, up to symbols in $S^{-\infty}$.
\newline
One can finally show that $b(x,D)$, constructed as above, is a (two sided) parametrix of $a(x,D)$, see e.g. \cite[Ch. 4]{Folland} for more details.
\newline
In view of Proposition \ref{symb_calc_prop}, to end up it is sufficient to show that the symbol $b_0(x,\xi)\in S^{-m}_{M,\delta}$, defined in \eqref{b0}, actually belongs to $S^{-m}_{M,\delta,\kappa}$, that is its derivatives satisfy estimates \eqref{sm-sym_k_est_1}, \eqref{sm-sym_k_est_2}.
Since these estimates only require a more specific behavior of $x-$derivatives, compared to a generic symbol in $S^\infty_{M,\delta}$, we may reduce to check their validity for $x-$derivatives alone. Because $\langle\xi\rangle^{-m}_M a(x,\xi)\in S^0_{M,\delta,\kappa}$, we are going to only treat the case of a symbol $a(x,\xi)\in S^0_{M,\delta,\kappa}$.
\newline
For an arbitrary nonzero multi-index $\beta\neq 0$, from Fa\`a di Bruno's formula, we first recover
\begin{equation}\label{FdB}
\begin{split}
\vert\partial_x^\beta b_0(x,\xi)\vert &\le\sum\limits_{k=1}^{\vert\beta\vert}C_k\!\!\!\!\!\!\sum\limits_{\beta^1+\dots+\beta^k=\beta}\!\!\!\!\!\!\vert\partial^{\beta^1}_xa(x,\xi)\vert\dots\vert\partial^{\beta^k}_xa(x,\xi)\vert\,,
\end{split}
\end{equation}
where $C_k$ is a suitable positive constant depending only on $k\ge 0$ (notice that the function $F$ is bounded in $\mathbb C$ together with all its derivatives), and where, for each integer $k$ satisfying $1\le   k\le\vert\beta\vert$, the second sum in the right-hand side above is extended over all systems $\{\beta^1,\dots,\beta^k\}$ of nonzero multi-indices $\beta^j$ ($j=1,\dots,k$) such that $\beta^1+\dots+\beta^k=\beta$.
\newline
To apply estimates \eqref{sm-sym_k_est_1}, \eqref{sm-sym_k_est_2}, different cases must be considered separately.
\newline
Let us first assume that $\langle\beta,1/M\rangle\neq\kappa$. Since $a\in S^0_{M,\delta,\kappa}$, we have
\begin{equation}\label{stime_beta}
\vert\partial^{\beta^j}_xa(x,\xi)\vert\le C_j\langle\xi\rangle_M^{\delta(\langle\beta^j,1/M\rangle-\kappa)_+}\quad\mbox{or}\quad\vert\partial^{\beta^j}_xa(x,\xi)\vert\le C_j\log(1+\langle\xi\rangle_M^{\delta})\,,
\end{equation}
for all integers $1\le k\le\vert\beta\vert$ and $1\le j\le k$, according to whether $\langle\beta^j,1/M\rangle\neq\kappa$ or $\langle\beta^j,1/M\rangle=\kappa$, and suitable constants $C_j>0$.
\newline
If $\langle\beta,1/M\rangle<\kappa$ then $\langle\beta^j,1/M\rangle<\kappa$ for all $j=1,\dots,k$ and every  $1\le k\le\vert\beta\vert$, and
\[
\vert\partial_x^\beta b_0(x,\xi)\vert\le C_\beta\equiv C_\beta\langle\xi\rangle_M^{\delta(\left(\langle\beta,1/M\rangle-\kappa\right)_+}
\]
follows at once from \eqref{FdB} and \eqref{stime_beta}, with suitable $C_\beta>0$.
\newline
Assume now $\langle\beta,1/M\rangle>\kappa$, so that, for a given integer $1\le k\le\vert\beta\vert$ and an arbitrary system $\{\beta^1,\dots,\beta^k\}$ of multi-indices satisfying $\beta^1+\dots+\beta^k=\beta$, it could be either $\langle\beta^j,1/M\rangle\neq\kappa$ or $\langle\beta^j,1/M\rangle=\kappa$ for different indices $j=1,\dots,k$; up to a reordering of its elements, let $\{\beta^1,\dots,\beta^k\}$ be split into the sub-systems $\{\beta^1,\dots,\beta^{k^\prime}\}$ and $\{\beta^{k^\prime+1},\dots,\beta^{k}\}$ (for an integer $k^\prime$ with $1\le k^\prime<k$) such that $\langle\beta^{j},1/M\rangle\neq\kappa$ for all $1\le j\le k^\prime$ and $\langle\beta^\ell,1/M\rangle=\kappa$ for all $k^\prime+1\le\ell\le k$\footnote{Of course when $k=1$ then only $\langle\beta^1,1/M\rangle\equiv\langle\beta,1/M\rangle>\kappa$ can occur.}. In such a case, from \eqref{FdB} and \eqref{stime_beta} we get
\begin{equation}\label{FdB1}
\begin{split}
\vert\partial_x^\beta b_0(x,\xi)\vert\le \sum\limits_{k=1}^{\vert\beta\vert}C_k\!\!\!\!\!\!\sum\limits_{\beta^1+\dots+\beta^k=\beta}&
\langle\xi\rangle_M^{\delta\{(\langle\beta^1,1/M\rangle-\kappa)_++\dots+(\langle\beta^{k^\prime},1/M\rangle-\kappa)_+\}}\\
&\times\left(\log(1+\langle\xi\rangle_M^{\delta})\right)^{k-k^\prime}\,.
\end{split}
\end{equation}
Under the previous assumptions, it can be shown that
\begin{equation*}
(\langle\beta^1,1/M\rangle-\kappa)_++\dots+(\langle\beta^{k^\prime},1/M\rangle-\kappa)_+\le (\langle\beta^\prime,1/M\rangle-\kappa)_+\,,
\end{equation*}
where we have set $\beta^\prime:=\beta^1+\dots+\beta^{k^\prime}$.
Suppose $\langle\beta^\prime,1/M\rangle\le\kappa$ (thus $(\langle\beta^\prime,1/M\rangle-\kappa)_+=0$); since $\langle\beta,1/M\rangle>\kappa$, we have
\begin{equation}\label{FdB2}
\begin{split}
\langle &\xi\rangle_M^{\delta\{(\langle\beta^1,1/M\rangle-\kappa)_++\dots+(\langle\beta^{k^\prime},1/M\rangle-\kappa)_+\}}
\left(\log(1+\langle\xi\rangle_M^{\delta})\right)^{k-k^\prime}\\
&\le
\langle\xi\rangle^{\delta(\langle\beta^\prime,1/M\rangle-\kappa)_+}_M\left(\log(1+\langle\xi\rangle_M^{\delta})\right)^{k-k^\prime}\equiv \left(\log(1+\langle\xi\rangle_M^{\delta})\right)^{k-k^\prime}\\
&\le c_{\beta,k,k^\prime}\langle\xi\rangle^{\delta(\langle\beta,1/M\rangle-\kappa)}_M\equiv c_{\beta,k,k^\prime}\langle\xi\rangle^{\delta(\langle\beta,1/M\rangle-\kappa)_+}_M\,.
\end{split}
\end{equation}
Suppose now $\langle\beta^\prime,1/M\rangle>\kappa$ (hence $(\langle\beta^\prime,1/M\rangle-\kappa)_+=\langle\beta^\prime,1/M\rangle-\kappa$). Since $\langle\beta,1/M\rangle>\langle\beta^\prime,1/M\rangle$, we get
\begin{equation}\label{FdB3}
\begin{split}
\langle &\xi\rangle_M^{\delta\{(\langle\beta^1,1/M\rangle-\kappa)_++\dots+(\langle\beta^{k^\prime},1/M\rangle-\kappa)_+\}}
\left(\log(1+\langle\xi\rangle_M^{\delta})\right)^{k-k^\prime}\\
&\le\langle\xi\rangle^{\delta(\langle\beta^\prime,1/M\rangle-\kappa)_+}_M\left(\log(1+\langle\xi\rangle_M^{\delta})\right)^{k-k^\prime}\equiv \langle\xi\rangle^{\delta(\langle\beta^\prime,1/M\rangle-\kappa)}_M \left(\log(1+\langle\xi\rangle_M^{\delta})\right)^{k-k^\prime}\\
&\le c_{\beta,\beta^\prime,k,k^\prime}\langle\xi\rangle^{\delta\{(\langle\beta^\prime,1/M\rangle-\kappa)+(\langle\beta,1/M\rangle-\langle\beta^\prime,1/M\rangle)\}}_M= c_{\beta,\beta^\prime,k,k^\prime}\langle\xi\rangle_M^{\delta(\langle\beta,1/M\rangle-\kappa)}\\
&\equiv c_{\beta,\beta^\prime,k,k^\prime}\langle\xi\rangle_M^{\delta(\langle\beta,1/M\rangle-\kappa)_+} \,.
\end{split}
\end{equation}
In the boarder cases of a system $\{\beta^1,\dots,\beta^k\}$ where either $\langle\beta^j,1/M\rangle\neq\kappa$ for all $j$ or $\langle\beta^j,1/M\rangle=\kappa$ for all $j$\footnote{Notice that, under $\langle\beta,1/M\rangle>\kappa$, this second case can only occur when $k\ge 2$.}, all preceding arguments can be repeated, by formally taking $k^\prime=k$ in \eqref{FdB2} or $\beta^\prime=0$ and $k^\prime=0$ in \eqref{FdB3} respectively; thus we end up with the same estimates as above. Using \eqref{FdB2}, \eqref{FdB3} in the right-hand side of \eqref{FdB1} leads to
\begin{equation*}
\vert\partial_x^\beta b_0(x,\xi)\vert\le C_\beta\langle\xi\rangle_M^{\delta(\left(\langle\beta,1/M\rangle-\kappa\right)_+}\,.
\end{equation*}
\end{proof}

\subsection{Continuity of pseudodifferential operators with symbols in $S^m_{M,\delta,\kappa}$}\label{cont_smooth_symb_sct} 
Throughout the rest of this section, we assume that $M\in\mathbb R^n_+$ has all integer components. The Fourier-Lebesgue continuity of pseudodifferential operators with symbols in $S^m_{M,\delta,\kappa}$ is recovered as a consequence of Theorem \ref{FL_thm}.

\medskip
Taking advantage from growing estimates \eqref{sm-sym_k_est_1}, \eqref{sm-sym_k_est_2}, we first analyze the relations between smooth {\em local} symbols of type $S^m_{M,\delta,\kappa}$ and symbols of limited Fourier-Lebesgue smoothness introduced in Sect. \ref{FLcont_sct}  .
\begin{proposition}\label{FL_sym_prop}
For $M=(\mu_1,\dots,\mu_n)$ with strictly positive integer components, $m\in\mathbb R$, $\delta\in[0,1]$ and $\kappa>0$, let the symbol $a(x,\xi)\in S^m_{M,\delta,\kappa}$ satisfy the \textit{localization} condition \eqref{localization}
for some compact set $\mathcal K\subset\mathbb R^n$. The for all integers $N\ge 0$ and multi-indices $\alpha\in\mathbb Z^n_+$ there exists a postive constant $C_{\alpha,N,\mathcal K}$ such that:
\begin{eqnarray}
&\langle\eta\rangle_M^N\vert\partial_\xi^\alpha\hat{a}(\eta,\xi)\vert\le C_{\alpha,N,\mathcal K}\langle\xi\rangle_M^{m-\langle\alpha,1/M\rangle+\delta(N-\kappa)_+}\,,\quad\mbox{if}\,\,\,N\neq\kappa\,,\label{FL_sym_est_1}\\
&\langle\eta\rangle_M^N\vert\partial_\xi^\alpha\hat{a}(\eta,\xi)\vert\le C_{\alpha,N,\mathcal K}\langle\xi\rangle_M^{m-\langle\alpha,1/M\rangle}\log(1+\langle\xi\rangle_M^{\delta})\,,\quad\mbox{if}\,\,\,N=\kappa\,,\label{FL_sym_est_2}
\end{eqnarray}
where $\hat{a}(\eta,\xi)$ is the partial Fourier transform of $a(x,\xi)$ with respect to $x$:
\[
\hat{a}(\eta,\xi):=\widehat{a(\cdot,\xi)}(\eta)\,,\quad\forall\,(\eta,\xi)\in\mathbb R^{2n}\,.
\]
\end{proposition}
\begin{proof}
For an arbitrary integer $N\ge 0$ we estimate
\begin{equation}\label{M-h-w_est}
\langle\eta\rangle_M^N\le C_{N}\sum\limits_{\langle\beta,1/M\rangle\le N}\vert\eta^\beta\vert\,,\quad\forall\,\eta\in\mathbb R^n\,,
\end{equation}
with some positive constant $C_N>0$ (independent of $M$), hence for any $\alpha\in\mathbb Z^n_+$
\begin{equation*}
\begin{split}
\langle &\eta\rangle_M^N\vert\partial_\xi^\alpha\hat{a}(\eta,\xi)\vert\le C_{N}\!\!\!\sum\limits_{\langle\beta,1/M\rangle\le N}\vert\eta^\beta\partial_\xi^\alpha\hat{a}(\eta,\xi)\vert=C_{N}\!\!\!\sum\limits_{\langle\beta,1/M\rangle\le N}\vert\widehat{\partial_x^\beta\partial_\xi^\alpha a}(\eta,\xi)\vert\\
&=C_{N}\!\!\!\sum\limits_{\langle\beta,1/M\rangle\le N}\left\vert\int_{\mathcal K} e^{-i\eta\cdot x}\partial_x^\beta\partial_\xi^\alpha a(x,\xi)dx\right\vert\le C_{N}\!\!\!\sum\limits_{\langle\beta,1/M\rangle\le N}\int_{\mathcal K}\vert\partial_x^\beta\partial_\xi^\alpha a(x,\xi)\vert dx\,.
\end{split}
\end{equation*}
Thus, we end up by using estimates \eqref{sm-sym_k_est_1}, \eqref{sm-sym_k_est_2} under the integral sign above. 
\end{proof}
\begin{remark}\label{rmk:FL_est}
{\rm Notice that estimates \eqref{FL_sym_est_2} are satisfied only when $\kappa>0$ is an integer number.}
\end{remark}
As a consequence of Proposition \ref{FL_sym_prop} we get the proof of Theorem \ref{FL_sym_cor}

\begin{proof}[Proof of Theorem \ref{FL_sym_cor}.]
For $\kappa$ satisfying \eqref{kappa_ass}, consider the estimates \eqref{FL_sym_est_1}, \eqref{FL_sym_est_2} of $\hat{a}(\eta,\xi)$ with $N=N_\ast:=[n/\mu_\ast]+1$. For sure, estimates \eqref{FL_sym_est_2} cannot occur, since $N_\ast$ is smaller than $\kappa$, whereas estimates \eqref{FL_sym_est_1} reduce to
\begin{equation}\label{FL1_est_1}
\vert\partial_\xi^\alpha\hat{a}(\eta,\xi)\vert\le C_{\alpha,N_\ast,\mathcal K}\langle\xi\rangle_M^{m-\langle\alpha,1/M\rangle}\langle\eta\rangle_M^{-N_\ast}\,,\quad\forall\,(\eta,\xi)\in\mathbb R^n\,.
\end{equation}
On the other hand, the left inequality in \eqref{PG} yields
\begin{equation*}
\langle\eta\rangle_M^{-N_\ast}\le C\langle\eta\rangle^{-\mu_\ast N_\ast}\,,\quad\forall\,\eta\in\mathbb R^n\,,
\end{equation*}
from  which, $\langle\cdot\rangle_M^{-N_\ast}\in L^1(\mathbb R^n)$ follows, since $\mu_\ast N_\ast>n$. Then integrating in $\mathbb R^n_\eta$ both sides of \eqref{FL1_est_1} leads to
\begin{equation}\label{FL1_est}
\Vert\partial_\xi^\alpha a(\cdot,\xi)\Vert_{\mathcal FL^1}\le \tilde C_{\alpha,N,\mathcal K}\langle\xi\rangle_M^{m-\langle\alpha,1/M\rangle}\,,\quad\forall\,\xi\in\mathbb R^n\,,
\end{equation}
which are just estimates \eqref{FL_est}.

\smallskip
For an arbitrary integer $r>0$, we consider again estimates \eqref{FL_sym_est_1}, \eqref{FL_sym_est_2} of $\hat{a}(\eta,\xi)$ with $N=N_r:=r+[n/\mu_\ast]+1$. Notice that from \eqref{kappa_ass}
\[
N_r-\kappa<N_r-[n/\mu_\ast]-1=r\,,\quad\mbox{hence}\quad (N_r-\kappa)_+\le r_+=r\,.
\]
Then \eqref{FL_sym_est_1}, \eqref{FL_sym_est_2} lead to
\begin{eqnarray*}
&\langle\eta\rangle_M^r\vert\partial_\xi^\alpha\hat{a}(\eta,\xi)\vert\le C_{\alpha,N_r,{\mathcal K}}\langle\eta\rangle_M^{-N_\ast}\langle\xi\rangle_M^{m-\langle\alpha,1/M\rangle+\delta r}\,,\quad\mbox{if}\,\,\,N_r\neq\kappa\,,\\
&\langle\eta\rangle_M^r\vert\partial_\xi^\alpha\hat{a}(\eta,\xi)\vert\le C_{\alpha,N_r,{\mathcal K}}\langle\eta\rangle_M^{-N_\ast}\langle\xi\rangle_M^{m-\langle\alpha,1/M\rangle}\log(1+\langle\xi\rangle_M^{\delta})\,,\quad\mbox{otherwise}\,,
\end{eqnarray*}
where $N_\ast=\left[n/\mu_\ast\right]+1$ as before. Then using the trivial estimate
\begin{equation}\label{triv_est}
\log(1+\langle\xi\rangle_M^{\delta})\le C_r\langle\xi\rangle_M^{\delta r}\,,\quad\forall\,\xi\in\mathbb R^n
\end{equation}
and integrating in $\mathbb R^n_\eta$ both sides of inequalities above gives
\begin{equation}\label{FLrM_est}
\Vert\partial_\xi^\alpha a(\eta,\xi)\Vert_{\mathcal FL^1_{r,M}}\le C_{\alpha,N_r,{\mathcal K}}\langle\xi\rangle_M^{m-\langle\alpha,1/M\rangle+\delta r}\,,\quad\forall\,\xi\in\mathbb R^n\,,
\end{equation}
which are nothing else estimates \eqref{X_est} with $p=1$ (so $q=+\infty$). Together with \eqref{FL1_est}, estimates above tell us that $a(x,\xi)\in\mathcal FL^1_{r,M}S^m_{M,\delta}(N)$, for all integer numbers $r>0$ and $N>0$ arbitrarily large.

\smallskip
Then applying to $a(x,\xi)$ the result of Theorem \ref{FL_thm} with $p=1$ and an arbitrary integer $r>0$ shows that $a(x,D)$ fulfils the boundedness in \eqref{FL_bdd} with $p=1$.

\medskip
Now we are going to prove that the same symbol $a(x,\xi)$ also belongs to the class $\mathcal FL^\infty_{r,M}(N)$ with an arbitrary integer number $r>n/\mu_\ast$ and $N>0$ arbitrarily large, so as to apply again Theorem \ref{FL_thm} to $a(x,D)$ with $p=+\infty$. To do so, it is enough considering once again estimates \eqref{FL_sym_est_1} for $\hat{a}(\eta,\xi)$ with an arbitrary integer $N\equiv r>\kappa$; noticing that, under the assumption \eqref{kappa_ass},
\[
r-\kappa<r-n/\mu_\ast\,,\quad\mbox{hence}\quad (r-\kappa)_+\le (r-n/\mu_\ast)_+=r-n/\mu_\ast\,,
\]
estimates \eqref{FL_sym_est_1} just reduce to
\begin{equation}\label{FLinfM}
\Vert\partial_\xi^\alpha a(\cdot,\xi)\Vert_{\mathcal FL^\infty_{r,M}}\le C_{\alpha,r,{\mathcal K}}\langle\xi\rangle_M^{m-\langle\alpha,1/M\rangle+\delta(r-n/\mu_\ast)}\,,\quad\forall\,\xi\in\mathbb R^n\,,
\end{equation}
which are exactly estimates \eqref{X_est} with $p=+\infty$ (the number of $\xi-$derivatives which these estimates apply to can be chosen here arbitrarily large). So, as announced before, Theorem \ref{FL_thm} can be applied to make the conclusion that the boundedness property \eqref{FL_bdd} holds true for $a(x,D)$ with $p=+\infty$ and an arbitrary integer $r>\kappa$, and this shows that $a(x,D)$ also exhibits the boundedness in \eqref{FL_bdd} with $p=+\infty$.

\medskip
To recover \eqref{FL_bdd} with an arbitrary summability exponent $1<p<+\infty$ it is then enough to argue by complex interpolation through Riesz-Thorin's Theorem.
\end{proof}
\begin{remark}
{\rm Let us remark that assumption \eqref{localization} on the $x$ support of the symbol $a(x,\xi)$ amounts to say that the continuous prolongement of $a(x,D)$ on $\mathcal FL^p_{s+m,M}$ takes values in $\mathcal FL^p_{s,M}$ only {\em locally}, see the next Definition \ref{loc_mcl_FL}.}
\end{remark}
\section{Decomposition of $M-$Fourier Lebesgue symbols}\label{propation_sing_FL_sct_0}
As in the preceding Sect. \ref{smooth_symb_sct}, we will assume later on that vector $M=(\mu_1,\dots,\mu_n)$ has strictly positive integer components. 
\newline
For $m, r\in\mathbb R$, $p\in[1,+\infty]$, $\delta\in[0,1]$, we set
\[
\mathcal FL^p_{r,M}S^m_{M,\delta}:=\bigcap\limits_{N=1}^{\infty}\mathcal FL^p_{r,M}S^m_{M,\delta}(N)
\]
and $\mathcal FL^p_{r,M}S^m_{M}:=\mathcal FL^p_{r,M}S^m_{M,0}$.
In order to develop a regularity theory of $M-$elliptic linear PDEs with $M-$homogeneous Fourier Lebesgue coefficients, in the absence of a symbolic calculus for pseudodifferential operators with Fourier Lebesgue symbols (in particular the lack of a parametrix of an $M-$elliptic operator with non smooth coefficients), following the approach of Taylor \cite[\S 1.3]{MT-NLPDE}, we introduce here a decomposition of a $M-$Fourier Lebesgue symbol $a(x,\xi)\in\mathcal FL^p_{r,M}S^m_{M}$ as the sum of two terms: one is a $M-$homogeneous smooth symbol in $S^m_{M,\delta}$ and the other is still a Fourier Lebesgue symbol of lower order, decreased from $m$ by a positive quantity proportional to $\delta$, where $0<\delta<1$ is given, while arbitrary.
\newline
Such a decomposition is made by applying to the symbol $a(x,\xi)$ a suitable ``cut-off'' Fourier multiplier, ``splitting in the frequency space the (nonsmooth) coefficients of $a(x,\xi)$ as a sum of two contributions''.

\smallskip
Let us first consider a $C^\infty-$function $\phi$ such that $\phi(\xi)=1$
for $\langle\xi\rangle_M\le 1$ and $\phi(\xi)=0$ for
$\langle\xi\rangle_M> 2$. With a given $\varepsilon>0$, we set
 $\phi(\varepsilon^{\frac1{M}}\xi):=\phi(\varepsilon^{\frac1{m_1}}\xi_1,\dots,\varepsilon^{\frac1{m_n}}\xi_n)$ and let $\phi(\varepsilon^{\frac1{M}}D)$ denote the associated Fourier multiplier.
\newline
The following $M-$homogeneous version of \cite[Lemma 1.3.A]{MT-NLPDE}, shows the behavior of $\phi(\varepsilon^{\frac1{M}}D)$ on $M-$homogeneous Fourier Lebesgue spaces.
\begin{lemma}\label{le:2}
Let $p\in[1,+\infty]$ and $\varepsilon>0$ be arbitrarily fixed.
\begin{itemize}
\item[(i)] For every $\beta\in\mathbb Z^n_+$ and $r\in\mathbb R$, the Fourier multiplier $D^\beta\phi(\varepsilon^{\frac1{M}}D)$ extends as a bounded linear operator $D^\beta\phi(\varepsilon^{\frac1{M}}D):\mathcal FL^p_{r,M}\rightarrow\mathcal FL^p_{r,M}$ and there is a positive constant $C_\beta$, independent of $\varepsilon$, such that for all $u\in\mathcal FL^p_{r,M}$:
\begin{equation}
\Vert
D^{\beta}\phi(\varepsilon^{\frac1{M}}D)u\Vert_{\mathcal FL^p_{r,M}}\le C_{\beta}\varepsilon^{-\langle\beta,\frac1{M}\rangle}\Vert
u\Vert_{\mathcal FL^p_{r,M}}\,,\quad\forall\,u\in\mathcal FL^p_{r,M}\,;\label{ineq:1}
\end{equation}
\item[(ii)] For all $r\in\mathbb R$ and $t\ge 0$, the Fourier multiplier $I-\phi(\varepsilon^{\frac1{M}}D)$ (where $I$ denotes the identity operator) extends as a bounded linear operator $I-\phi(\varepsilon^{\frac1{M}}D):\mathcal FL^p_{r,M}\rightarrow\mathcal FL^p_{r-t,M}$ and there exists a constant $C_t>0$, independent of $\varepsilon$, such that:
\begin{equation}
\Vert u-\phi(\varepsilon^{\frac1{M}}D)u\Vert_{\mathcal FL^p_{r-t,M}}\le
C_t\varepsilon^t\Vert u\Vert_{\mathcal FL^p_{r,M}}\,,\quad\forall\,u\in\mathcal FL^p_{r,M}\,;\label{ineq:2}
\end{equation}
\item[(iii)] If $r>\frac{n}{\mu_\ast q}$, where $\frac1{p}+\frac1{q}=1$, and $\beta\in\mathbb Z^n_+$, then $D^\beta\phi(\varepsilon^{1/M}D)$ and $I-\phi(\varepsilon^{\frac1{M}}D)$ extend as bounded linear operators $D^\beta\phi(\varepsilon^{1/M}D),\, I-\phi(\varepsilon^{\frac1{M}}D):\mathcal FL^p_{r,M}\rightarrow\mathcal FL^1$ and there are constants $C_{r,\beta}$ and $C_r$, independent of $\varepsilon$, such that:
\begin{equation}\label{ineq:3}
\begin{array}{ll}
\Vert D^\beta\phi(\varepsilon^{1/M}D)u\Vert_{\mathcal FL^1}&\le C_{r,\beta}\varepsilon^{-\left(\langle\beta,1/M\rangle-(r-\frac{n}{\mu_\ast q})\right)_+}\Vert u\Vert_{\mathcal FL^p_{M,r}},\\
&\mbox{if}\,\,\langle\beta,1/M\rangle\neq r-\frac{n}{\mu_\ast q}\,,\\
\\
\Vert D^\beta\phi(\varepsilon^{1/M}D)u\Vert_{\mathcal FL^1}&\le C_{r}\log^{1/q}(1+\varepsilon^{-1})\Vert u\Vert_{\mathcal FL^p_{M,r}}\,,\\
&\mbox{if}\,\,\,\langle\beta,1/M\rangle= r-\frac{n}{\mu_\ast q}\,,\\
\\
\Vert u-\phi(\varepsilon^{\frac1{M}}D)u\Vert_{\mathcal FL^{1}}&\le
C_r\varepsilon^{r-\frac{n}{\mu_\ast q}}\Vert u\Vert_{\mathcal FL^p_{r,M}}\,,\quad\forall\,u\in\mathcal FL^p_{r,M}\,.
\end{array}
\end{equation}
\end{itemize}
\end{lemma}
\begin{proof}
{\rm (i)}: From the properties of function $\phi$, one can readily show that for any $\beta\in\mathbb Z^n_+$ there exists a constant $C_\beta>0$ such that:
\[
\vert\xi^{\beta}\phi(\varepsilon^{\frac1{M}}\xi)\vert\le C_\beta\varepsilon^{-\langle\beta,1/M\rangle}\,,\quad\forall\,\xi\in\mathbb R^n\,,\,\,\,\forall\,\varepsilon\in]0,1]\,.
\]
Then estimate \eqref{ineq:1} follows at once from H\"older's inequality.

\smallskip
{\rm (ii)}: Similarly as for {\rm (i)}, for $t\ge 0$, one can find a positive constant $C_t$ such that:
\[
\big\vert\langle\xi\rangle_M^{-t}(1-\phi(\varepsilon^{1/M}\xi))\big\vert\le C_t\varepsilon^t\,,\quad\forall\,\xi\in\mathbb R^n\,,\,\,\,\forall\,\varepsilon\in]0,1]\,,
\]
then estimate \eqref{ineq:2} follows once again from H\"older's inequality.

\smallskip
{\rm (iii)}: The extension of $D^\beta\phi(\varepsilon^{\frac1{M}}D)$ and $I-\phi(\varepsilon^{\frac1{M}}D)$ as linear bounded operators from $\mathcal FL^p_{r,M}$ to $\mathcal FL^1$ follows at once from a combination of the continuity properties stated in (i), (ii) and the fact that the space $\mathcal FL^p_{r,M}$ is imbedded into $\mathcal FL^1$ when $r>\frac{n}{\mu_\ast q}$.

\smallskip
For $\langle\beta,1/M\rangle<r-\frac{n}{\mu_\ast q}$, we directly have
\[
\Vert D^\beta\phi(\varepsilon^{1/M}D)u\Vert_{\mathcal FL^1}=\int\vert\xi^\beta\vert\phi(\varepsilon^{1/M}\xi)\vert\widehat{u}(\xi)\vert d\xi\,,
\]
and $0\le\phi\le 1$ implies $\vert\xi^\beta\vert\phi(\varepsilon^{1/M}\xi)\le\langle\xi\rangle_M^{\langle\beta,1/M\rangle}$. Combining the above and since $\langle\cdot\rangle^{\langle\beta,1/M\rangle-r}_M\in L^q$ as $r-\langle\beta,1/M\rangle>\frac{n}{\mu_\ast q}$, H\"older's inequality yields
\[
\Vert D^\beta\phi(\varepsilon^{1/M}D)u\Vert_{\mathcal FL^1}\le C_{r,\beta,p}\Vert u\Vert_{\mathcal FL^p_{r,M}}\,,
\]
where $C_{r,\beta,p}:=\left(\int\frac1{\langle\xi\rangle_M^{(r-\langle\beta,1/M\rangle)q}}d\xi\right)^{1/q}$. The above is exactly $\eqref{ineq:3}_1$ for $\langle\beta,1/M\rangle <r-\frac{n}{\mu_\ast q}$.
\newline
For $\langle\beta,1/M\rangle>r-\frac{n}{\mu_\ast q}$, we first write
\begin{equation}\label{ineq:3_1}
\Vert D^\beta\phi(\varepsilon^{1/M}D)u\Vert_{\mathcal FL^1}=\left\Vert\xi^\beta\!\!\sum\limits_{h=-1}^{+\infty}\phi(\varepsilon^{1/M}\xi)\widehat u_h\right\Vert_{L^1}\,,
\end{equation}
where, for every integer $h\ge -1$, we set $\widehat u_h=\varphi_h\widehat u$, being $\left\{\varphi_h\right\}_{h=-1}^{\infty}$ the dyadic partition of unity introduced in Sect. \ref{FL-sp_sct}.
\newline
Since $\phi(\varepsilon^{1/M}\xi)\widehat{u}_h\equiv 0$, as long as the integer $h\ge 0$ satisfies $2\varepsilon^{-1}<\frac1{K}2^{h-1}$ (that is $h>\log_2(4K/\varepsilon)$), cf. \eqref{dyad_cov}, \eqref{cond_phi}, from \eqref{ineq:3_1}, $0\le\phi\le 1$,
\begin{equation}\label{stima_xi_beta}
\vert\xi^\beta\vert\le\vert\xi\vert_M^{\langle\beta,1/M\rangle}\le C_{K,\beta}2^{h\langle\beta,1/M\rangle}\,,\quad\mbox{for}\,\,\xi\in\mathcal C^{M,K}_h\,,
\end{equation}
with a constant $C_{K,\beta}>0$ independent of $h$, and H\"older's inequality, it follows
\begin{equation}\label{ineq:3_2}
\begin{split}
\Vert D^\beta &\phi(\varepsilon^{1/M}D)u\Vert_{\mathcal FL^1}\le\sum\limits_{h=-1}^{\left[\log_2(4K/\varepsilon)\right]}\int\limits_{\mathcal C^{M,k}_h}\vert\xi^\beta\vert\phi(\varepsilon^{1/M}\xi)\vert\widehat u_h(\xi)\vert d\xi\\
&\le C_{K,\beta}\sum\limits_{h=-1}^{\left[\log_2(4K/\varepsilon)\right]}2^{h\langle\beta,1/M\rangle}\int\limits_{\mathcal C^{M,k}_h}\vert\widehat u_h(\xi)\vert d\xi\\
&=C_{K,\beta}\sum\limits_{h=-1}^{\left[\log_2(4K/\varepsilon)\right]}2^{h\sigma}\int\limits_{\mathcal C^{M,k}_h}2^{-h\frac{n}{\mu_\ast q}}2^{hr}\vert\widehat u_h(\xi)\vert d\xi\\
&\le C_{K,\beta}\sum\limits_{h=-1}^{\left[\log_2(4K/\varepsilon)\right]}2^{h\sigma}\left(\int\limits_{\mathcal C^{M,k}_h}2^{-h\frac{n}{\mu_\ast}}\right)^{1/q}\Vert 2^{hr}\widehat{u}_h\Vert_{L^p}\\
&\le C_{K,\beta,n,p}\sum\limits_{h=-1}^{\left[\log_2(4K/\varepsilon)\right]}2^{h\sigma}\Vert 2^{hr}\widehat{u}_h\Vert_{L^p}\,,
\end{split}
\end{equation}
where we used $\int\limits_{\mathcal C^{M,k}_h}d\xi\le C_{\ast,K,n}2^{h\frac{n}{\mu_\ast}}$, for a constant $C_{\ast,K,n}$ independent of $h$, and it is set $C_{K,\beta,n,p}:=C_{K,\beta}C_{\ast,K,n}^{1/q}$ and $\sigma:=\langle\beta,1/M\rangle-(r-\frac{n}{\mu_\ast q})$. Hence, we use discrete H\"older's inequality with conjugate exponents $(p,q)$ and the characterization of $M-$homogeneous Fourier Lebesgue spaces provided by Proposition \ref{prop_1} to end up with
\begin{equation}\label{ineq:3_3}
\begin{split}
\Vert D^\beta &\phi(\varepsilon^{1/M}D)u\Vert_{\mathcal FL^1}\le C_{K,\beta,n,p}\left(\sum\limits_{h=-1}^{\left[\log_2(4K/\varepsilon)\right]}2^{h\sigma q}\right)^{1/q}\Vert u\Vert_{\mathcal FL^p_{r,M}}\,,
\end{split}
\end{equation}
and
\begin{equation}\label{ineq:3_4}
\begin{array}{ll}
\sum\limits_{h=-1}^{\left[\log_2(4K/\varepsilon)\right]}2^{h\sigma q}&=2^{\sigma q\left(\left[\log_2(4K/\varepsilon)\right]\right)}\sum\limits_{h=-1}^{\left[\log_2(4K/\varepsilon)\right]}2^{-\sigma q(\left[\log_2(4K/\varepsilon)\right]-h)}\\
&\le (4K/\varepsilon)^{\sigma q}C_{\sigma,q}=C_{K,\sigma,q}\varepsilon^{-\sigma q}\,,
\end{array}
\end{equation}
where $C_{\sigma,q}:=\sum\limits_{j\ge 0}2^{-\sigma qj}$ is convergent, as $\sigma>0$, and $C_{K,\sigma,q}:=4KC_{\sigma,q}$ is independent of $\varepsilon$. Inequality $\eqref{ineq:3}_1$ for $\langle\beta,1/M\rangle>r-\frac{n}{\mu_\ast q}$ follows from combining \eqref{ineq:3_3}, \eqref{ineq:3_4}.

\smallskip
To prove $\eqref{ineq:3}_2$, we repeat the arguments leading to \eqref{ineq:3_1}--\eqref{ineq:3_3} where $\langle\beta,1/M\rangle=r-\frac{n}{\mu_\ast q}$ (that is $\sigma=0$), use discrete H\"older's inequality and Proposition \ref{prop_1}, to get:
\begin{equation}\label{ineq:3_5}
\begin{split}
\Vert D^\beta \phi(\varepsilon^{1/M}D)u\Vert_{\mathcal FL^1}&\le C_{K,r,n,p}\!\!\!\sum\limits_{h=-1}^{\left[\log_2(4K/\varepsilon)\right]}\Vert 2^{hr}\widehat{u}_h
\Vert_{L^p}\\
&\le\tilde C_{K,r,n,p}\left(\sum\limits_{h=-1}^{\left[\log_2(4K/\varepsilon)\right]}1\right)^{1/q}\!\!\!\!\!\!\Vert u\Vert_{\mathcal FL^p_{r,M}}\\
&=\tilde C_{K,r,n,p}\left(2+\left[\log_2(4K/\varepsilon)\right]\right)^{1/q}\Vert u\Vert_{\mathcal FL^p_{r,M}}\\
&\le C^\prime_{K,r,n,p}\log^{1/q}(1+\varepsilon^{-1})\Vert u\Vert_{\mathcal FL^p_{r,M}}\,.
\end{split}
\end{equation}

\smallskip
The proof of inequality $\eqref{ineq:3}_3$ follows along the same arguments used above. We resort once again to Proposition \ref{prop_1} and H\"older's inequality to get
\begin{equation*}
\begin{split}
\Vert(I-\phi(\varepsilon^{1/M}D))u\Vert_{\mathcal FL^1}&=\Vert(1-\phi(\varepsilon^{1/M}\cdot))\widehat u\Vert_{L^1}=\left\Vert(1-\phi(\varepsilon^{1/M}\cdot))\sum\limits_{h=-1}^{\infty}\widehat u_h\right\Vert_{L^1}\\
&\le\sum\limits_{h>\log_2\left(\frac{1}{2K\varepsilon}\right)}\Vert(1-\phi(\varepsilon^{1/M}\cdot))\widehat u_h\Vert_{L^1}\\
&\le\sum\limits_{h>\log_2\left(\frac{1}{2K\varepsilon}\right)}\left\Vert\frac{(1-\phi(\varepsilon^{1/M}\cdot))\chi_h}{\langle\cdot\rangle_M^r}\right\Vert_{L^q}\Vert\langle\cdot\rangle_M^r\widehat u_h\Vert_{L^p}
\end{split}
\end{equation*}
where for an integer $h\ge -1$, $\chi_h$ is the characteristic function of $\mathcal C^{M,K}_h$ and we use  $(1-\phi(\varepsilon^{1/M}\cdot))\varphi_h\equiv 0$ for $K2^{h+1}\le 1/\varepsilon$, cf. \eqref{dyad_cov}, \eqref{cond_phi}. Arguing as in the proof of Proposition \ref{prop_1} yields
\[
\Vert\langle\cdot\rangle_M^r\widehat u_h\Vert_{L^p}\le C_{r,p}2^{rh}\Vert\widehat u_h\Vert_{L^p}\,,\quad\forall\,h\ge -1\,,
\]
with positive constant $C_{r,p}$ depending only on $r$ and $p$. Using again the properties of functions $\phi$ and $\varphi_h$'s, we also get, for any $h\geq -1$, 
\[
\begin{split}
\left\Vert\frac{(1-\phi(\varepsilon^{1/M}\cdot))\chi_h}{\langle\cdot\rangle_M^r}\right\Vert_{L^q}^q &=\int_{\mathcal C^{M,K}_h}\left\vert\frac{(1-\phi(\varepsilon^{1/M}\xi))}{\langle\xi\rangle_M^r}\right\vert^q d\xi\le \int_{\mathcal C^{M,K}_h}\frac{1}{\langle\xi\rangle_M^{rq}}d\xi\\
&\le C_{r,q}2^{-rhq}\int_{\mathcal C^{M,K}_h}d\xi\le C_{r,p,\mu_\ast,K,n}2^{h(-rq+n/\mu_\ast)}
\end{split}
\]
(with obvious modifications in the case of $q=\infty$, that is $p=1$); here and later on, $C_{r,p, \mu_\ast, K, n}$ will denote some positive constant, depending only on $r$, $p$, $\mu_\ast$, $K$ and the dimension $n$, that may be different from an occurrence to another.

\smallskip
Using the above inequalities in the previous estimate of the $L^1-$norm of $(1-\phi(\varepsilon^{1/M}\cdot))\widehat u$, together with H\"older's inequality and Proposition \ref{prop_1}, we end up with
\[
\begin{split}
\Vert(I- &\phi(\varepsilon^{1/M}D))u\Vert_{\mathcal FL^1}\le C_{r,p,\mu_\ast,K,n}\sum\limits_{h>\log_2\left(\frac{1}{2K\varepsilon}\right)}2^{h(-r+\frac{n}{\mu_\ast q})}\,2^{rh}\Vert\widehat u_h\Vert_{L^p}\\
&\le C_{r,p,\mu_\ast,K,n}\left(\sum\limits_{h>\log_2\left(\frac{1}{2K\varepsilon}\right)}2^{h(-rq+n/\mu_\ast)}\right)^{1/q}\left(\sum\limits_{h\ge -1}2^{rhp}\Vert\widehat u_h\Vert_{L^p}^p\right)^{1/p}\\
&\le C_{r,p,\mu_\ast,K,n}\left(\frac{1}{2K\varepsilon}\right)^{-r+\frac{n}{\mu_\ast q}}\left(\sum\limits_{\ell>0}2^{\ell(-rq+n/\mu_\ast)}\right)^{1/q}\Vert u\Vert_{\mathcal FL^p_{M,r}}\\
&\le C_{r,p,\mu_\ast,K,n}\varepsilon^{r-\frac{n}{\mu_\ast q}}\Vert u\Vert_{\mathcal FL^p_{M,r}}\,,
\end{split}
\]
since the geometric series $\sum\limits_{\ell>0}2^{\ell(-rq+n/\mu_\ast)}$ is convergent for $r>\frac{n}{\mu_\ast q}$.
\end{proof}
\begin{remark}\label{rmk:i-iii_1}
{\rm As already noticed in the proof of the above Lemma \ref{le:2}, for $r>\frac{n}{\mu_\ast q}$ the continuity of the operator $D^\beta\phi(\varepsilon^{\frac1{M}}D)$ from $\mathcal FL^p_{r,M}$ to $\mathcal FL^1$ readily follows from the continuity of the same operator in $\mathcal FL^p_{r,M}$ and the validity of the continuous imbedding of $\mathcal FL^p_{r,M}$ into $\mathcal FL^1$; combining the above with the inequality \eqref{ineq:1} also gives the following continuity estimate
\[
\Vert
D^{\beta}\phi(\varepsilon^{\frac1{M}}D)u\Vert_{\mathcal FL^1}\le C_{\beta}\varepsilon^{-\langle\beta,\frac1{M}\rangle}\Vert
u\Vert_{\mathcal FL^p_{r,M}}\,,\quad\forall\,u\in\mathcal FL^p_{r,M}\,.
\]
Notice however that inequalities $\eqref{ineq:3}_{1,2}$ provide an improvement of such a continuity estimate above, as they give a sharper control of the norm of $D^{\beta}\phi(\varepsilon^{\frac1{M}}D)$, with respect to $\varepsilon$, as a linear bounded operator in $\mathcal L(\mathcal FL^p_{r,M};\mathcal FL^1)$.}
\end{remark}
\begin{remark}\label{rmk:i-iii_2}
{\rm In the case of $r>\frac{n}{\mu_\ast q}$, applying statement (ii) of Lemma \ref{le:2} with $0\le t<r-\frac{n}{\mu_\ast q}$ and taking account of $\mathcal FL^p_{r,M}\subset\mathcal FL^1$, with continuous imbedding, yields that
\begin{equation}
\Vert u-\phi(\varepsilon^{\frac1{M}}D)u\Vert_{\mathcal FL^{1}}\le
C_t\varepsilon^{t}\Vert u\Vert_{\mathcal FL^p_{r,M}}\,,\quad\forall\,u\in\mathcal FL^p_{r,M},\label{ineq:4}
\end{equation}
holds true with some positive constant $C_t$, independent of $\varepsilon$. Notice, however, that the endpoint case $t=r-\frac{n}{\mu_\ast q}$ (corresponding to statement (iii) of Lemma \ref{le:2}) cannot be reached by treating it along the same arguments used to prove statement (ii) above; indeed, in general, $\mathcal FL^p_{\frac{n}{\mu_\ast q},M}$ is not imbedded in $\mathcal FL^1$ (that is $\langle\cdot\rangle^{-\frac{n}{\mu_\ast q}}\notin L^q$).
}
\end{remark}

\noindent
Let $a(x,\xi)$ belong to $\mathcal FL^p_{r,M}S^{m}_M$ and take $\delta\in]0,1]$; we
define $a^{\#}(x,\xi)$ by the following
\begin{equation}\label{a_sharp}
a^{\#}(x,\xi):=\sum\limits_{h=-1}^{\infty}\phi(2^{-\frac{h\delta}{M}}D_x)a(x,\xi)\varphi_h(\xi)\,.
\end{equation}
We also set
\begin{equation}\label{a_natural}
a^{\natural}(x,\xi):=a(x,\xi)-a^{\#}(x,\xi)\,.
\end{equation}
As a consequence of Lemma \ref{le:2}, one can prove the following result; this will play a fundamental role in the analysis made in Sect. \ref{thm1.4_sct}.
\begin{proposition}\label{pro:2}
For $r>\frac{n}{\mu_\ast q}$ and $m\in\mathbb R$, let $a(x,\xi)\in\mathcal FL^p_{r,M}S^{m}_M$ and take an arbitrary
$\delta\in]0,1]$. Then
$$
a^{\#}(x,\xi)\in S^m_{M,\delta,\kappa},
$$
where $\kappa=r-\frac{n}{\mu_\ast q}$; moreover $a^\natural(x,\xi)\in \mathcal FL^p_{r,M}S^{m-\delta\left(r-\frac{n}{\mu_\ast q}\right)}_{M,\delta}$.
\end{proposition}
\begin{proof}
For arbitrary $\alpha,\beta\in\mathbb{Z}^n_+$, from Leibniz's rule we get
\begin{equation}\label{a_sharp_1}
\begin{split}
\Vert
D^{\beta}_xD^{\alpha}_{\xi} &a^{\#}(\cdot,\xi)\Vert_{\mathcal FL^1}\le\sum\limits_{\nu\le\alpha}\binom{\alpha}{\nu}\sum\limits_{h=-1}^{+\infty}\Vert
D^{\beta}_x\phi(2^{-\frac{\delta h}{M}}D)D^{\alpha-\nu}_{\xi}a(\cdot,\xi)\Vert_{\mathcal FL^1}\vert
D^{\nu}_{\xi}\varphi_h(\xi)\vert\\
&=\sum\limits_{\nu\le\alpha}\binom{\alpha}{\nu}\sum\limits_{h=\tilde h_0}^{h_0+N_0}\Vert
D^{\beta}_x\phi(2^{-\frac{\delta h}{M}}D)D^{\alpha-\nu}_{\xi}a(\cdot,\xi)\Vert_{\mathcal FL^1}\vert
D^{\nu}_{\xi}\varphi_h(\xi)\vert\,,
\end{split}
\end{equation}
where, for every $\xi\in\mathbb R^n$, the integers $N_0>0$ (independent of $\xi$), $h_0=h_0(\xi)\ge -1$ and $\tilde h_0=\tilde h_0(\xi)$ are the same as considered in \eqref{overlap_cond}, \eqref{conv_finite}.
\newline
On the other hand, because $r>\frac{r}{\mu_\ast q}$, applying to $u=D^{\alpha-\nu}_{\xi}a(\cdot,\xi)$ the inequalities $\eqref{ineq:3}_{1,2}$  with $\varepsilon=2^{-h\delta}$ and using estimates \eqref{X_est} and \eqref{der_phi_h}, we get for $h\ge -1$ and $\xi\in\mathcal C^{M,K}_h$
\[
\begin{split}
\Vert
D^{\beta}_x &\phi(2^{-\frac{\delta h}{M}}D)D^{\alpha-\nu}_{\xi}a(\cdot,\xi)\Vert_{\mathcal FL^1}\le C_{r,\beta}2^{h\delta\left(\langle\beta,1/M\rangle-\kappa\right)_+}\Vert D^{\alpha-\nu}_\xi a(\cdot,\xi)\Vert_{\mathcal FL^p_{M,r}}\\
&\le C_{r,\alpha,\beta,\nu}\langle\xi\rangle_M^{m-\langle\alpha-\nu,1/M\rangle+\delta\left(\langle\beta,1/M\rangle-\kappa\right)_+}\,,\quad\mbox{if}\,\,\,\langle\beta,1/M\rangle\neq\kappa\,,
\end{split}
\]

\[
\begin{split}
\Vert
D^{\beta}_x &\phi(2^{-\frac{\delta h}{M}}D)D^{\alpha-\nu}_{\xi}a(\cdot,\xi)\Vert_{\mathcal FL^1}\le C_{r}\log^{1/q}(1+2^{h\delta})\Vert D^{\alpha-\nu}_\xi a(\cdot,\xi)\Vert_{\mathcal FL^p_{M,r}}\\
&\le C_{r,\alpha,\nu}\log^{1/q}(1+\langle\xi\rangle_M^{\delta})\langle\xi\rangle_M^{m-\langle\alpha-\nu,1/M\rangle}\\
&\le C_{r,\alpha,\nu}\log(1+\langle\xi\rangle_M^{\delta})\langle\xi\rangle_M^{m-\langle\alpha-\nu,1/M\rangle}\,,\quad\mbox{if}\,\,\,\langle\beta,1/M\rangle=\kappa\,,
\end{split}
\]
and
\[
\vert D^\nu\varphi_h(\xi)\vert\le C_\nu\langle\xi\rangle_M^{-\langle\nu,1/M\rangle}\,,
\]
with suitable positive constants $C_{r,\beta}$, $C_{r,\alpha,\beta,\nu}$, $C_r$, $C_{r,\alpha,\nu}$, $C_\nu$ independent of $h$.
Then summing the above inequalities over all $h$'s such that $\tilde h_0\le h\le h_0+N_0$, from \eqref{a_sharp_1} it follows that
\begin{equation}\label{a_sharp_2}
\begin{split}
&\Vert
D^{\beta}_xD^{\alpha}_{\xi} a^{\#}(\cdot,\xi)\Vert_{\mathcal FL^1}\le C_{\alpha,\beta}\langle\xi\rangle_M^{m-\langle\alpha,1/M\rangle+\delta\left(\langle\beta,1/M\rangle-\kappa\right)_+}\,,\quad\mbox{if}\,\,\,\langle\beta,1/M\rangle\neq\kappa\,,\\
&\Vert
D^{\beta}_xD^{\alpha}_{\xi} a^{\#}(\cdot,\xi)\Vert_{\mathcal FL^1}\le C_{\alpha,\beta}\langle\xi\rangle_M^{m-\langle\alpha,1/M\rangle}\log(1+\langle\xi\rangle_M^{\delta})\,,\quad\mbox{if}\,\,\,\langle\beta,1/M\rangle=\kappa\,,
\end{split}
\end{equation}
from which $a^{\#}(x,\xi)\in S^m_{M,\delta,\kappa}$ follows at once, recalling that $\mathcal FL^1$ is imbedded in the space of bounded continuous functions in $\mathbb R^n$.

\smallskip
As regards to symbol $a^\natural(x,\xi)$ defined in \eqref{a_natural}, applying inequalities \eqref{ineq:2} with $t=0$, together with estimates \eqref{X_est} and \eqref{der_phi_h}, and using similar arguments as above, for all integers $h\ge -1$ and $\xi\in\mathcal C^{M,K}_h$ we find
\[
\begin{split}
\Vert
D^{\alpha}_{\xi}a^{\natural}(\cdot,\xi)&\Vert_{\mathcal FL^p_{r,M}}\\
&\le\sum\limits_{\nu\le\alpha}\sum\limits_{h=\tilde h_0}^{h_0+N_0}\binom{\alpha}{\nu}\Vert (I-\phi(2^{-h\delta/M}D))(D^{\alpha-\nu}_\xi a(\cdot,\xi))\Vert_{\mathcal FL^p_{r,M}}\vert D^\nu\varphi_h(\xi)\vert\\
&\le\sum\limits_{\nu\le\alpha}\sum\limits_{h=\tilde h_0}^{h_0+N_0}C_{\alpha,\nu}\Vert D^{\alpha-\nu}_\xi a(\cdot,\xi)\Vert_{\mathcal FL^p_{r,M}}\vert D^\nu\varphi_h(\xi)\vert\\
&\le\sum\limits_{\nu\le\alpha}\sum\limits_{h=\tilde h_0}^{h_0+N_0}C^\prime_{\alpha,\nu}\langle\xi\rangle_M^{m-\langle\alpha-\nu,1/M\rangle}\langle\xi\rangle_M^{-\langle\nu,1/M\rangle}\le C_\alpha\langle\xi\rangle_M^{m-\langle\alpha,1/M\rangle}\,,
\end{split}
\]
with positive constants $C_{\alpha,\nu}$, $C^\prime_{\alpha,\nu}$,  $C_\alpha$ independent of $h$; similarly, replacing \eqref{X_est} with \eqref{FL_est} and \eqref{ineq:2} with $\eqref{ineq:3}_3$ (with $\varepsilon=2^{-h\delta}$) in the above estimates, we find
\[
\begin{split}
\Vert
D^{\alpha}_{\xi}a^{\natural}(\cdot,\xi)\Vert_{\mathcal FL^1}&\le\sum\limits_{\nu\le\alpha}\sum\limits_{h=\tilde h_0}^{h_0+N_0}\binom{\alpha}{\nu}\Vert (I-\phi(2^{-h\delta/M}D))(D^{\alpha-\nu}_\xi a(\cdot,\xi))\Vert_{\mathcal FL^1}\vert D^\nu\varphi_h(\xi)\vert\\
&\le\sum\limits_{\nu\le\alpha}\sum\limits_{h=\tilde h_0}^{h_0+N_0}C_{\alpha,\nu} 2^{-h\delta\left(r-\frac{n}{\mu_\ast q}\right)}\Vert D^{\alpha-\nu}_\xi a(\cdot,\xi)\Vert_{\mathcal FL^p_{r,M}}\vert D^\nu\varphi_h(\xi)\vert\\
&\le\sum\limits_{\nu\le\alpha}\sum\limits_{h=\tilde h_0}^{h_0+N_0}C^\prime_{\alpha,\nu}\langle\xi\rangle_M^{-\delta\left(r-\frac{n}{\mu_\ast q}\right)}\langle\xi\rangle_M^{m-\langle\alpha-\nu,1/M\rangle}\langle\xi\rangle_M^{-\langle\nu,1/M\rangle}\\
&\le C_\alpha\langle\xi\rangle_M^{m-\delta\left(r-\frac{n}{\mu_\ast q}\right)-\langle\alpha,1/M\rangle}\,,
\end{split}
\]
where the numerical constants involved above are independent of $h$. The above inequalities yields $a^\natural(x,\xi)\in\mathcal FL^p_{r,M}S^{m-\left(r-\frac{n}{\mu_\ast q}\right)}_{M,\delta}$, because of the arbitrariness of $h$ and that the $\mathcal C^{M,K}_h$'s cover $\mathbb R^n$.
\end{proof}
\section{Microlocal properties}\label{propation_sing_FL_sct}

In order to study the microlocal propagation of weighted Fourier Lebesgue singularities for PDEs, this section is devoted to define local/microlocal versions of $M-$Fourier Lebesgue spaces as well as $M-$homogeneous smooth symbols previously introduced in Sects. \ref{FL-sp_sct}, \ref{smooth_symb_sct}, and to collect some basic tools and a few results needed at this purpose.

\subsection{Local and microlocal function spaces}\label{mcl_FL_sp_sct}
While the main focus of this paper is on $M-$homogeneous Fourier Lebesgue spaces, in this section we define general scales of function spaces, where the microlocal propagation of singularities of pseudodifferential operators with $M-$homogeneous symbols, as defined in Sect. \ref{smooth_symb_sct}, will be then studied.

\smallskip
Let us consider a one-parameter family $\{\mathcal X_s\}_{s\in\mathbb R}$ of Banach spaces $\mathcal X_s$, $s\in\mathbb R$, such that
\begin{equation}\label{emb_Xs}
\mathcal S(\mathbb R^n)\subset\mathcal X_t\subset\mathcal X_s\subset\mathcal S^\prime(\mathbb R^n)\,,\quad\mbox{with continuous embedding}\,,
\end{equation}
for arbitrary $s<t$. Following Taylor \cite{MT-NLPDE}, we say that $\{\mathcal X_s\}_{s\in\mathbb R}$ is a {\em $M-$microlocal scale} provided that there exists a constant $\kappa_0>0$ such that for all $m\in\mathbb R$, $\delta\in[0,1[$, $\kappa>\kappa_0$ and $a(x,\xi)\in S^m_{M,\delta,\kappa}$ satisfying \eqref{localization} for some compact $\mathcal K\subset\mathbb R^n$, the pseudodifferential operator $a(x,D)$ extends to a linear bounded operator
\begin{equation}\label{cont_Xs}
a(x,D):\mathcal X_{s+m}\rightarrow \mathcal X_s\,,\quad\forall\,s\in\mathbb R\,.
\end{equation}
In view of Theorem \ref{FL_sym_cor}, it is clear that for every $p\in[1,+\infty]$ the $M-$homogeneous Fourier Lebesgue spaces $\{\mathcal FL^p_{s,M}\}_{s\in\mathbb R}$ constitute a $M-$microlocal scale, according to definition above; in this case the threshold $\kappa_0$ from \eqref{cont_Xs} is given by $\kappa_0=\left[n/\mu_\ast\right]+1$. Other examples of $M-$microlocal spaces are provided by {\em $M-$homogeneous Sobolev} and {\em H\"older} spaces studied in \cite{GM-09}\footnote{Actually for $M-$homogeneous Sobolev and H\"older spaces, the continuity property \eqref{cont_Xs} is extended to all pseudodifferential operators with symbol in $S^m_{M,\delta}$, without the need of the more restrictive decay conditions in Definition \ref{sm-sym_k_dfn} and of the {\em locality} condition \eqref{localization}, see \cite[Theorem 3.3 and Corollary 3.4]{GM-09}.}.

\smallskip
In order to allow the microlocal analysis performed in subsequent Sect. \ref{mcl_symb_sct}, the following local and microlocal counterparts of spaces $\mathcal X_s$, $s\in\mathbb R$, are given.
\begin{definition}\label{loc_mcl_FL}
Let $s\in\mathbb R$, $x_0\in\mathbb R^n$ and $\xi^0\in\mathbb R^n\setminus\{0\}$. We say that a distribution $u\in\mathcal S^\prime(\mathbb R^n)$ belongs to the local space $\mathcal X_{s,{\rm loc}}(x_0)$ if there exists a function $\phi\in C^\infty_0(\mathbb R^n)$, satisfying $\phi(x_0)\neq 0$, such that
\begin{equation}\label{Xs_loc}
\phi u\in\mathcal X_s\,.
\end{equation}
We say that $u\in\mathcal S^\prime(\mathbb R^n)$ belongs to the microlocal space $\mathcal X_{s, {\rm mcl}}(x_0,\xi^0)$ provided that there exist a function $\phi\in C^\infty_0(\mathbb R^n)$, satisfying $\phi(x_0)\neq 0$, and a symbol $\psi(\xi)\in S^0_M$, satisfying $\psi(\xi)\equiv 1$ on $\Gamma_M\cap\{\vert\xi\vert_M>\varepsilon_0\}$ for suitable $M-$conic neighborhood $\Gamma_M\subset\mathbb R^n\setminus\{0\}$ of $\xi^0$ and $0<\varepsilon_0<\vert\xi^0\vert_M$, such that
\begin{equation}\label{Xs_mcl}
\psi(D)(\phi u)\in\mathcal X_s\,.
\end{equation}
Under the same assumptions as above, we also write
\begin{equation}\label{sing_WF_Xs}
x_0\notin\mathcal X_s-{\rm singsupp}\,(u)\quad\mbox{and}\quad(x_0,\xi^0)\notin WF_{\mathcal X_s}(u)
\end{equation}
respectively.
\end{definition}
In the case $\mathcal X^s\equiv\mathcal FL^p_{s,M}$, it is clear that Definition \ref{loc_mcl_FL} reduces to Definition \ref{pfl_def_WF}.
\newline
It can be easily proved that $\mathcal X_s-{\rm singsupp}\,(u)$ is a closed subset of $\mathbb R^n$ and is called the {\em $\mathcal X_s-$singular support} of the distribution $u$, whereas
$WF_{\mathcal X_s}(u)$ is a closed subset of $T^\circ\mathbb R^n$, $M-conic$ with respect to the $\xi$ variable, and is called the {\em $\mathcal X_s-$wave front set} of $u$. The previous notions are natural generalizations of the classical notions of singular support and wave front set of a distribution introduced by H\"ormander \cite{HOR-book}, see also \cite{HOR-hyp}.
\newline
Let $\pi_1$ be the canonical projection of $T^{\circ}\mathbb{R}^n$ onto $\mathbb{R}^n$, that is $\pi_1(x,\xi)=x$. Arguing as in the classical case, one can prove the following.
\begin{proposition}\label{proiezione}
if $u\in\mathcal X_{s,{\rm mcl}}(x_0,\xi^0)$, with $(x_0,\xi^0)\in T^{\circ}\mathbb R^n$, then,  for any $\varphi\in C^\infty_0(\mathbb R^n)$, such that $\varphi(x_0)\neq 0$, $\varphi u\in {\rm mcl}\mathcal X_{s,{\rm mcl}}(x_0,\xi^0)$. Moreover, we have:
$$
\mathcal X_s-{\rm singsupp}(u)=\pi_1(WF_{\mathcal X_s}(u))\,.
$$
\end{proposition}
\subsection{Microlocal symbol classes}\label{mcl_symb_sct}

We introduce now microlocal counterparts of the smooth symbol classes studied in Sect. \ref{smooth_symb_sct}, cf. Definitions \ref{sm-sym_dfn}, \ref{sm-sym_k_dfn}.
\begin{definition}\label{microlocalsymbol}
let $U$ be an open subset of $\mathbb{R}^n$ and $\Gamma_M\subset\mathbb{R}^n\setminus\{0\}$ an open $M-$conic set. For $m\in\mathbb{R}$, $\delta\in[0,1]$ and $\kappa>0$; we say that $a\in S^{\prime}(\mathbb{R}^{2n})$ belongs to $S^m_{M,\delta}$ (resp. to $S_{M,\delta,\kappa}$) microlocally on $U\times\Gamma_{M}$  if $a_{|\,\,U\times\Gamma_M}\in C^{\infty}(U\times\Gamma_M)$ and for all $\alpha,\beta\in\mathbb{Z}^n_+$ there exists $C_{\alpha,\beta}>0$ such that \eqref{sm-sym_est} (resp. \eqref{sm-sym_k_est_1}, \eqref{sm-sym_k_est_2}) holds true for all $(x,\xi)\in U\times\Gamma_M$. We will write in this case $a\in mcl S^m_{M,\delta}(U\times\Gamma_M)$ (resp. $a\in mcl S^m_{M,\delta,\kappa}(U\times\Gamma_M)$).
For $(x_0,\xi^0)\in T^{\circ}\mathbb{R}^n$, we set
\begin{equation}\label{microlocalset}
mcl S^m_{M,\delta\,(,\kappa)}(x_0,\xi^0):=\bigcup_{U,\,\Gamma_M}mcl S^m_{M,\delta\,(,\kappa)}(U\times\Gamma_M)\,,
\end{equation}
where the union in the right-hand side is taken over all of the open neighborhoods $U\subset\mathbb{R}^n$ of $x_0$ and the open $M-$conic neighborhoods $\Gamma_M\subset\mathbb{R}^n\setminus\{0\}$ of $\xi^0$.
\end{definition}

With the above stated notation, we say that $a\in \mathcal S'(\mathbb R^n)$ is {\it microlocally regularizing} on $U\times \Gamma_M$ if $a_{|\,\,U\times\Gamma_M}\in C^{\infty}(U\times\Gamma_M)$ and for every $\mu>0$ and all $\alpha,\beta\in\mathbb{Z}^n_+$ a positive constant $C_{\mu,\alpha,\beta}>0$ is found in such a way that:
\begin{equation}\label{microregineq}
|\partial^{\alpha}_{\xi}\partial^{\beta}_x a(x,\xi)|\le C_{\mu,\alpha,\beta}(1+|\xi|)^{-\mu}\,,\quad\forall\,(x,\xi)\in U\times\Gamma_M\,.
\end{equation}
Let us denote by $mcl S^{-\infty}(U\times\Gamma_M)$ the set of all {\em microlocally regularizing symbols} on $U\times\Gamma_M$. For $(x_0,\xi^0)\in T^{\circ}\mathbb{R}^n$, we  set:
\begin{equation}\label{microregpuntuale}
mcl S^{-\infty}(x_0,\xi^0):=\bigcup_{U,\,\Gamma_M}mcl S^{-\infty}(U\times\Gamma_M)\,;
\end{equation}
it is easily seen that $mcl S^{-\infty}(U\times\Gamma_M)=\bigcap_{m>0}mcl S^{-m}_{M,\delta}(U\times\Gamma_M)$ for all $\delta\in[0,1]$ and $M\in\mathbb{N}^n$, and a similar identity holds for $mcl S^{-\infty}(x_0,\xi^0)$.
\newline
It is immediate to check that symbols in $mcl S^m_{M,\delta}(U\times\Gamma_M)$, $mcl S^m_{M,\delta}(x_0,\xi^0)$ behave according to the same rules of ``global'' symbols, collected in Proposition \ref{symb_calc_prop}. Moreover $S^m_{M,\delta\,(,\kappa)}\subset mcl S^m_{M,\delta\,(,\kappa)}(U\times\Gamma_M)\subset mcl S^m_{M,\delta\,(,\kappa)}(x_0,\xi^0)$ hold true, whenever $(x_0,\xi^0)\in T^{\circ}\mathbb{R}^n$, $U$ is an open neighborhood of $x_0$ and $\Gamma_M$ is an open $M-$conic neighborhood of $\xi^0$.
A slight modification of the arguments used to prove Proposition \ref{M_elliptic_prop}, see also \cite[Proposition 4.4]{GM-09}, leads to the following microlocal counterpart.
\begin{proposition}{\bf (Microlocal parametrix).}\label{microparametrix}
Assume that $0\le\delta<\mu_\ast/\mu^\ast$ and $\kappa>0$ and let $a(x,\xi)\in S^m_{M,\delta,\kappa}$ be microlocally $M-$elliptic at $(x_0,\xi^0)\in T^{\circ}\mathbb{R}^n$. Then there exist symbols $b(x,\xi),c(x,\xi)\in S^{-m}_{M,\delta,\kappa}$ such that
\begin{equation}\label{leftrightinverse}
c(x,D)a(x,D)=I+l(x,D)\quad and\quad a(x,D)b(x,D)=I+r(x,D)\,,
\end{equation}
and $l(x,\xi),r(x,\xi)\in{\rm mcl}S^{-\infty}(x_0,\xi^0)$.
\end{proposition}
The notion of microlocal $M-$ellipticity, as well as the characteristic set, see Definition \ref{Melliptic}, can be readily extended to non-smooth $M-$homogeneous symbols (as, in principle, it only needs that the symbol $a(x,\xi)$ be a continuous function, at least for sufficiently large $\xi$); in particular, microlocally $M-$elliptic symbols in $\mathcal FL^p_{r,M}S^m_{M}$, with sufficiently large $r>0$, must be considered later on. For a symbol $a(x,\xi)\in\mathcal FL^p_{r,M}S^m_{M}$, with $r>\frac{n}{\mu_\ast q}$, $p\in[1,+\infty]$ and $\frac1{p}+\frac1{q}=1$, for every $0<\delta\le 1$ let the symbol $a^\#(x,\xi)$ and $a^\natural(x,\xi)$ be defined as in \eqref{a_sharp}, \eqref{a_natural}.
\newline
The following result can be proved along the same lines of the proof of \cite[Proposition 7.3]{GM-09}.
\begin{proposition}\label{SHARPEL}
 If $a(x,\xi)\in\mathcal FL^p_{r,M}S^{m}_{M}$, $m\in\mathbb R$, is microlocally $M-$elliptic at $(x_0,\xi^0)\in T^\circ\mathbb R^n$, then $a^{\#}(x,\xi)\in S^{m}_{M,\delta,\kappa}$ (with $\kappa$ as in the statement of Proposition \ref{pro:2}) is also microlocally $M-$elliptic at $(x_0,\xi^0)$ for any $0<\delta\le 1$.
\end{proposition}
\subsection{Microlocal continuity and regularity results}\label{Xs_mcl_res_sct}
Let $\{\mathcal X_s\}_{s\in\mathbb R}$ be a $M-$mic\-rolocal scale as defined in Sect. \ref{mcl_FL_sp_sct}. The following microlocal counterpart of the boundedness property \eqref{cont_Xs} and microlocal $\mathcal X_s-$regularity follow along the same lines of the proof of \cite[Theorem 5.4 and Theorem 6.1]{GM-09}.
\begin{proposition}\label{microsobolevaction}
for $0\le\delta<\mu_\ast/\mu^\ast$, $\kappa>\kappa_0$ (being $\kappa_0>0$ the same threshold involved in \eqref{cont_Xs}), $m\in\mathbb R$ and $(x_0,\xi^0)\in T^{\circ}\mathbb{R}^n$, assume that $a(x,\xi)\in S^{\infty}_{M,\delta}\cap {\rm mcl} S^m_{M,\delta,\kappa}(x_0,\xi^0)$. Then for all $s\in\mathbb{R}$
\begin{equation}\label{microcont}
u\in\mathcal X_{s+m, {\rm mcl}}(x_0,\xi^0)\quad\Rightarrow\quad a(x,D)u\in \mathcal X_{s, {\rm mcl}}(x_0,\xi^0)\,.
\end{equation}
\end{proposition}

\begin{proposition}\label{mcl_Xs-reg}
for $0\le\delta<\mu_\ast/\mu^\ast$, $\kappa>\kappa_0$, $m\in\mathbb R$, let $a(x,\xi)\in S^m_{M,\delta,\kappa}$ be microlocally $M-$elliptic at $(x_0,\xi^0)\in T^{\circ}\mathbb{R}^n$. Then for all $s\in\mathbb R$
\begin{equation}
u\in\mathcal S^\prime(\mathbb R^n)\,\,\,\,\mbox{and}\,\,\,\,a(x,D)u\in\mathcal X_{s, {\rm mcl}}(x_0,\xi^0)\quad\Rightarrow\quad u\in\mathcal X_{s+m, {\rm mcl}}(x_0,\xi^0)\,.
\end{equation}
\end{proposition}
Resorting on the notions of $M-$homogeneous wave front set of a distribution and characteristic set of a symbol, the results of the above propositions can be also restated in the following
\begin{corollary}\label{Xs_cor}
for $0\le\delta<\mu_\ast/\mu^\ast$, $\kappa>\kappa_0$, $m\in\mathbb R$, $a(x,\xi)\in S^m_{M,\delta,\kappa}$ and $u\in\mathcal S^\prime(\mathbb R^n)$, the following inclusions
\begin{equation*}
WF_{\mathcal X_{s}}(a(x,D)u)\subset WF_{\mathcal X_{s+m}}(u)\subset WF_{\mathcal X_s}(a(x,D)u)\cup{\rm Char}(a)
\end{equation*}
hold true for every $s\in\mathbb R$.
\end{corollary}
As particular case of Corollary \ref{Xs_cor} we obtain the result in Theorem \ref{FL_cor}
\subsection{Proof of Theorem \ref{cor_APPLICATION}}\label{thm1.4_sct}
This section is devoted to the proof of Theorem \ref{cor_APPLICATION} concerning the microlocal propagation of Fourier Lebesgue singularities of the linear PDE \eqref{APPL:1}. As it will be seen below, the statement of Theorem \ref{cor_APPLICATION} can be deduced as an immediate consequence of a more gerenal result concerning a suitable class of pseudodifferential operators.

\smallskip
Since the coefficients $c_\alpha$ in the equation \eqref{APPL:1} belong to $\mathcal FL^p_{r,M,{\rm loc}}(x_0)$, it follows that the {\em localized symbol} $a_\phi(x,\xi):=\phi(x)a(x,\xi)$ belongs to the symbol class $\mathcal FL^p_{r,M}S^1_{M}:=\mathcal FL^p_{r,M}S^1_{M,0}$, for some function $\phi\in C^\infty_0(\mathbb R^n)$ supported on a sufficiently small compact neighborhood of $x_0$ and satisfying $\phi(x_0)\neq 0$ (see Definition \ref{def_bdd_smooth_symb}); moreover, by exploiting the $M-$homogeneity in $\xi$ of the $M-$principal part of $a(x,\xi)$, the localized symbol $a_\phi(x,\xi)$ amounts to be microlocally $M-$elliptic at $(x_0,\xi^0)$ according to Definition \ref{Melliptic}.
\newline
It is also clear that, by a locality argument, for any $u\in\mathcal S^\prime(\mathbb R^n)$
\begin{equation}\label{locality}
a_\phi(x,D)u=a_\phi(x,D)(\psi u)\,,
\end{equation}
where $\psi\in C^\infty_0(\mathbb R^n)$ is some cut-off function, depending only on $\phi$, that satisfies
\begin{equation}\label{psi}
0\le\psi\le 1\,,\qquad\mbox{and}\qquad\psi\equiv 1\,,\quad\mbox{on}\,\,\,{\rm supp}\,\phi\,.
\end{equation}
It tends out that only the identity \eqref{locality} will be really exploited in the subsequent analysis; thus the symbol of a differential operator of the type considered in \eqref{APPL:1}, with {\em point-wise local} $M-$homogeneous Fourier Lebesgue coefficients, can be replaced with any symbol $a(x,\xi)$ of positive order $m$ and local Fourier Lebesgue coefficients at some point $x_0$, namely 
\begin{equation}\label{FL_loc_symb}
a_\phi(x,\xi)\in\mathcal FL^p_{r,M}S^m_M\,,\quad\mbox{for some}\,\,\,\phi\in C^\infty_0(\mathbb R^n)\,\,\,\mbox{satisfying}\,\,\,\phi(x_0)\neq 0\,,
\end{equation}
so that the related pseudodifferential operator $a(x,D)$ be {\em properly supported}: while locality does not hold for a general symbol in $\mathcal FL^p_{r,M}S^m_M$ (unless it is a polynomial in $\xi$ variable ), identity \eqref{locality} is still true whenever $a(x,D)$ is properly supported (see \cite{AG-book} for the definition and properties of a properly supported operator). For shortness here below we write $a(x,\xi)\in\mathcal FL^p_{r,M}S^m_M(x_0)$ to mean that condition \eqref{FL_loc_symb} is satisfied by $a(x,\xi)$.
\begin{theorem}\label{APPLICATION}
 For $(x_0,\xi^0)\in T^\circ\mathbb R^n$, $p\in[1,+\infty]$ and $r>\frac{n}{\mu_\ast q}+\left[\frac{n}{\mu_\ast}\right]+1$, where $q$ is the conjugate exponent of $p$, let  $a(x,\xi)\in\mathcal FL^p_{r,M}S^m_M(x_0)$ be, microlocally $M-$elliptic at $(x_0,\xi^0)$ with positive order $m$, such that $a(x,D)$ is properly supported. For all $0<\delta<\mu_\ast/\mu^{\ast}$ and $m+(\delta-1)\left(r-\frac{n}{\mu_\ast q}\right)<s\le r+m$ we have
 \begin{equation}\label{mcl_ell_lim_smoooth}
 \begin{split}
 & u\in\mathcal FL^p_{s-\delta\left(r-\frac{n}{\mu_\ast q}\right), M, {\rm loc}}(x_0)\,,\\
 & \mbox{and}\quad a(x,D)u\in\mathcal FL^p_{s-m, M, {\rm mcl}}(x_0,\xi^0)
 \end{split}\quad\Rightarrow\quad u\in\mathcal FL^p_{s, M, {\rm mcl}}(x_0,\xi^0)\,.
 \end{equation}
\end{theorem}
\begin{proof} Let us set $f:=a(x,D)u$ for $u\in\mathcal FL^p_{s-\delta\left(r-\frac{n}{\mu_\ast q}\right), M, {\rm loc}}(x_0)$. Since $a(x,D)$ is properly supported, suitable smooth functions $\phi\in C^\infty_0(\mathbb R^n)$ and $\psi$ satisfying \eqref{locality} and \eqref{psi} can be found, supported on such a sufficiently small neighborhood of $x_0$ that $\psi u\in FL^p_{s-\delta\left(r-\frac{n}{\mu_\ast q}\right), M}$ and $a_\phi(x,\xi)\in\mathcal FL^p_{r,M}S^m_M$, cf. Definition \ref{loc_mcl_FL} and \eqref{FL_loc_symb}. Following the decomposition method illustrated in \eqref{a_sharp}, \eqref{a_natural}, for $0<\delta<\mu_\ast/\mu^{\ast}$ let $a^\#_\phi(x,\xi)\in S^m_{M,\delta,\kappa}$ and $a^\natural_\phi(x,\xi)\in\mathcal FL^p_{r,M}S^{m-\delta\left(r-\frac{n}{\mu_\ast q}\right)}_{M,\delta}$ be defined as in \eqref{a_sharp}, \eqref{a_natural}, with $a_\phi$ instead of $a$ and where $\kappa=r-\frac{n}{\mu_\ast q}$, hence $u$ satisfies the equation
$$
a^\sharp_\phi(x,D)(\psi u)=\phi f-a^\natural_\phi(x,D)(\psi u)\,.
$$
Because $a^{\natural}_\phi(x,\xi)\in \mathcal FL^p_{r,M}S^{m-\delta\left(r-\frac{n}{\mu_\ast q}\right)}_{M,\delta}$, $\psi u\in \mathcal FL^p_{s-\delta\left(r-\frac{n}{\mu_\ast q}\right),M}$, whereas $f$ (so also $\phi f$) belongs to $\mathcal FL^p_{s-m,M,{\rm mcl}}(x_0,\xi_0)$ (cf. Proposition \ref{proiezione}), for the range of $s$ belonging as prescribed in the statement of Theorem \ref{APPLICATION} (notice in particular that from $0<\delta<\mu_\ast/\mu^\ast\le 1$ even the endpoint $s=r+m$ is allowed), one can apply Theorem \ref{FL_thm} to find
\[
a^\#_\phi(x,D)(\psi u)\in\mathcal FL^p_{s-m,M,{\rm mcl}}(x_0,\xi^0)\,;
\]
hence, because $\kappa>\left[n/\mu_\ast\right]+1$, applying Theorem \ref{FL_cor} to $a_\phi^{\#}(x,\xi)$ yields that $\psi u$, hence $u$, belongs to $\mathcal FL^p_{s,M,{\rm mcl}}(x_0,\xi^0)$, which ends the proof.
\end{proof}
It is worth noticing that the result of Theorem \ref{APPLICATION} can be restated in terms of characteristic set of a symbol and Fourier Lebesgue of a distribution as in the next result.
\begin{proposition}\label{APPLICATION1}
Let $r$, $m$, $p$, $s$ and $\delta$ satisfy the same conditions as in Theorem \ref{APPLICATION}. Then for $a(x,\xi)\in\mathcal FL^p_{r,M}S^m_M$ and $u\in\mathcal FL^p_{s-\delta\left(r-\frac{n}{\mu_\ast q}\right), M}$ we have
$$
\begin{array}{ll}
WF_{\mathcal FL^p_{s,M}}(u)\subset WF_{\mathcal FL^p_{s-m,M}}(a(x,D)u)\cup{\rm Char}(a)\,.
\end{array}
$$
\end{proposition}
The statement of Theorem \ref{APPLICATION}, as well as Proposition \ref{APPLICATION1}, applies in particular to the linear PDE \eqref{APPL:1} considered at the beginning of this section, thus Theorem \ref{cor_APPLICATION} is proved.

\medskip

{\small
Gianluca Garello\\
Dipartimento di Matematica\\
Universit\`a di Torino\\
Via Carlo Alberto 10, I-10123 Torino, Italy\\
gianluca.garello@unito.it}
\medskip

\noindent {\small
Alessandro Morando\\
DICATAM-Sezione di Matematica\\
Universit\`a di Brescia\\
Via Valotti 9, I-25133 Brescia, Italy\\
alessandro.morando@unibs.it}

\end{document}